\documentclass[11pt]{article}
\usepackage{amsmath}
\usepackage{amsfonts}
\usepackage{amssymb}
\usepackage{geometry}
\usepackage[doublespacing]{setspace}
\usepackage[longnamesfirst]{natbib}
\usepackage{hyperref}
\usepackage{verbatim}

\DeclareFontFamily{U}{mathx}{\hyphenchar\font45}
\DeclareFontShape{U}{mathx}{m}{n}{
      <5> <6> <7> <8> <9> <10>
      <10.95> <12> <14.4> <17.28> <20.74> <24.88>
      mathx10
      }{}
\DeclareSymbolFont{mathx}{U}{mathx}{m}{n}
\DeclareFontSubstitution{U}{mathx}{m}{n}
\DeclareMathAccent{\widecheck}{0}{mathx}{"71}
\newtheorem {theorem}{Theorem}[section]

\newtheorem {condition}{Condition}[section]
\newtheorem {lemma}{Lemma}[section]

\newtheorem {assumption}{Assumption}

\newtheorem {remark}{Remark}[section]
\newtheorem {corollary}{Corollary}[section]

\newenvironment {proof}[1][Proof]{\noindent \textbf {#1.} }{\ \rule {0.5em}{0.5em}}

\begin{document}

\author{Xiaohong Chen\thanks{%
Corresponding Author. Cowles Foundation for Research in Economics, Yale
University, Box 208281, New Haven, CT 06520, USA. Tel: +1 203 432 5852; fax:
+1 203 432 6167. E-mail address: \texttt{xiaohong.chen@yale.edu}} \ and
Timothy M. Christensen\thanks{%
Department of Economics, New York University, 19 West 4th Street, New York,
NY 10012, USA. E-mail address: \texttt{timothy.christensen@nyu.edu}}}
\title{Optimal Uniform Convergence Rates and Asymptotic Normality for Series
Estimators Under Weak Dependence and Weak Conditions\thanks{%
We thank the guest coeditor, two anonymous referees, Bruce Hansen, Jianhua Huang, Stefan
Schneeberger, and conference participants of SETA2013 in Seoul and AMES2013
in Singapore for useful comments. This paper is an expanded version of
Sections 4 and 5 of our Cowles Foundation Discussion Paper No. 1923 %
\citep{ChenChristensenNPIVold}, the remaining parts of which are currently
undergoing a major revision. Support from the Cowles Foundation is
gratefully acknowledged. Any errors are the responsibility of the authors.}}
\date{First version January 2012; Revised August 2013, August 2014}
\maketitle

\begin{abstract}
\noindent We show that spline and wavelet series regression estimators for
weakly dependent regressors attain the optimal uniform (i.e. sup-norm)
convergence rate $(n/\log n)^{-p/(2p+d)}$ of \cite{Stone1982}, where $d$ is
the number of regressors and $p$ is the smoothness of the regression
function. The optimal rate is achieved even for heavy-tailed martingale
difference errors with finite $(2+(d/p))$th absolute moment for $d/p<2$. We
also establish the asymptotic normality of t statistics for possibly
nonlinear, irregular functionals of the conditional mean function under weak
conditions. The results are proved by deriving a new exponential inequality
for sums of weakly dependent random matrices, which is of independent
interest. \bigskip

\noindent \textbf{JEL Classification:} C12, C14, C32

\medskip \noindent \textbf{Keywords:} Nonparametric series regression;
Optimal uniform convergence rates; Weak dependence; Random matrices;
Splines; Wavelets; (Nonlinear) Irregular Functionals; Sieve t statistics
\end{abstract}

\newpage

\section{Introduction}

We study the nonparametric regression model
\begin{equation}
\begin{array}{rcl}
Y_{i} & = & h_{0}(X_{i})+\epsilon _{i} \\
E[\epsilon _{i}|X_{i}] & = & 0%
\end{array}
\label{npreg}
\end{equation}%
where $Y_{i}\in {\mathbb{R}}$ is a scalar response variable, $X_{i}\in {%
\mathcal{X}}\subseteq {\mathbb{R}}^{d}$ is a $d$-dimensional regressor
(predictor variable), and the conditional mean function $%
h_{0}(x)=E[Y_{i}|X_{i}=x]$ belongs to a H\"{o}lder space of smoothness $p>0$%
. We are interested in series least squares (LS) estimation\footnote{%
Other terms for series LS appearing in the literature include series
regression and linear sieve regression, but we use series LS hereafter. The
series LS estimator falls into the general class of nonparametric sieve
M-estimators.} of $h_{0}$ under sup-norm loss and inference on possibly
nonlinear functionals of $h_{0}$ allowing for weakly dependent regressors
and heavy-tailed errors $\epsilon _{i}$.\footnote{%
The error $\epsilon _{i}$ is heavy-tailed in the sense that $E[\left\vert
\epsilon _{i}\right\vert ^{2+\delta }]=\infty $ for $\delta >d/p$ is
allowed; say $E[\left\vert \epsilon _{i}\right\vert ^{4}]=\infty $ is
allowed.} 

For i.i.d. data, \cite{Stone1982} shows that $(n/\log n)^{-p/(2p+d)}$ is the minimax lower
bound in sup-norm risk for estimation of $h_{0}$ over a H\"{o}lder ball of
smoothness $p>0$. For strictly stationary beta-mixing regressors, we show
that spline and wavelet series LS estimators $\widehat{h}$ of $h_{0}$ attain
the optimal uniform rate of \cite{Stone1982} under a mild unconditional
moment condition $E[\left\vert \epsilon _{i}\right\vert ^{2+(d/p)}]<\infty $
imposed on the martingale difference errors.

More generally, we assume the error process $\{\epsilon _{i}\}_{i=-\infty
}^{\infty }$ is a martingale difference sequence but impose no explicit weak
dependence condition on the regressor process $\{X_{i}\}_{i=-\infty
}^{\infty }$. Rather, weak dependence of the regressor process is formulated
in terms of convergence of a certain random matrix. We verify this condition
for absolutely regular (beta-mixing) sequences by deriving a new exponential
inequality for sums of weakly dependent random matrices. This new inequality
then leads to a sharp upper bound on the sup-norm variance term of series LS
estimators with an arbitrary basis. When combined with a general upper bound
on the sup-norm bias term of series LS estimators, the sharp sup-norm
variance bound immediately leads to a general upper bound on the sup-norm
convergence rate of series LS estimators with an arbitrary basis and weakly
dependent data.

In our sup-norm bias and variance decomposition of series LS estimators, the
bound on the sup-norm bias term depends on the sup norm of the empirical $%
L^{2}$ projection onto the linear sieve space. The sup norm of the empirical
$L^{2}$ projection varies with the choice of the (linear sieve) basis. For
spline regression with i.i.d. data, \cite{Huang2003} shows that the sup norm
of the empirical $L^{2}$ projection onto splines is bounded with probability
approaching one (wpa1). Using our new exponential inequality for sums of
weakly dependent random matrices, his bound is easily extended to spline
regression with weakly dependent regressors. In addition, we show in Theorem \ref{c-pstable} that, for
either i.i.d. or weakly dependent regressors, the sup norm of the empirical $%
L^{2}$ projection onto compactly supported wavelet bases is also bounded
wpa1 (this property is called sup-norm stability of empirical $L^2$ projection). These tight bounds lead to sharp sup-norm bias control for spline and
wavelet series LS estimators. They allow us to show that spline and wavelet
series LS estimators achieve the optimal sup-norm convergence rate with weakly dependent data and heavy-tailed errors (e.g., $E[\epsilon_i^4] = \infty$ is allowed).

Sup-norm (uniform) convergence rates of series LS estimators have
been studied previously by \cite{Newey1997} and \cite{deJong2002} for i.i.d. data and \cite{LeeRobinson} for spatially
dependent data. But the uniform convergence rates obtained in these papers
are slower than the optimal rate of \cite{Stone1982}.\footnote{%
See \cite{hansen-ET}, \cite{Kristensen2009}, \cite{masry}, \cite{Tsybakov2009}, \cite%
{CattaneoFarrell} and the references therein for attainability of the
optimal uniform convergence rates with kernel, local linear regression and
partitioning estimators.} In an important paper on series LS regression with i.i.d. data, \cite{BCK2013} establish the attainability of the optimal sup-norm rates of series LS estimators using spline, local polynomial
partition, wavelet and other series possessing the property of sup-norm stability of the $L^2$ projection (or bounded Lebesgue constant)  under the
conditional moment condition $\sup_{x}E[\left\vert \epsilon _{i}\right\vert
^{2+\delta }|X_{i}=x]<\infty $ for some $\delta >d/p$. Our Theorem \ref{t-wavd} on the sup-norm stability of $L^2$ projection of the wavelet basis is used by \cite{BCK2013} to show that the wavelet series LS estimator achieves the optimal sup-norm rate under their conditional moment requirement. Our paper contributes to the literature by showing that spline and wavelet series
LS estimators attain the optimal sup-norm rate with
strictly stationary beta-mixing regressors under the weaker unconditional
moment requirement $E[\left\vert \epsilon _{i}\right\vert ^{2+(d/p)}]<\infty
$.

As another application of our new exponential inequality, under very weak
conditions we obtain sharp $L^{2}$ convergence rates for series LS
estimators with weakly dependent regressors.
For example, under the minimal bounded conditional second moment restriction
($\sup_{x}E[\left\vert \epsilon _{i}\right\vert ^{2}|X_{i}=x]<\infty $), our
$L^{2}$-norm rates for trigonometric polynomial, spline or wavelet series LS
estimators attain \cite{Stone1982}'s optimal $L^{2}$-norm rate of $%
n^{-p/(2p+d)}$ with strictly stationary, exponentially beta-mixing
(respectively algebraically beta-mixing at rate $\gamma $) regressors with $%
p>0$ (resp. $p>d/(2\gamma )$), while the power series LS estimator attains
the same optimal rate with exponentially (resp. algebraically) beta-mixing
regressors for $p>d/2$ (resp. $p>d(2+\gamma )/(2\gamma )$). It is
interesting to note that for a smooth conditional mean function, we obtain
the optimal $L^{2}$ convergence rates for these commonly used series LS
estimators with weakly dependent regressors without requiring the existence
of higher-than-second unconditional moments of the error terms. Previously, \cite{Newey1997} derived the optimal $L^2$ convergence rates of series LS estimator under i.i.d. data and the restriction of $K^2 /n = o(1)$ for spline and trigonometric series (and $K^3 / n=o(1)$ for power series), where $K$ is the series number of terms. The restriction on $K$ is relaxed to $K (\log K)/ n = o(1)$ in \cite{Huang2003b} for splines and in \cite{BCK2013} for wavelets, trigonometric and other series under i.i.d. data. We show that the optimal $L^2$ convergence rates are still attainable for splines, wavelets, trigonometric and other series under exponentially beta-mixing and $K (\log K)^2 / n = o(1)$.

We also show that feasible asymptotic inference can be performed on a
possibly nonlinear functional $f(h_{0})$ using the plug-in series LS
estimator $f(\widehat{h})$. We establish the asymptotic normality of $f(%
\widehat{h})$ and of the corresponding Student t statistic for weakly
dependent data under mild low-level conditions. When specializing to general
irregular (i.e., slower than $\sqrt{n}$-estimable) but sup-norm bounded
\emph{linear} functionals of spline or wavelet series LS estimators with
i.i.d. data, we obtain the asymptotic normality of $f(\widehat{h})$
\begin{equation*}
\frac{\sqrt{n}(f(\widehat{h})-f(h_{0}))}{V_{K}^{1/2}}\rightarrow _{d}N(0,1)
\end{equation*}%
under remarkably mild conditions of (1) uniform integrability ($\sup_{x\in
\mathcal{X}}E[\epsilon _{i}^{2}\{|\epsilon _{i}|>\ell
(n)\}|X_{i}=x]\rightarrow 0$ for any $\ell (n)\rightarrow \infty $ as $%
n\rightarrow \infty $), and (2) $K^{-p/d}\sqrt{n/V_{K}}=o(1)$, $(K\log
K)/n=o(1)$, where $V_{K}$ is the sieve
variance that grows with $K$ for irregular functionals. These conditions
coincide with the weakest known conditions in \cite{Huang2003} for the
pointwise asymptotic normality of spline LS estimators, except we also allow
for other irregular linear functionals of spline or wavelet LS estimators.\footnote{Most of the literature has focused on the case of i.i.d. or strictly stationary data. See \cite{Andrews1991} for asymptotic normality of linear functionals of series LS estimators with non-identically distributed data.}
When specializing to general irregular but sup-norm bounded \emph{nonlinear}
functionals of spline or wavelet series LS estimators with i.i.d. data, we
obtain asymptotic normality of $f(\widehat{h})$ (and of its t statistic)
under conditions (1) and (3) $K^{-p/d}\sqrt{n/V_{K}}=o(1)$, $K^{(2+\delta
)/\delta }(\log n)/n\lesssim 1$ (and $K^{(2+\delta )/\delta }(\log n)/n=o(1)$
for the t statistic) for $\delta \in (0,2)$ such that $E[\left\vert \epsilon
_{i}\right\vert ^{2+\delta }]<\infty $. These conditions are much weaker
than the well-known conditions in \cite{Newey1997} for the asymptotic
normality of a nonlinear functional and its t statistic of spline LS
estimator, namely $K^{-p/d}\sqrt{n}=o(1)$, $K^{4}/n=o(1)$ and $%
\sup_{x}E[\left\vert \epsilon _{i}\right\vert ^{4}|X_{i}=x]<\infty $.
Moreover, under a slightly more restrictive growth condition on $K$ but
without the need to increase $\delta $, we show that our mild sufficient
conditions for the i.i.d. case extend naturally to the weakly dependent case. Our conditions for the weakly dependent case relax the higher-order-moment requirement in \cite{ChenLiaoSun} for sieve t inference on nonlinear functionals of linear sieve time series LS regressions.

Our paper improves upon the existing results on the asymptotic normality of t statistics of possibly nonlinear functionals of a series LS estimator under dependent data by allowing for heavy-tailed errors $\epsilon_i$ and relaxing the growth rates of the series term $K$ but maintaining the bounded conditional error variance assumption. For i.i.d. data, \cite{Hansen2014} derives pointwise asymptotic normality for linear functionals of a series LS estimator allowing for unbounded regressors, unbounded conditional error variance, but requiring $E[\left\vert \epsilon
_{i}\right\vert ^{4+\eta }]<\infty $ for some $\eta > 0$. In addition to pointwise limiting distribution results, \cite{BCK2013} also provide uniform limit theory and uniform confidence intervals for linear functionals of a series LS estimator with i.i.d. data.

Our paper, \cite{BCK2013} and \cite{Hansen2014} all employ tools from recent random matrix theory to derive various new results for series LS estimation. \cite{BCK2013} are the first to apply the non-commutative Khinchin random matrix inequality for i.i.d. data. Instead, we apply an exponential inequality for sums of independent random matrices due to \cite{Tropp2012}. Our results for series LS with weakly dependent data rely crucially on our extension of Tropp's matrix exponential inequality from i.i.d. data to weakly dependent data. See \cite{Hansen2014} for other applications of Tropp's inequality to series LS estimators with i.i.d. data.

Since economic and financial time series data often have infinite forth
moments, the new improved rates and inference results in our paper should be
very useful to the literatures on nonparametric estimation and testing of
nonlinear time series models (see, e.g., \cite{Robinson1989}, \cite%
{LiHsiaoZinn}, \cite{FanYao}, \cite{Chen2013}). Moreover, our new
exponential inequality for sums of weakly dependent random matrices should
be useful in series LS estimation of spatially dependent models and in other
contexts as well.\footnote{%
In our ongoing work on sieve estimation of semi/nonparametric conditional
moment restriction models with time series data, this new exponential
inequality also enables us to establish asymptotic properties under weaker
conditions.}

The rest of the paper is organized as follows. Section \ref{main sec} first
derives general upper bounds on the sup-norm convergence rates of series LS
estimators with an arbitrary basis. It then shows that spline and wavelet
series LS estimators attain the optimal sup-norm rates, allowing for weakly
dependent data and heavy tailed error terms. It also presents general sharp $%
L^{2}$-norm convergence rates of series LS estimators with an arbitrary
basis under very mild conditions. Section \ref{inference sec} provides the
asymptotic normality of sieve t statistics for possibly nonlinear
functionals of $h_{0}$. Section \ref{ei sec} provides new exponential
inequalities for sums of weakly dependent random matrices, and a
reinterpretation of equivalence of the theoretical and empirical $L^{2}$
norms as a criterion regarding convergence of a certain random matrix.
Section \ref{sec-stable} shows the sup-norm stability of the empirical $%
L^{2} $ projections onto compactly supported wavelet bases, which provides a
tight upper bound on the sup-norm bias term for the wavelet series LS
estimator. The results in Sections \ref{ei sec} and \ref{sec-stable} are of
independent interest. Section \ref{sieve def} contains a brief review of
spline and wavelet sieve bases. Proofs and ancillary results are presented
in Section \ref{proofs sec}.

\paragraph{Notation:}

Let $\lambda _{\min }(\cdot )$ and $\lambda _{\max }(\cdot )$ denote the
smallest and largest eigenvalues, respectively, of a matrix. The exponent $%
^{-}$ denotes the Moore-Penrose generalized inverse. $\Vert \cdot \Vert $
denotes the Euclidean norm when applied to vectors and the matrix spectral
norm (i.e., largest singular value) when applied to matrices, and $\Vert
\cdot \Vert _{\ell ^{p}}$ denotes the $\ell ^{p}$ norm when applied to
vectors and its induced operator norm when applied to matrices (thus $\Vert
\cdot \Vert =\Vert \cdot \Vert _{\ell ^{2}}$). If $\{a_{n}:n\geq 1\}$ and $%
\{b_{n}:n\geq 1\}$ are two sequences of non-negative numbers, $a_{n}\lesssim
b_{n}$ means there exists a finite positive $C$ such that $a_{n}\leq Cb_{n}$
for all $n$ sufficiently large, and $a_{n}\asymp b_{n}$ means $a_{n}\lesssim
b_{n}$ and $b_{n}\lesssim a_{n}$. $\#{\mathcal{S}}$ denotes the cardinality
of a set ${\mathcal{S}}$ of finitely many elements. Given a strictly
stationary process $\{X_{i}\}$ and $1\leq p<\infty $, we let $L^{p}(X)$
denote the function space consisting of all (equivalence classes) of
measurable functions $f$ for which the $L^{p}(X)$ norm $\Vert f\Vert
_{L^{p}(X)}\equiv E[|f(X_{i})|^{p}]^{1/p}$ is finite, and we let $L^{\infty
}(X)$ denote the space of bounded functions under the sup norm $\Vert \cdot
\Vert _{\infty }$, i.e., if $f:{\mathcal{X}}\rightarrow {\mathbb{R}}$ then $%
\Vert f\Vert _{\infty }\equiv \sup_{x\in {\mathcal{X}}}|f(x)|$.

\section{Uniform Convergence Rates}

\label{main sec}

In this section we present some general results on uniform convergence
properties of nonparametric series LS estimators with weakly dependent data.

\subsection{Estimator and basic assumptions}

In nonparametric series LS estimation, the conditional mean function $h_{0}$
is estimated by least squares regression of $Y_{1},\ldots ,Y_{n}$ on a
vector of sieve basis functions evaluated at $X_{1},\ldots ,X_{n}$. The
standard series LS estimator of the conditional mean function $h_{0}$ is
\begin{equation}
\widehat{h}(x)=b^{K}(x)^{\prime }(B^{\prime }B)^{-}B^{\prime }Y
\label{e-hhatold}
\end{equation}%
where $b_{K1},\ldots ,b_{KK}$ are a collection of $K$ sieve basis functions
and
\begin{eqnarray}
b^{K}(x) &=&(b_{K1}(x),\ldots ,b_{KK}(x))^{\prime } \\
B &=&(b^{K}(X_{1}),\ldots ,b^{K}(X_{n}))^{\prime } \\
Y &=&(Y_{1},\ldots ,Y_{n})^{\prime }\,.
\end{eqnarray}%
Choosing a particular class of sieve basis function and the dimension $K$
are analogous to choosing the type of kernel and bandwidth, respectively, in
kernel regression techniques. The basis functions are chosen such that their
closed linear span $B_{K}=clsp\{b_{K1},\ldots ,b_{KK}\}$ can well
approximate the space of functions in which $h_{0}$ is assumed to belong.

When the data $\{(X_{i},Y_{i})\}_{i=1}^{n}$ are a random sample it is often
reasonable to assume that $X$ is supported on a compact set $\mathcal{X}%
\subset \mathbb{R}^{d}$. However, in a time-series setting it may be
necessary to allow the support $\mathcal{X}$ of $X$ to be infinite, as in
the example of nonparametric autoregressive regression with a student t
distributed error term. See, e.g., \cite{FanYao} and \cite{Chen2013} for
additional examples and references.

To allow for possibly unbounded support $\mathcal{X}$ of $X_{i}$ we modify
the usual series LS estimator and notion of convergence. First, we weight
the basis functions by a sequence of non-negative weighting functions $w_{n}:%
\mathcal{X}\rightarrow \{0,1\}$ given by
\begin{equation}
w_{n}(x)=\left\{
\begin{array}{cc}
1 & \mbox{if }x\in \mathcal{D}_{n} \\
0 & \mbox{otherwise}%
\end{array}%
\right.
\end{equation}%
where $\mathcal{D}_{n}\subseteq \mathcal{X}$ is compact, convex, and has
nonempty interior, and $\mathcal{D}_{n}\subseteq \mathcal{D}_{n+1}$ for all $%
n$. The resulting series LS estimator is then
\begin{equation}
\widehat{h}(x)=b_{w}^{K}(x)^{\prime }(B_{w}^{\prime }B_{w})^{-}B_{w}^{\prime
}Y  \label{e-newhdef}
\end{equation}%
where $b_{K1},\ldots ,b_{KK}$ are a collection of $K$ sieve basis functions
and
\begin{eqnarray}
b_{w}^{K}(x) &=&(b_{K1}(x)w_{n}(x),\ldots ,b_{KK}(x)w_{n}(x))^{\prime } \\
B_{w} &=&(b_{w}^{K}(X_{1}),\ldots ,b_{w}^{K}(X_{n}))^{\prime }\,.
\end{eqnarray}%
Second, we consider convergence in the (sequence of) weighted sup norm(s) $%
\Vert \cdot \Vert _{\infty ,w}$ given by
\begin{equation}
\Vert f\Vert _{\infty ,w}=\sup_{x}|f(x)w_{n}(x)|=\sup_{x\in \mathcal{D}%
_{n}}|f(x)|
\end{equation}%
This modification is made because simple functions, such as polynomials of $%
x $, have infinite sup norm when $X_{i}$ has unbounded support, but will
have finite weighted sup norm.

\begin{remark}
When $\mathcal{X}$ is compact we may simply set $w_n(x) = 1$ for all $x \in
\mathcal{X}$ and all $n$. With such a choice of weighting, the series LS
estimator with weighted basis trivially coincides with the series LS
estimator in (\ref{e-hhatold}) with unweighted basis, and $%
\|\cdot\|_{\infty,w} = \|\cdot\|_\infty$.
\end{remark}

When $\mathcal{X}$ is unbounded there are several possible choices for $%
\mathcal{D}_n$. For instance, we may take $\mathcal{D}_n = \mathcal{D}$ for
all $n$, where $\mathcal{D }\subseteq \mathcal{X}$ is a fixed compact convex
set with nonempty interior. This approach is not without precedent in the
nonparametric analysis of nonlinear time series models. For example, \cite%
{HuangShen2004} use a similar approach to trim extreme observations in
nonparametric functional coefficient regression models, following \cite%
{TjostheimAuestad1994}. More generally, we can consider an expanding
sequence of compact nonempty sets $\mathcal{D}_n \subset \mathcal{X}$ with $%
\mathcal{D}_n \subseteq \mathcal{D}_{n+1}$ for all $n$ and set $w_n(x) = \{x
\in \mathcal{D}_n\}$ for all $n$. For example, if $\mathcal{X }= \mathbb{R}%
^d $ we could take $\mathcal{D}_n = \{x \in \mathbb{R}^d : \|x\|_p \leq
r_n\} $ where $0 < r_n \leq r_{n+1} < \infty$ for all $n$. This approach is
similar to excluding functions far from the support of the data when
performing series LS estimation with a compactly-supported wavelet basis for
$L^2(\mathbb{R})$ or $L^2(\mathbb{R}^d) $. We defer estimation with smooth
weighting functions of the form $w_n(x) = (1+\|x\|^2)^{-\omega}$ or $w_n(x)
= \exp(-\|x\|^\omega)$ to future research.

We first introduce some mild regularity conditions that are satisfied by
typical regression models and most linear sieve bases.

\begin{assumption}
\label{X reg} (i) $\{X_{i}\}_{i=-\infty }^{\infty }$ is strictly stationary,
(ii) ${\mathcal{X}}\subseteq \mathbb{R}^{d}$ is convex and has nonempty
interior.
\end{assumption}

\begin{assumption}
\label{resid reg} (i) $\{\epsilon _{i},{\mathcal{F}}_{i-1}\}_{i=1}^{n}$ with
${\mathcal{F}}_{i-1}=\sigma (X_{i},\epsilon _{i-1},X_{i-1},\ldots )$ is a
strictly stationary martingale difference sequence, (ii) $E[\epsilon
_{i}^{2}|{\mathcal{F}}_{i-1}]$ is uniformly bounded for all $i\geq 1$,
almost surely, (iii) $E[|\epsilon _{i}|^{2+\delta }]<\infty $ for some $%
\delta >0$.
\end{assumption}

Let $N(\mathcal{D }_{n},\epsilon )$ denote the internal $\epsilon $-covering
number of $\mathcal{D }_{n}$ with respect to the Euclidean norm (i.e. the
minimum number of points $x_{1},\ldots ,x_{m}\in \mathcal{D }_{n}$ such that
the collection of $\epsilon $-balls centered at each of $x_{1},\ldots ,x_{m}$
cover $\mathcal{D }_{n})$.

\begin{assumption}
\label{a-wdelta} (i) $\mathcal{D }_{n}$ is compact, convex, has nonempty
interior, and $\mathcal{D}_{n}\subseteq \mathcal{D}_{n+1}$ for all $n$, (ii)
there exists $\nu _{1},\nu _{2}>0$ such that $N(\mathcal{D}_n,\epsilon
)\lesssim n^{\nu _{1}}\epsilon ^{-\nu _{2}}$.
\end{assumption}

Define $\zeta _{K,n}\equiv \sup_{x}\Vert b_{w}^{K}(x)\Vert $ and $\lambda
_{K,n}\equiv \left[ \lambda _{\min
}(E[b_{w}^{K}(X_{i})b_{w}^{K}(X_{i})^{\prime }])\right] ^{-1/2}$.

\begin{assumption}
\label{sieve reg gen} (i) there exist $\omega _{1},\omega _{2}\geq 0$ s.t. $%
\sup_{x\in \mathcal{D}_{n}}\Vert \nabla b_{w}^{K}(x)\Vert \lesssim n^{\omega
_{1}}K^{\omega _{2}}$, (ii) there exist $\varpi _{1}\geq 0,\varpi _{2}>0$
s.t. $\zeta _{K,n}\lesssim n^{\varpi _{1}}K^{\varpi _{2}}$, (iii) $\lambda
_{\min }(E[b_{w}^{K}(X_{i})b_{w}^{K}(X_{i})^{\prime }])>0$ for each $K$ and $%
n$.
\end{assumption}

Assumptions \ref{X reg} and \ref{resid reg} trivially nest i.i.d. sequences,
but also allow the regressors to exhibit quite general weak dependence. Note
that Assumption \ref{resid reg}(ii) reduces to $\sup_{x}E[\epsilon
_{i}^{2}|X_{i}=x]<\infty $ in the i.i.d. case.
Suitable choice of $\delta $ in Assumption \ref{resid reg}(iii) for
attainability of the optimal uniform rate will be explained subsequently.
Strict stationarity of $\{\epsilon _{i}\}$ in Assumption \ref{resid reg} may
be dropped provided the sequence $\{|\epsilon _{i}|^{2+\delta }\}$ is
uniformly integrable. However, strict stationarity is used to present simple
sufficient conditions for the asymptotic normality of functionals of $%
\widehat{h}$ in Section \ref{inference sec}.

Assumption \ref{a-wdelta} is trivially satisfied when $\mathcal{X}$ is
compact and $\mathcal{D}_n = \mathcal{X}$ for all $n$. More generally, when $%
\mathcal{X}$ is noncompact and $\mathcal{D}_{n}$ is an expanding sequence of
compact subsets of $\mathcal{X}$ as described above, Assumption \ref%
{a-wdelta}(ii) is satisfied provided each $\mathcal{D}_{n}$ is contained in
an Euclidean ball of radius $r_{n}\lesssim n^{\nu }$ for some $\nu >0$.%
\footnote{%
By translational invariance we may assume that $\mathcal{D}_{n}$ is centered
at the origin. Then $\mathcal{D}_{n}\subseteq \mathcal{R}_{n}=[-
r_{n},r_{n}]^{d}$. We can cover $\mathcal{R}_{n}$ with $(r_n/\epsilon)^d$ $%
\ell^\infty$-balls of radius $\epsilon $, each of which is contained in an
Euclidean ball of radius $\epsilon \sqrt d$. Therefore, $N(\mathcal{D}%
_{n},\epsilon )\leq (\sqrt dr_{n})^{d}\epsilon ^{-d}\lesssim n^{\nu
d}\epsilon ^{-d}$.}

Assumption \ref{sieve reg gen} is a mild regularity condition on the sieve
basis functions. When $\mathcal{X}$ is compact and rectangular this
assumption is satisfied by all the widely used series (or linear sieve
bases) with $\lambda _{K,n}\lesssim 1$, and $\zeta _{K,n}\lesssim \sqrt{K}$
for tensor-products of univariate polynomial spline, trigonometric
polynomial or wavelet bases, and $\zeta _{K,n}\lesssim K$ for
tensor-products of power series or orthogonal polynomial bases (see, e.g.,
\cite{Newey1997}, \cite{Huang1998}, and \cite{Chen2007}). See \cite%
{DeVoreLorentz} and \cite{BCK2013} for additional bases with either $\zeta _{K,n}\lesssim \sqrt{%
K}$ or $\zeta _{K,n}\lesssim K$ properties.

Let $\widetilde{b}_{w}^{K}(x)$ denote the orthonormalized vector of basis
functions, namely
\begin{equation}
\widetilde{b}_{w}^{K}(x)=E[b_{w}^{K}(X_{i})b_{w}^{K}(X_{i})^{\prime
}]^{-1/2}b_{w}^{K}(x)\,,
\end{equation}%
and let $\widetilde{B}_{w}=(\widetilde{b}_{w}^{K}(X_{1}),\ldots ,\widetilde{b%
}_{w}^{K}(X_{n}))^{\prime }$.

\begin{assumption}
\label{a-gram} Either: (a) $\{X_{i}\}_{i=1}^{n}$ is i.i.d. and $\zeta
_{K,n}\lambda _{K,n}\sqrt{(\log K)/n}=o(1)$, or \newline
(b) $\Vert (\widetilde{B}_{w}^{\prime }\widetilde{B}_{w}/n)-I_{K}\Vert
=o_{p}(1)$.
\end{assumption}

Assumption \ref{a-gram} is a mild but powerful condition that ensures the
empirical and theoretical $L^{2}$ norms are equivalent over the linear sieve
space wpa1 (see Section \ref{ei sec} for details). In fact, to establish
many of our results below with weakly dependent data, nothing further about
the weak dependence properties of the regressor process $\{X_{i}\}_{i=1}^{n}$
needs to be assumed beyond convergence of $\Vert \widetilde{B}_{w}^{\prime }%
\widetilde{B}_{w}/n-I_{K}\Vert $ to zero. In the i.i.d. case, the following
Lemma shows that part (a) of Assumption \ref{a-gram} automatically implies $%
\Vert (\widetilde{B}_{w}^{\prime }\widetilde{B}_{w}/n)-I_{K}\Vert =o_{p}(1)$.

\begin{lemma}
\label{Bconvi.i.d.} Under Assumption \ref{sieve reg gen}(iii), if $%
\{X_{i}\}_{i=1}^{n}$ is i.i.d. then
\begin{equation}
\Vert (\widetilde{B}_{w}^{\prime }\widetilde{B}_{w}/n)-I_{K}\Vert
=O_{p}\left( \zeta _{K,n}\lambda _{K,n}\sqrt{{(\log K)}/{n}}\right) =o_{p}(1)
\notag
\end{equation}%
provided $\zeta _{K,n}\lambda _{K,n}\sqrt{{(\log K)}/{n}}=o(1)$.
\end{lemma}

\begin{remark}
\label{rmk2.2} Consider the compact support case in which $\mathcal{X}%
=[0,1]^{d}$ and $w_{n}(x)=1$ for all $x\in \mathcal{X}$ and all $n$ (so that
$b_{Kk}(x)w_{n}(x)=b_{Kk}(x)$ for all $n$ and $K$) and suppose the density
of $X_{i}$ is uniformly bounded away from zero and infinity over $\mathcal{X}
$. In this setting, we have $\lambda _{K,n}\lesssim 1$. If $%
\{X_{i}\}_{i=1}^{n}$ is i.i.d., then Assumption \ref{a-gram} is satisfied
with $\sqrt{(K\log K)/n}=o(1)$ for spline, trigonometric polynomial or
wavelet bases, and with $K\sqrt{(\log K)/n}=o(1)$ for (tensor-product) power
series.
\end{remark}

When the regressors are $\beta $-mixing (see Section \ref{ei sec} for
definition), the following Lemma shows that Assumption \ref{a-gram}(b) is
still easily satisfied.

\begin{lemma}
\label{Bconvbeta} Under Assumption \ref{sieve reg gen}(iii), if $%
\{X_{i}\}_{i=-\infty }^{\infty }$ is strictly stationary and $\beta $-mixing
with mixing coefficients such that one can choose an integers $q=q(n)\leq
n/2 $ with $\beta (q)n/q=o(1)$, then
\begin{equation}
\Vert (\widetilde{B}_{w}^{\prime }\widetilde{B}_{w}/n)-I_{K}\Vert
=O_{p}\left( \zeta _{K,n}\lambda _{K,n}\sqrt{{(q\log K)}/{n}}\right)
=o_{p}(1)  \notag
\end{equation}%
provided $\zeta _{K,n}\lambda _{K,n}\sqrt{{(q\log K)}/{n}}=o(1)$.
\end{lemma}

\begin{remark}
Consider the compact support case from Remark \ref{rmk2.2}, with $%
\{X_{i}\}_{i=-\infty }^{\infty }$ strictly stationary and $\beta $-mixing.

\begin{enumerate}
\item[(i)] Exponential $\beta $-mixing: Assumption \ref{a-gram}(b) is
satisfied with $\sqrt{K(\log n)^{2}/n}=o(1)$ for (tensor-product) spline,
trigonometric polynomial or wavelet bases, and with $K\sqrt{(\log n)^{2}/n}%
=o(1)$ for (tensor-product) power series.

\item[(ii)] Algebraic $\beta $-mixing at rate $\gamma $: Assumption \ref%
{a-gram}(b) is satisfied with $\sqrt{(K\log K)/n^{\gamma /(1+\gamma )}}=o(1)$
for (tensor-product) spline, trigonometric polynomial or wavelet bases, and
with $K\sqrt{(\log K)/n^{\gamma /(1+\gamma )}}=o(1)$ for (tensor-product)
power series.
\end{enumerate}
\end{remark}

\subsection{A general upper bound on uniform convergence rates}

Let $B_{K,w}=clsp\{b_{K1}w_{n},\ldots ,b_{KK}w_{n}\}$ be a general weighted
linear sieve space. Let $\widetilde{h}$ denote the projection of $h_{0}$
onto $B_{K,w}$ under the empirical measure, that is,
\begin{equation}
\widetilde{h}(x)=b_{w}^{K}(x)^{\prime }(B_{w}^{\prime
}B_{w})^{-}B_{w}^{\prime }H_{0}=\widetilde{b}_{w}^{K}(x)^{\prime }(%
\widetilde{B}_{w}^{\prime }\widetilde{B}_{w})^{-}\widetilde{B}_{w}^{\prime
}H_{0}  \label{htilde0}
\end{equation}%
where $H_{0}=(h_{0}(X_{1}),\ldots ,h_{0}(X_{n}))^{\prime }$. The sup-norm
distance $\Vert \widehat{h}-h_{0}\Vert _{\infty ,w}$ may be trivially
bounded using
\begin{eqnarray}
\Vert \widehat{h}-h_{0}\Vert _{\infty ,w} &\leq &\Vert h_{0}-\widetilde{h}%
\Vert _{\infty ,w}+\Vert \widehat{h}-\widetilde{h}\Vert _{\infty ,w} \\
&=: &\mbox {bias term}+\mbox {variance term}\,.
\end{eqnarray}%

\textbf{Sharp bound on the sup-norm variance term}. The following result
establishes a sharp uniform convergence rate of the variance term for an
arbitrary linear sieve space. Convergence is established in sup norm rather
than the weighted sup norm $\Vert \cdot \Vert _{\infty ,w}$ because both $%
\widehat{h}$ and $\widetilde{h}$ have support $\mathcal{D}_{n}$. Therefore, $%
\Vert \widehat{h}-\widetilde{h}\Vert _{\infty }=\sup_{x\in \mathcal{D}_{n}}|%
\widehat{h}(x)-\widetilde{h}(x)|=\Vert \widehat{h}-\widetilde{h}\Vert
_{\infty ,w}$.

\begin{lemma}
\label{reglem gen} Let Assumptions \ref{X reg}(i)(ii), \ref{resid reg}%
(i)(ii)(iii), \ref{a-wdelta}, \ref{sieve reg gen}, and \ref{a-gram} hold.
Then
\begin{equation*}
\Vert \widehat{h}-\widetilde{h}\Vert _{\infty }=O_{p}\left( \zeta
_{K,n}\lambda _{K,n}\sqrt{(\log n)/{n}}\right) =o_{p}(1)
\end{equation*}%
as $n,K\rightarrow \infty $ provided the following are satisfied:

\begin{itemize}
\item[(i)] $(\zeta _{K,n}\lambda _{K,n})^{(2+\delta )/\delta }\lesssim \sqrt{%
(n/\log n)}$;

\item[(ii)] either: (a) $\{(X_{i},Y_{i})\}_{i=1}^{n}$ are i.i.d., or (b) $%
\sqrt{\frac{K}{\log n}}\times \Vert (\widetilde{B}_{w}^{\prime }\widetilde{B}%
_{w}/n)-I_{K}\Vert =O_{p}(1)$.
\end{itemize}
\end{lemma}

\begin{remark}
Weak dependence of the regressor process $\{X_{i}\}$ is implicitly captured
by the speed of convergence of $\Vert (\widetilde{B}_{w}^{\prime }\widetilde{%
B}_{w}/n)-I_{K}\Vert$. If $\{X_{i}\}$ is exponentially $\beta $-mixing
(respectively algebraically $\beta $-mixing at rate $\gamma $), condition
(ii)(b) in Lemma \ref{reglem gen} is satisfied provided $\zeta _{K,n}\lambda
_{K,n}\sqrt{(K\log n)/n}=O(1)$ (respectively $\zeta _{K,n}\lambda _{K,n}%
\sqrt{K/n^{\gamma /(1+\gamma )}}=O(1)$); see Lemma \ref{Bconvbeta}.
\end{remark}

\textbf{General bound on the sup-norm bias term}. With our sharp bound on
the variance term $\Vert \widehat{h}-\widetilde{h}\Vert _{\infty ,w}$ in
hand it remains to provide a calculation for the bias term $\Vert h_{0}-%
\widetilde{h}\Vert _{\infty ,w}$. Let $P_{K,w,n}$ be the (empirical)
projection operator onto $B_{K,w}\equiv clsp\{b_{K1}w_{n},\ldots
,b_{KK}w_{n}\}$, namely
\begin{equation}
P_{K,w,n}h(x)=b_{w}^{K}(x)^{\prime }\left( \frac{B_{w}^{\prime }B_{w}}{n}%
\right) ^{-}\frac{1}{n}\sum_{i=1}^{n}b_{w}^{K}(X_{i})h(X_{i})=\widetilde{b}%
_{w}^{K}(x)^{\prime }(\widetilde{B}_{w}^{\prime }\widetilde{B}_{w})^{-}%
\widetilde{B}_{w}^{\prime }H  \label{e-pkwn def}
\end{equation}%
where $H=(h(X_{1}),\ldots ,h(X_{n}))^{\prime }$. $P_{K,w,n}$ is a well
defined operator: if $L_{w,n}^{2}(X)$ denotes the space of functions with
norm $\Vert \cdot \Vert _{w,n}$ where $\Vert f\Vert _{w,n}^{2}=\frac{1}{n}%
\sum_{i=1}^{n}f(X_{i})^{2}w_{n}(X_{i})$, then $P_{K,w,n}:L_{w,n}^{2}(X)%
\rightarrow L_{w,n}^{2}(X)$ is an orthogonal projection onto $B_{K,w}$
whenever $B_{w}^{\prime }B_{w}$ is invertible (which it is wpa1 under
Assumptions \ref{sieve reg gen}(iii) and \ref{a-gram}).

One way to control the bias term $\Vert h_{0}-\widetilde{h}\Vert _{\infty
,w} $ is to bound $P_{K,w,n}$ in sup norm. Note that $\widetilde{h}%
=P_{K,w,n}h_{0}$. Let $L_{w,n}^{\infty }(X)$ denote the space of functions
for which $\sup_{x}|f(x)w_{n}(x)|<\infty $ and let
\begin{equation*}
\Vert P_{K,w,n}\Vert _{\infty ,w}=\sup_{h\in L_{w,n}^{\infty }(X):\Vert
h\Vert _{\infty ,w}\neq 0}\frac{\Vert P_{K,w,n}h\Vert _{\infty ,w}}{\Vert
h\Vert _{\infty ,w}}
\end{equation*}%
denote the (weighted sup) operator norm of $P_{K,w,n}$. The following crude
bound on $\Vert P_{K,w,n}\Vert _{\infty }$ is valid for general linear sieve
bases and weakly dependent regressors.

\begin{remark}
\label{Pkw bound} Let Assumptions \ref{sieve reg gen}(iii) and \ref{a-gram}
hold. Then: $\Vert P_{K,w,n}\Vert _{\infty }\leq \sqrt{2}\zeta _{K,n}\lambda
_{K,n}$ wpa1.
\end{remark}

More refined bounds on $\Vert P_{K,w,n}\Vert _{\infty ,w}$ may be derived
for particular linear sieves with local properties, such as splines and
wavelets stated below. These more refined bounds, together with the
following Lemma, lead to the optimal uniform convergence rates of series LS
estimators with the particular linear sieves.

\begin{lemma}
\label{rate gen} Let the assumptions and conditions of Lemma \ref{reglem gen}
hold. Then: (1)
\begin{equation*}
\Vert \widehat{h}-{h}_{0}\Vert _{\infty }\leq O_{p}\left( \zeta
_{K,n}\lambda _{K,n}\sqrt{(\log n)/{n}}\right) +\left( 1+\Vert
P_{K,w,n}\Vert _{\infty ,w}\right) \inf_{h\in B_{K,w}}\Vert h_{0}-h\Vert
_{\infty ,w}.
\end{equation*}%
(2) Further, if the linear sieve satisfies $\zeta _{K,n}\lambda
_{K,n}\lesssim \sqrt{K}$ and $\Vert P_{K,w,n}\Vert _{\infty ,w}=O_{p}(1)$,
then
\begin{equation*}
\Vert \widehat{h}-{h}_{0}\Vert _{\infty }\leq O_{p}\left( \sqrt{(K\log n)/{n}%
}+\inf_{h\in B_{K,w}}\Vert h_{0}-h\Vert _{\infty ,w}\right) .
\end{equation*}
\end{lemma}

\subsection{Attainability of optimal uniform convergence rates}

We now turn to attainability of the optimal uniform rate of \cite{Stone1982}
by specific series LS estimators. To fix ideas, in what follows we take $%
\mathcal{D}_{n}=\mathcal{D}=[0,1]^{d}\subseteq \mathcal{X}$ for all $n$,
whence $\Vert f\Vert _{\infty ,w}=\sup_{x\in \mathcal{D}}|f(x)|$. Let $%
\Lambda ^{p}([0,1]^{d})$ denote a H\"{o}lder space of smoothness $p$ on the
domain $[0,1]^{d}$ (see, e.g. \cite{Chen2007} for definition). Let $%
\mbox
{BSpl}(K,[0,1]^{d},{\gamma })$ denote a B-spline sieve of degree $\gamma $
and dimension $K$ on the domain $[0,1]^{d}$, and let $\mbox {Wav}%
(K,[0,1]^{d},{\gamma })$ denote a Wavelet sieve basis of regularity $\gamma $
and dimension $K$ on the domain $[0,1]^{d}$ (see Section \ref{sieve def} for
details on construction of these sieve bases). Because our bases have been
constructed to have support $[0,1]^{d}$ we trivially have $%
b_{Kk}(x)=b_{Kk}(x)w_{n}(x)$ for all $k=1,\ldots ,K$ and all $n$ and $K$.
Recall $B_{K}\equiv clsp\{b_{K1},\ldots ,b_{KK}\}$ is the linear sieve space.

The following assumptions on the conditional mean function and the sieve
basis functions are sufficient for attaining the optimal uniform convergence
rate.

\setcounter{assumption}{0}

\begin{assumption}[continued]
(iii) $\mathcal{D}_{n}=\mathcal{D}=[0,1]^{d}\subseteq \mathcal{X}$ for all $%
n $, (iv) the unconditional density of $X_{i}$ is uniformly bounded away
from zero and infinity on ${\mathcal{D}}$.
\end{assumption}

\setcounter{assumption}{5}

\begin{assumption}
\label{parameter regression} The restriction of $h_{0}$ to $[0,1]^{d}$
belongs to $\Lambda ^{p}([0,1]^{d})$ for some $p>0$.
\end{assumption}

\begin{assumption}
\label{b sieve} The sieve $B_{K}$ is BSpl$(K,[0,1]^{d},{\gamma })$ or Wav$%
(K,[0,1]^{d},{\gamma })$ with $\gamma > \max \{p,1\}$.
\end{assumption}

Assumptions \ref{X reg} and \ref{parameter regression} are standard
regularity conditions used in derivation of optimal uniform convergence
rates \citep{Stone1982,Tsybakov2009}. Assumption \ref{X reg}(iii) implies
Assumption \ref{a-wdelta}. Assumptions \ref{X reg} and \ref{b sieve} imply
Assumption \ref{sieve reg gen} with $\zeta _{K,n}\lesssim \sqrt{K}$ and $%
\lambda _{K,n}\lesssim 1$.

Let $h_{0,K}^{\ast }\in B_{K}$ solve $\inf_{h\in B_{K}}\Vert h_{0}-h\Vert
_{\infty ,w}$. Assumptions \ref{X reg}, \ref{parameter regression} and \ref%
{b sieve} imply that $\Vert h_{0}-h_{0,K}^{\ast }\Vert _{\infty ,w}\lesssim
K^{-p/d}$ (see, e.g. \cite{DeVoreLorentz}, \cite{Huang1998}, \cite{Chen2007}%
). Previously \cite{Huang2003} showed that $\Vert P_{K,w,n}\Vert _{\infty
,w}\lesssim 1$ wpa1 for spline bases with i.i.d. data. In the proof of
Theorem \ref{sup norm rate regression} we extend his result to allow for
weakly dependent regressors. In addition, Theorem \ref{c-pstable} in Section %
\ref{sec-stable} shows that $\Vert P_{K,w,n}\Vert _{\infty ,w}\lesssim 1$
wpa1 for wavelet bases with i.i.d. or weakly dependent regressors.

\begin{theorem}
\label{sup norm rate regression} Let Assumptions \ref{X reg}, \ref{resid reg}%
(i)(ii)(iii) (with $\delta \geq d/p$), \ref{parameter regression} and \ref{b
sieve} hold. If $K\asymp (n/\log n)^{d/(2p+d)}$, then
\begin{equation*}
\Vert \widehat{h}-h_{0}\Vert _{\infty ,w}=O_{p}((n/\log n)^{-p/(2p+d)})
\end{equation*}%
provided that either (a), (b), or (c) is satisfied:

\begin{itemize}
\item[(a)] $\{(X_{i},Y_{i})\}_{i=1}^{n}$ is i.i.d.;

\item[(b)] $\{X_{i}\}_{i=1}^{n}$ is exponentially $\beta $-mixing and $d<2p$;

\item[(c)] $\{X_{i}\}_{i=1}^{n}$ is algebraically $\beta $-mixing at rate $%
\gamma $ and $(2+\gamma )d<2\gamma p$.
\end{itemize}
\end{theorem}

Theorem \ref{sup norm rate regression} states that the optimal uniform
convergence rates of \cite{Stone1982} are achieved by spline and wavelet
series LS estimators with i.i.d. data whenever $\delta \geq d/p$. If the
regressors are exponentially $\beta $-mixing the optimal rate of convergence
is achieved with $\delta \geq d/p$ and $d<2p$. The restrictions $\delta \geq
d/p$ and $(2+\gamma )d<2\gamma p$ for algebraically $\beta $-mixing (at a
rate $\gamma $) reduces naturally towards the exponentially $\beta $-mixing
restrictions as the dependence becomes weaker (i.e. $\gamma $ becomes
larger). In all cases, for a fixed dimension $d\geq 1$, a smoother function
(i.e. bigger $p$) means a lower value of $\delta $, and hence fatter-tailed
error terms $\epsilon _{i}$, are permitted while still obtaining the optimal
uniform convergence rate. In particular this is achieved with $\delta =d/p<2$%
.

\textbf{Discussion of closely related results}. Under Assumption \ref{X reg}
with i.i.d. data and compact $\mathcal{X}$, we can set the weight to be $%
w_{n}=1$ for all $n$. Let $P_{K}$ denote the $L^{2}(X)$ orthogonal
projection operator onto $B_{K}$, given by
\begin{equation}
P_{K}h(x)=b^{K}(x)^{\prime }\left( E[b^{K}(X_{i})b^{K}(X_{i})^{\prime
}]\right) ^{-1}E[b^{K}(X_{i})h(X_{i})]  \label{e-pkdef}
\end{equation}%
for any $h\in L^{2}(X)$, and define its $L^{\infty }$ operator norm:
\begin{equation}
\Vert P_{K}\Vert _{\infty }:=\sup_{h\in L^{\infty }(X):\Vert h\Vert _{\infty
}\neq 0}\frac{\Vert P_{K}h\Vert _{\infty }}{\Vert h\Vert _{\infty }}.
\label{e-pksupnorm}
\end{equation}%
Under i.i.d. data, Assumptions \ref{X reg}, \ref{resid reg}(i)(ii) (i.e., $%
E[\epsilon _{i}|X_{i}]=0,$ $\sup_{x}E[\left\vert \epsilon _{i}\right\vert
^{2}|X_{i}=x]<\infty $), \ref{parameter regression}, and the conditions $%
\lambda _{K,n}\lesssim 1$, $\zeta _{K,n}^{2}K/{n}=o(1)$ and $\Vert
h_{0}-h_{0,K}^{\ast }\Vert _{\infty }\lesssim K^{-p/d}$ on the series basis,
\cite{Newey1997} derived the following sup-norm convergence rates for series
LS estimators with an arbitrary basis:%
\begin{equation}
\Vert \widehat{h}-{h}_{0}\Vert _{\infty }\leq O_{p}\left( \zeta _{K,n}\sqrt{%
K/{n}}+\zeta _{K,n}K^{-p/d}\right) .  \label{newey-bound}
\end{equation}%
By Remark \ref{Pkw bound}, under the same set of mild conditions imposed in
\cite{Newey1997} (except allowing for weakly dependent regressors), the
bound (\ref{newey-bound}) can be slightly improved to
\begin{equation}
\Vert \widehat{h}-{h}_{0}\Vert _{\infty }\leq O_{p}\left( \zeta _{K,n}\sqrt{%
K/{n}}+\Vert P_{K,w,n}\Vert _{\infty ,w}K^{-p/d}\right) .
\label{sup-rate-L2}
\end{equation}%
It is clear that the general bounds in (\ref{newey-bound}) and (\ref%
{sup-rate-L2}) with an arbitrary basis are not optimal, but they are derived
under the minimal moment restriction of Assumption \ref{resid reg}(ii)
without the existence of higher-than-second moments (i.e., $\delta =0$ in
Assumption \ref{resid reg}(iii)).

Under the extra moment condition $\sup_{x}E[\left\vert \epsilon
_{i}\right\vert ^{4}|X_{i}=x]<\infty $, \cite{deJong2002} obtained the
following general bound on sup-norm rates for series LS estimators with an
arbitrary basis:%
\begin{equation}
\Vert \widehat{h}-{h}_{0}\Vert _{\infty }\leq O_{p}\left( \zeta _{K,n}\sqrt{%
(\log n)/{n}}+K^{-p/d}+\Vert P_{K}h_{0}-h_{0,K}^{\ast }\Vert _{\infty
}\right) .  \label{dejong-rate}
\end{equation}%
\cite{deJong2002} did not provide sharp bounds for $\Vert
P_{K}h_{0}-h_{0,K}^{\ast }\Vert _{\infty }$ for any particular basis, and
was therefore unable to attain the optimal convergence rate $\Vert \widehat{h%
}-h_{0}\Vert _{\infty }=O_{p}((n/\log n)^{-p/(2p+d)})$ of \cite{Stone1982}.
Note that%
\begin{equation}
\Vert P_{K}h_{0}-h_{0,K}^{\ast }\Vert _{\infty }=\Vert
P_{K}(h_{0}-h_{0,K}^{\ast })\Vert _{\infty }\leq \Vert P_{K}\Vert _{\infty
}\Vert h_{0}-h_{0,K}^{\ast }\Vert _{\infty }\lesssim \Vert P_{K}\Vert
_{\infty }K^{-p/d}.
\end{equation}%
Given the newly derived sharp bounds of $\Vert P_{K}\Vert _{\infty }\lesssim
1$ in \cite{Huang2003} for splines, in \cite{BCK2013} for the local
polynomial partition series, and in our paper (Theorem \ref{t-wavd}) for
wavelets, one could now apply \cite{deJong2002}'s result (\ref{dejong-rate})
to conclude the attainability of the optimal sup-norm rate by spline, local
polynomial partition and wavelet series LS estimators for i.i.d. data.
However, \cite{deJong2002}'s result (\ref{dejong-rate}) is proved under the
strong bounded conditional fourth moment condition of $\sup_{x}E[\left\vert
\epsilon _{i}\right\vert ^{4}|X_{i}=x]<\infty $ and the side condition $%
\zeta _{K,n}^{2}K/{n}=o(1)$.

Recently, \cite{CattaneoFarrell} proved that a local
polynomial partitioning regression estimator can attain the optimal sup-norm
rate under i.i.d. data and the conditional moment condition $%
\sup_{x}E[\left\vert \epsilon _{i}\right\vert ^{2+\delta }|X_{i}=x]<\infty $
for some $\delta \geq \max (1,d/p)$. \cite{BCK2013} show that spline,
local polynomial partition and wavelet LS estimators attain the optimal sup-norm rate
under i.i.d. data and the conditional moment condition $\sup_{x}E[\left\vert
\epsilon _{i}\right\vert ^{2+\delta }|X_{i}=x]<\infty $ for some $\delta
>d/p $.\footnote{Our result on sup-norm stability of wavelet basis is used by \cite{BCK2013} for the optimal sup-norm rate of the wavelet LS estimator.} By contrast, we require a weaker unconditional moment condition $%
E[\left\vert \epsilon _{i}\right\vert ^{2+(d/p)}]<\infty $ for spline and
wavelet LS estimators to attain the optimal uniform convergence rate,
allowing for both i.i.d. data and weakly dependent regressors.\footnote{%
Our proof strategy for the i.i.d. data case is similar to that in a draft note by \cite{ChenHuang03}, who applied \cite{Huang2003}'s sup-norm stability of spline basis to obtain the
optimal sup-norm rate for spline series LS estimator under
the condition $E[\left\vert \epsilon _{i}\right\vert ^{2+\delta }]<\infty $
for some $\delta >d/p$. The authors carelessly set $\delta =2$, but did not like
the strong condition $E[\left\vert \epsilon _{i}\right\vert ^{4}]<\infty $
and hence did not finish the paper. The note was circulated among
some colleagues and former students of the authors, and is available upon request.} It remains an
open question whether one could obtain the optimal sup-norm convergence rate
without imposing a finite higher-than-second unconditional moment of the
error term, however.

\subsection{A general sharp bound on $L^{2}$ convergence rates}

In this subsection, we present a simple but sharp upper bound on the $L^{2}$
(or root mean square) convergence rates of series LS estimators with an
arbitrary basis and weakly dependent regressors.

Recall that $B_{K,w}=clsp\{b_{K1}w_{n},\ldots ,b_{KK}w_{n}\}$ is a general
weighted linear sieve space and $\widetilde{h}=P_{K,w,n}h_{0}$ is defined in
(\ref{htilde0}). Let $h_{0,K}$ be the $L^{2}(X)$ orthogonal projection of
the conditional mean function $h_{0}$ onto $B_{K,w}$.

\begin{lemma}
\label{lem-hhat l2} Let Assumptions \ref{X reg}(i), \ref{resid reg}(i)(ii), %
\ref{sieve reg gen}(iii) and \ref{a-gram} hold. Then:
\begin{equation*}
\Vert \widehat{h}-\widetilde{h}\Vert _{L^{2}(X)}=O_{p}\left( \sqrt{K/n}%
\right) \text{ and }\Vert \widetilde{h}-h_{0}\Vert _{L^{2}(X)}=O_{p}\left(
\Vert h_{0}-h_{0,K}\Vert _{L^{2}(X)}\right) \,.
\end{equation*}
\end{lemma}

Lemmas \ref{lem-hhat l2}, \ref{Bconvi.i.d.} and \ref{Bconvbeta} immediately
imply the following result.

\begin{remark}
\label{L2-rate} Let Assumptions \ref{X reg}(i), \ref{resid reg}(i)(ii) and $%
\lambda _{K,n}\lesssim 1$ hold. Then: (1)%
\begin{equation}
\Vert \widehat{h}-h_{0}\Vert _{L^{2}(X)}=O_{p}\left( \sqrt{K/n}+\Vert
h_{0}-h_{0,K}\Vert _{L^{2}(X)}\right) ,  \label{L2rate}
\end{equation}%
provided that either (1.a), (1.b) or (1.c) is satisfied:

\begin{itemize}
\item[(1.a)] $\{X_{i}\}_{i=1}^{n}$ is i.i.d., and $\zeta _{K,n}\sqrt{{(\log
K)}/{n}}=o(1)$;

\item[(1.b)] $\{X_{i}\}_{i=1}^{n}$ is exponentially $\beta $-mixing, and $%
\zeta _{K,n}\sqrt{(\log n)^{2}/{n}}=o(1)$;

\item[(1.c)] $\{X_{i}\}_{i=1}^{n}$ is algebraically $\beta $-mixing at rate $%
\gamma $, and $\zeta _{K,n}\sqrt{{(\log K)}/{n}^{\gamma /(1+\gamma )}}=o(1).$
\end{itemize}

(2) Further, if Assumptions \ref{X reg} and \ref{parameter regression} hold
and $K\asymp (n/\log n)^{d/(2p+d)}$, then%
\begin{equation*}
\Vert \widehat{h}-h_{0}\Vert _{L^{2}(X)}=O_{p}(n^{-p/(2p+d)})
\end{equation*}%
provided that either (2.a) or (2.b) is satisfied:

\begin{itemize}
\item[(2.a)] $\{X_{i}\}_{i=1}^{n}$ is i.i.d. or exponentially $\beta $%
-mixing: with $p>0$ for trigonometric polynomial, spline or wavelet series,
and $p>d/2$ for power series;

\item[(2.b)] $\{X_{i}\}_{i=1}^{n}$ is algebraically $\beta $-mixing at rate $%
\gamma $: with $p>d/(2\gamma )$ for trigonometric polynomial, spline or
wavelet series; and $p>d(2+\gamma )/(2\gamma )$ for power series.
\end{itemize}
\end{remark}

With i.i.d. data under the condition $\lambda _{K,n}\lesssim 1$, \cite%
{Newey1997} derived the same sharp $L^{2}$ rate in (\ref{L2rate}) for series
LS estimators under the restriction $\zeta _{K,n}^{2}K/{n}=o(1)$. \cite%
{Huang2003b} showed that spline LS estimator has the same $L^{2}$ rate under
the much weaker condition $K(\log K)/n=o(1)$. Both \cite{BCK2013} and our Remark \ref{L2-rate} part (1.a)
extend \cite{Huang2003b}'s weakened condition to other bases satisfying $%
\zeta _{K,n}\lesssim \sqrt{K}$ (such as trigonometric polynomial and
wavelet) for series LS regression with i.i.d. data. In addition, Remark \ref{L2-rate} part (1.b) shows that the mild condition $K(\log
K)^2 /n=o(1)$ suffices for trigonometric polynomial, wavelet, spline and other
bases satisfying $\zeta _{K,n}\lesssim \sqrt{K}$ for exponentially $\beta $%
-mixing regressors.

With weakly-dependent data, \cite{ChenShen} derived $L^{2}$ rates for LS
regression using various linear or nonlinear sieves with beta-mixing
sequence under higher-than-second moment restriction (see Proposition 5.1 in
\cite{ChenShen}). \cite{HuangYang} and others derived the optimal $L^{2}$
rate for spline LS regression with strongly mixing sequence assuming a
uniformly bounded higher-than-second conditional moment. Thanks to our Lemma %
\ref{Bconvbeta}, we are able to show that series LS estimators with
arbitrary bases attain the optimal $L^{2}$ convergence rate with beta-mixing
regressors under a uniformly bounded second conditional moment condition on
the residuals. Our result should be very useful to nonparametric series
regression for financial time series data with heavy-tailed errors.

\section{Inference on possibly nonlinear functionals}

\label{inference sec}

We now study inference on possibly nonlinear functionals $f:L^{2}(X)\cap
L^{\infty }(X)\rightarrow \mathbb{R}$ of the regression function $h_{0}$.
Examples of functionals include, but are not limited to, the pointwise
evaluation functional, the partial mean functional, and consumer surplus
(see, e.g., \cite{Newey1997} for examples). The functional $f(h_{0})$ may be
estimated using the plug-in series LS estimator $f(\widehat{h})$, for which
we now establish feasible limit theory.

As with \cite{Newey1997} and \cite{ChenLiaoSun}, our results allow
researchers to perform inference on nonlinear functionals $f$ of $h_{0}$
without needing to know whether or not $f(h_{0})$ is regular (i.e., $\sqrt{n}
$-estimable). However, there is already a large literature on the $\sqrt{n}$%
-asymptotic normality and the consistent variance estimation for series
estimators of regular functionals of conditional mean functions with weakly
dependent data (see, e.g., \cite{ChenShen}, \cite{Chen2007}, \cite{LiRacine}%
). To save space and to illustrate the usefulness of our new sup-norm
convergence rate results, we focus on asymptotic normality of $f(\widehat{h}%
) $ and the corresponding sieve t statistic when the functional is irregular
(i.e., slower than $\sqrt{n}$-estimable) in this section.

We borrow some notation and definitions from \cite{ChenLiaoSun}. Denote the
pathwise derivative of $f$ at $h_{0}$ in the direction $v\in \mathcal{V}%
:=(L^{2}(X)-\{h_{0}\})$ by
\begin{equation}
\frac{\partial f(h_{0})}{\partial h}[v]:=\lim_{\tau \rightarrow 0^{+}}\frac{%
f(h_{0}+\tau v)}{\tau }
\end{equation}%
and assume it is linear. Let $v_{K}^{\ast }\in \mathcal{V}%
_{K}:=(B_{K,w}-\{h_{0,K}\})$ be the sieve Riesz representer of $\frac{%
\partial f(h_{0})}{\partial h}[\cdot ]$ on $V_{K}$, i.e. $v_{K}^{\ast }$ is
the unique element of $\mathcal{V}_{K}$ such that
\begin{equation}
\frac{\partial f(h_{0})}{\partial h}[v]=E[v_{K}^{\ast }(X_{i})v(X_{i})]\quad
\text{for all\quad }v\in \mathcal{V}_{K}.
\end{equation}%
It is straightforward to verify that
\begin{equation}
v_{K}^{\ast }(\cdot )=b_{w}^{K}(\cdot )^{\prime }\left(
E[b^{K}_w(X_{i})b^{K}_w(X_{i})^{\prime }]\right) ^{-1}\frac{\partial f(h_{0})%
}{\partial h}[b_{w}^{K}]=\widetilde{b}_{w}^{K}(\cdot )^{\prime }\frac{%
\partial f(h_{0})}{\partial h}[\widetilde{b}_{w}^{K}]
\end{equation}%
where $\frac{\partial f(h_{0})}{\partial h}[b_{w}^{K}]$ is understood to be
the vector formed by evaluating $\frac{\partial f(h_{0})}{\partial h}[\cdot
] $ at each element of $b_{w}^{K}(\cdot )$. Let $\Vert v_{K}^{\ast }\Vert
_{L^{2}(X)}^{2}=E[v_{K}^{\ast }(X_{i})^{2}]$. It is clear that $\Vert
v_{K}^{\ast }\Vert _{L^{2}(X)}^{2}=(\frac{\partial f(h_{0})}{\partial h}[%
\widetilde{b}_{w}^{K}])^{\prime }(\frac{\partial f(h_{0})}{\partial h}[%
\widetilde{b}_{w}^{K}])$.

Following \cite{ChenLiaoSun}, we say that $f$ is a regular (or $L^2$-norm
bounded) functional if $\Vert v_{K}^{\ast }\Vert _{L^{2}(X)}\nearrow \Vert
v^{\ast }\Vert _{L^{2}(X)}<\infty $ where $v^{\ast }\in \mathcal{V}$ is the
unique solution to
\begin{equation*}
\frac{\partial f(h_{0})}{\partial h}[v]=E[v^{\ast }(X_{i})v(X_{i})]\quad
\text{for all\quad }v\in \mathcal{V}.
\end{equation*}
We say that $f$ is an irregular (or $L^2$-norm unbounded) functional if $%
\Vert v_{K}^{\ast }\Vert _{L^{2}(X)}\nearrow +\infty $. Note that a
functional could be irregular but still sup-norm bounded (see Remark \ref%
{rmk-snb} below).

Given the martingale difference errors (Assumption \ref{resid reg}(i)), we
can define the sieve variance associated with $f(\widehat{h})$ as $%
V_{K}:=\Vert v_{K}^{\ast }\Vert _{sd}^{2}:=E[(\epsilon _{i}v_{K}^{\ast
}(X_{i}))^{2}]$. It is clear that%
\begin{eqnarray*}
V_{K} &=&\left( \frac{\partial f(h_{0})}{\partial h}[b_{w}^{K}]\right)
^{\prime }\left( E[b^{K}_w(X_{i})b^{K}_w(X_{i})^{\prime }]\right)
^{-1}E[\epsilon _{i}^{2}b^{K}_w(X_{i})b^{K}_w(X_{i})^{\prime }]\left(
E[b^{K}_w(X_{i})b^{K}_w(X_{i})^{\prime }]\right) ^{-1}\left( \frac{\partial
f(h_{0})}{\partial h}[b_{w}^{K}]\right) \\
&=&\left( \frac{\partial f(h_{0})}{\partial h}[\widetilde{b}_{w}^{K}]\right)
^{\prime }E[\epsilon _{i}^{2}\widetilde{b}_{w}^{K}(X_{i})\widetilde{b}%
_{w}^{K}(X_{i})^{\prime }]\left( \frac{\partial f(h_{0})}{\partial h}[%
\widetilde{b}_{w}^{K}]\right) .
\end{eqnarray*}%
The sieve variance $V_{K}=\Vert v_{K}^{\ast }\Vert _{sd}^{2}$ is estimated
with the simple plug-in estimator $\widehat{V}_{K}=\widehat{\Vert
v_{K}^{\ast }\Vert }_{sd}^{2}$, where
\begin{equation}
\begin{array}{rcl}
\widehat{\Vert v_{K}^{\ast }\Vert }_{sd}^{2} & = & \displaystyle\frac{1}{n}%
\sum_{i=1}^{n}\widehat{v}_{K}^{\ast }(X_{i})^{2}(Y_{i}-\widehat{h}%
(X_{i}))^{2} \\
\widehat{v}_{K}^{\ast }(X_{i}) & = & b_{w}^{K}(X_{i})^{\prime
}(B_{w}^{\prime }B_{w}/n)^{-}\frac{\partial f(\widehat{h})}{\partial h}%
[b_{w}^{K}]\,.%
\end{array}%
\end{equation}

We first introduce a slight variant of Assumption \ref{resid reg}(ii).

\setcounter{assumption}{1}

\begin{assumption}
(iv) $\inf_{x \in \mathcal{X}}E[\epsilon _{i}^{2}|X_i = x] > 0$, (v) $%
\sup_{x\in \mathcal{X}}E[\epsilon _{i}^{2}\{|\epsilon _{i}|>\ell
(n)\}|X_{i}=x]\rightarrow 0$ as $n\rightarrow \infty $ for any positive
sequence $\ell :\mathbb{N}\rightarrow \mathbb{R}_{+}$ with $\ell
(n)\rightarrow \infty $ as $n\rightarrow \infty $.
\end{assumption}

Assumption \ref{resid reg}(ii) and (iv) together imply that $\Vert
v_{K}^{\ast }\Vert _{L^{2}(X)}^{2}\asymp \Vert v_{K}^{\ast }\Vert
_{sd}^{2}=V_{K}$. Assumption \ref{resid reg}(v) is a standard uniform
integrability condition, which is not needed for the asymptotic normality of
$f(\widehat{h})$ with i.i.d. data when $f$ is a regular functional (see,
e.g., \cite{Chen2007})

Before we establish the asymptotic normality of $f(\widehat{h})$ under
general weak dependence, we need an additional assumption on the joint
dependence of $X_{i}$ and $\epsilon _{i}^{2}$, since this is not captured by
the martingale difference property of $\{\epsilon _{i}\}$ (Assumption \ref%
{resid reg}(i)). Define the $K\times K$ matrices
\begin{equation}
\begin{array}{rcl}
\widehat{\Omega } & = & n^{-1}\sum_{i=1}^{n}\epsilon _{i}^{2}\widetilde{b}%
_{w}^{K}(X_{i})\widetilde{b}_{w}^{K}(X_{i})^{\prime } \\
\Omega & = & E[\epsilon _{i}^{2}\widetilde{b}_{w}^{K}(X_{i})\widetilde{b}%
_{w}^{K}(X_{i})^{\prime }]%
\end{array}%
\end{equation}

\setcounter{assumption}{7}

\begin{assumption}
\label{a-eucgce} $\Vert \widehat{\Omega }-\Omega \Vert =o_{p}(1)$.
\end{assumption}

The following Lemma is a useful technical result that is again derived using
our new exponential inequality for sums of weakly dependent random matrices.

\begin{lemma}
\label{lem-omcgce} Let Assumptions \ref{X reg}(i), \ref{resid reg}(ii)(iii),
and \ref{sieve reg gen}(iii) hold. Then Assumption \ref{a-eucgce} is
satisfied provided that either (a), (b) or (c) is satisfied:

\begin{itemize}
\item[(a)] $\{(X_{i},Y_{i})\}_{i=1}^{n}$ is i.i.d. and $(\zeta _{K,n}\lambda
_{K,n})^{(2+\delta )/\delta }\sqrt{(\log K)/n}=o(1)$;

\item[(b)] $\{(X_{i},Y_{i})\}_{i=1}^{n}$ is exponentially $\beta $-mixing
and $(\zeta _{K,n}\lambda _{K,n})^{(2+\delta )/\delta }\sqrt{(\log n)^{2}/n}%
=o(1)$;

\item[(c)] $\{(X_{i},Y_{i})\}_{i=1}^{n}$ is algebraically $\beta $-mixing at
rate $\gamma $ and $(\zeta _{K,n}\lambda _{K,n})^{(2+\delta )/\delta }\sqrt{%
(\log K)/n^{\gamma /(1+\gamma )}}=o(1)$.
\end{itemize}
\end{lemma}

\subsection{Asymptotic normality of $f(\widehat{h})$ for general irregular
functionals}

Let $N_{K,n}$ denote a convex neighborhood of $h_{0}$ such that $\widehat{h},%
\widetilde{h}\in N_{K,n}$ wpa1. The appropriate neighborhood will depend on
the properties of the functional under consideration. For regular and
irregular functionals we can typically take $N_{K,n}$ to be of the form $%
N_{K,n}=\{h\in B_{K,w}:{\Vert h-h_{0}\Vert _{L^{2}(X)}}\leq (\sqrt{K/n}%
+\Vert h_{0}-h_{0,K}\Vert _{L^{2}(X)})\times \log \log n\}$. However, for
sup-norm bounded nonlinear functionals (see Remark \ref{rmk-snb}) it may
suffice to take $N_{K,n}=\{h\in B_{K,w}:\Vert h-h_{0}\Vert _{L^{\infty
}(X)}<\epsilon \}$ for some fixed $\epsilon >0$, or even $N_{K,n}=L^{\infty
}(X)\cap B_{K,w}$ for sup-norm bounded linear functionals. Our sup-norm and $%
L^{2}$ rate results are clearly useful in defining an appropriate
neighborhood.

We now introduce some primitive regularity conditions on the functional $f$.

\begin{assumption}
\label{a-functional} (i) $v\mapsto \frac{\partial f(h_{0})}{\partial h}[v]$
is a linear functional;

(ii) $\sup_{h\in N_{K,n}}\sqrt{n}\Vert v_{K}^{\ast }\Vert
_{L^{2}(X)}^{-1}\left\vert f(h)-f(h_{0})-\frac{\partial f(h_{0})}{\partial h}%
[h-h_{0}]\right\vert =o(1)$ where $\widehat{h},\widetilde{h}\in N_{K,n}$
wpa1;

(iii) $\Vert v_{K}^{\ast }\Vert _{L^{2}(X)}\nearrow +\infty $, $\sqrt{n}%
\Vert v_{K}^{\ast }\Vert _{L^{2}(X)}^{-1}\left\vert \frac{\partial f(h_{0})}{%
\partial h}[\widetilde{h}-h_{0}]\right\vert =o_{p}(1)$.
\end{assumption}

Assumption \ref{a-functional} corresponds to Assumption 3.1 in \cite%
{ChenLiaoSun} and Assumption 2.1 in \cite{ChenLiaoiid} for irregular
functionals. We refer the reader to these papers for a detailed discussion
and verification of Assumption \ref{a-functional}. Note that parts (i) and
(ii) of Assumption \ref{a-functional} are automatically satisfied when $f$
is a linear functional.

\begin{remark}
\label{rmk-snb} Certain linear and nonlinear functionals may be irregular
yet may still be bounded with respect to the sup norm. Alternative
sufficient conditions for Assumption \ref{a-functional} may be provided for
such functionals:

\begin{itemize}
\item[(a)] Suppose $f$ is a linear, irregular functional but that $f$ is {%
sup-norm bounded}, i.e. $|f(h)|\lesssim \Vert h\Vert _{\infty }$ (e.g. the
evaluation functional $f(h)=h(x)$ for some fixed $x\in \mathcal{X}$ is
sup-norm bounded because $|f(h)|=|h(x)|\leq \Vert h\Vert _{\infty }$). Then
a sufficient condition for Assumption \ref{a-functional} is
\begin{equation*}
\sqrt{n}V_{K}^{-1/2}\Vert \widetilde{h}-h_{0}\Vert _{\infty }\lesssim _{p}%
\sqrt{n}V_{K}^{-1/2}\Vert P_{K,w,n}\Vert _{\infty }\Vert h_{0}-h_{0,K}^{\ast
}\Vert _{\infty }=o_{p}(1).
\end{equation*}%
When $\Vert P_{K,w,n}\Vert _{\infty }\lesssim 1$ and $\Vert
h_{0}-h_{0,K}^{\ast }\Vert _{\infty }=O(K^{-p/d})$ then Assumption \ref%
{a-functional} is satisfied provided $\sqrt{n}V_{K}^{-1/2}K^{-p/d}=o(1)$.

\item[(b)] Suppose $f$ is a nonlinear, irregular functional whose derivative
is sup-norm bounded. Then Assumption \ref{a-functional} may be replaced with:%
\newline
(i') $v\mapsto \frac{\partial f(h_{0})}{\partial h}[v]$ is a linear
functional; \newline
(ii') ${\left\vert f(h)-f(h_{0})-\frac{\partial f(h_{0})}{\partial h}%
[h-h_{0}]\right\vert }\lesssim \Vert h-h_{0}\Vert _{\infty }^{2}$ uniformly
for $h\in N_{K,n}$;\newline
(iii') $\left\vert \frac{\partial f(h_{0})}{\partial h}[h-h_{0}]\right\vert
\lesssim \Vert h-h_{0}\Vert _{\infty }$ uniformly for $h\in N_{K,n}$; and
\newline
(iv') $\widehat{h},\widetilde{h}\in N_{K,n}$ wpa1, $\sqrt{n}\Vert
v_{K}^{\ast }\Vert _{L^{2}(X)}^{-1}\left( \Vert \widetilde{h}-h_{0}\Vert
_{\infty }+\Vert \widetilde{h}-h_{0}\Vert _{\infty }^{2}+\Vert \widehat{h}-%
\widetilde{h}\Vert _{\infty }^{2}\right) =o_{p}(1)$\newline
where $N_{K,n}=\{h\in B_{K,w}:\Vert h-h_{0}\Vert _{\infty }\leq \epsilon \}$
for some fixed $\epsilon >0$.\newline
For example, \cite{Newey1997} shows that conditions (i')(ii')(iii') are
satisfied for consumer surplus functionals in demand estimation.
\end{itemize}
\end{remark}

\begin{theorem}
\label{t-dist-new} Let Assumptions \ref{X reg}(i), \ref{resid reg}%
(i)(ii)(iv)(v), \ref{sieve reg gen}(iii), \ref{a-gram} and \ref{a-functional}
hold. Then
\begin{equation*}
\frac{\sqrt{n}(f(\widehat{h})-f(h_{0}))}{V_{K}^{1/2}}\rightarrow _{d}N(0,1)
\end{equation*}%
as $n,K\rightarrow \infty $ provided that either (a) or (b) is satisfied:

\begin{itemize}
\item[(a)] $\{(X_i,Y_i)\}_{i=1}^n$ is i.i.d.;

\item[(b)] $\{X_i\}_{i=1}^n$ is weakly dependent: Assumption \ref{a-eucgce}
holds, and ${\|\widetilde B_w ^{\prime }\widetilde B_w/n - I_K\|} =
o_p(K^{-1/2})$.
\end{itemize}
\end{theorem}

We now consider the special case of irregular but sup-norm bounded linear or
nonlinear functionals as discussed in Remark \ref{rmk-snb}. The sup-norm
convergence rates for series LS estimators in Section \ref{main sec} are
employed to derive asymptotic normality of plug-in estimators of such
functionals under weak conditions.
To save space, for the weakly dependent case we only present sufficient
conditions for asymptotic normality of $f(\widehat{h})$ when the regression
error has no more than a finite 4th absolute moment (i.e., $E[|\epsilon
_{i}|^{2+\delta }]<\infty $ for some $0\leq \delta \leq 2$). We also take $%
\mathcal{X }= [0,1]^d$ and $w_{n}=1$ for all $n$ for simplicity.

\begin{corollary}
\label{inference-2} Let $f$ be an irregular but sup-norm bounded linear
functional, and let Assumptions \ref{X reg} (with $\mathcal{X}=[0,1]^{d}$), %
\ref{resid reg}(i)(ii)(iv)(v), \ref{parameter regression}, and \ref{b sieve}
hold. Then
\begin{equation*}
\frac{\sqrt{n}(f(\widehat{h})-f(h_{0}))}{V_{K}^{1/2}}\rightarrow _{d}N(0,1)
\end{equation*}%
as $n,K \to \infty$ provided that either (a), (b) or (c) is satisfied:

\begin{itemize}
\item[(a)] $\{(X_{i},Y_{i})\}_{i=1}^{n}$ is i.i.d.: $\sqrt{n}%
V_{K}^{-1/2}K^{-p/d}=o(1)$ and $(K\log K)/n=o(1)$;

\item[(b)] $\{(X_{i},Y_{i})\}_{i=1}^{n}$ is exponentially $\beta $-mixing:
Assumption \ref{resid reg}(iii) also holds, $\sqrt{n}%
V_{K}^{-1/2}K^{-p/d}=o(1)$, and $K^{(2+\delta )/\delta }(\log n)^{2}/n=o(1)$
with $\delta \leq 2$;

\item[(c)] $\{(X_{i},Y_{i})\}_{i=1}^{n}$ is algebraically $\beta $-mixing at
rate $\gamma $: Assumption \ref{resid reg}(iii) also holds, $\sqrt{n}%
V_{K}^{-1/2}K^{-p/d}=o(1)$, and $K^{(2+\delta )/\delta }(\log K)/n^{\gamma
/(1+\gamma )}=o(1)$ with $\delta \leq 2$.
\end{itemize}
\end{corollary}

Corollary \ref{inference-2} part (a) extends the weakest known result on
pointwise asymptotic normality of spline LS estimators in \cite{Huang2003}
to general sup-norm bounded linear functionals of spline or wavelet series
LS estimators.\footnote{%
Under the assumption of empirical identifiability (see equation (\ref{sm5}))
and other conditions similar to the ones listed in Corollary \ref%
{inference-2} part (a), \cite{ChenHuang03} derived the asymptotic normality
of plug-in spline LS estimators of sup-norm bounded linear functionals (see
their Theorem 4).}

\begin{corollary}
\label{inference-nl} Let $f$ be an irregular but sup-norm bounded nonlinear
functional, and let Assumptions \ref{X reg} (with $\mathcal{X}=[0,1]^{d} $), %
\ref{resid reg}, \ref{parameter regression}, \ref{b sieve} and \ref%
{a-functional}(i')(ii')(iii') hold. Then
\begin{equation*}
\frac{\sqrt{n}(f(\widehat{h})-f(h_{0}))}{V_{K}^{1/2}}\rightarrow _{d}N(0,1)
\end{equation*}%
as $n,K \to \infty$ provided that either (a), (b) or (c) is satisfied:

\begin{itemize}
\item[(a)] $\{(X_{i},Y_{i})\}_{i=1}^{n}$ is i.i.d.: $\sqrt{n}%
V_{K}^{-1/2}K^{-p/d}=o(1)$ and $K^{(2+\delta )/\delta }(\log n)/n \lesssim 1$
with $\delta < 2$;

\item[(b)] $\{(X_{i},Y_{i})\}_{i=1}^{n}$ is exponentially $\beta $-mixing: $%
\sqrt{n}V_{K}^{-1/2}K^{-p/d}=o(1)$ and $K^{(2+\delta )/\delta }(\log
n)^{2}/n=o(1)$ with $\delta \leq 2$;

\item[(c)] $\{(X_{i},Y_{i})\}_{i=1}^{n}$ is algebraically $\beta $-mixing at
rate $\gamma $: $\sqrt{n}V_{K}^{-1/2}K^{-p/d}=o(1)$ and $K^{(2+\delta
)/\delta }(\log K)/n^{\gamma /(1+\gamma )}=o(1)$ with $\delta \leq 2$.
\end{itemize}
\end{corollary}

Conditions for weakly dependent data in Corollary \ref{inference-nl} parts
(b) and (c) are natural extensions of those in part (a) for i.i.d. data,
which in turn are much weaker than the well-known conditions in \cite%
{Newey1997} for the asymptotic normality of nonlinear functionals of spline
LS estimators, namely $\sqrt{n}K^{-p/d}=o(1)$, $K^{4}/n=o(1)$ and $%
\sup_{x}E[\left\vert \epsilon _{i}\right\vert ^{4}|X_{i}=x]<\infty $.%

\subsection{Asymptotic normality of sieve t statistics for general
functionals}

We now turn to the consistent estimation of $V_{K}=\Vert v_{K}^{\ast }\Vert
_{sd}^{2}$ and feasible asymptotic inference for $f(h_{0})$.

\begin{assumption}
\label{t-functional} $\Vert v_{K}^{\ast }\Vert _{L^{2}(X)}^{-1}\left\Vert
\frac{\partial f(h)}{\partial h}[\widetilde{b}_{w}^{K}]-\frac{\partial
f(h_{0})}{\partial h}[\widetilde{b}_{w}^{K}]\right\Vert =o(1)$ uniformly
over $h\in N_{K,n}$ or $B_{\epsilon ,\infty }(h_{0})$.
\end{assumption}

Note that Assumption \ref{t-functional} is automatically satisfied when $f$
is a linear functional. It is only required to establish consistency of $%
\widehat{\Vert v_{K}^{\ast }\Vert }_{sd}$ for a nonlinear functional, and
corresponds to Assumption 3.1(iii) of \cite{ChenLiaoiid}.

The first part of the following Lemma establishes the consistency of the
sieve variance estimator under both the i.i.d. and general weakly dependent
data.

\begin{lemma}
\label{lem-varconsistent} Let Assumptions \ref{X reg}(i)(ii), \ref{resid reg}%
(i)(ii)(iv), \ref{sieve reg gen}(iii), \ref{a-gram}, \ref{a-eucgce} and \ref%
{t-functional} hold and $\Vert \widehat{h}-h_{0}\Vert _{\infty ,w}=o_{p}(1)$%
. Then:

(1) $\left\vert \frac{\widehat{\Vert v_{K}^{\ast }\Vert }_{sd}}{{\Vert
v_{K}^{\ast }\Vert }_{sd}}-1\right\vert =o_{p}(1)$ as $n,K\rightarrow \infty
$.

(2) Further, if Assumptions \ref{resid reg}(v) and \ref{a-functional} hold,
then:
\begin{equation*}
\frac{\sqrt{n}(f(\widehat{h})-f(h_{0}))}{\widehat{V}_{K}^{1/2}}\rightarrow
_{d}N(0,1)
\end{equation*}%
as $n,K\rightarrow \infty $ provided that either (a) $\{(X_{i},Y_{i})%
\}_{i=1}^{n}$ is i.i.d., or (b) $\{X_{i}\}_{i=1}^{n}$ is weakly dependent
with ${\Vert \widetilde{B}_{w}^{\prime }\widetilde{B}_{w}/n-I_{K}\Vert }%
=o_{p}(K^{-1/2})$ is satisfied.
\end{lemma}

Lemma \ref{lem-varconsistent} can be combined with different sufficient
conditions for Assumptions \ref{a-gram}, \ref{a-eucgce} and \ref%
{a-functional} to yield different special cases of the asymptotic normality
of sieve t statistics for general (possibly) nonlinear functionals. We state
three special cases below. The following Theorem is applicable to series LS
estimators with an arbitrary basis.

\begin{theorem}
\label{c-inference} Let Assumptions \ref{X reg}(i)(ii), \ref{resid reg}, \ref%
{a-wdelta}, \ref{sieve reg gen}, \ref{a-functional} and \ref{t-functional}
hold and $\Vert \widetilde{h}-h_{0}\Vert _{\infty }=o_{p}(1)$. Then
\begin{equation*}
\frac{\sqrt{n}(f(\widehat{h})-f(h_{0}))}{\widehat{V}_{K}^{1/2}}\rightarrow
_{d}N(0,1)
\end{equation*}%
as $n,K\rightarrow \infty $ provided that either (a), (b) or (c) is
satisfied:

\begin{itemize}
\item[(a)] $\{(X_{i},Y_{i})\}_{i=1}^{n}$ is i.i.d.: $(\zeta _{K,n}\lambda
_{K,n})^{(2+\delta )/\delta }\sqrt{(\log n)/n}=o(1)$;

\item[(b)] $\{(X_{i},Y_{i})\}_{i=1}^{n}$ is exponentially $\beta $-mixing: $%
\max (\sqrt{K},(\zeta _{K,n}\lambda _{K,n})^{2/\delta })\times (\zeta
_{K,n}\lambda _{K,n})\sqrt{\frac{(\log n)^{2}}{n}}=o(1)$;

\item[(c)] $\{(X_{i},Y_{i})\}_{i=1}^{n}$ is algebraically $\beta $-mixing at
rate $\gamma $: $\max (\sqrt{K},(\zeta _{K,n}\lambda _{K,n})^{2/\delta
})\times (\zeta _{K,n}\lambda _{K,n})\sqrt{\frac{\log n}{n^{\gamma
/(1+\gamma )}}}=o(1)$.
\end{itemize}
\end{theorem}

The following Corollaries are direct consequences of Theorem \ref%
{c-inference} for linear and nonlinear sup-norm bounded functionals, with
spline or wavelet bases. For simplicity, we take $w_{n}=1$ for all $n$ and $%
\mathcal{X}=[0,1]^{d}$.

\begin{corollary}
\label{c-inference-linear} Let Assumptions \ref{X reg} (with $\mathcal{X}%
=[0,1]^{d}$), \ref{resid reg}, \ref{parameter regression}, and \ref{b sieve}
hold for a sup-norm bounded linear functional. Then
\begin{equation*}
\frac{\sqrt{n}(f(\widehat{h})-f(h_{0}))}{\widehat{V}_{K}^{1/2}}\rightarrow
_{d}N(0,1)
\end{equation*}%
as $n,K\rightarrow \infty $ provided that either (a), (b) or (c) is
satisfied:

\begin{itemize}
\item[(a)] $\{(X_{i},Y_{i})\}_{i=1}^{n}$ is i.i.d.: $\sqrt{n}%
V_{K}^{-1/2}K^{-p/d}=o(1)$ and $K^{(2+\delta )/\delta }(\log n)/n=o(1)$;

\item[(b)] part (b) of Corollary \ref{inference-2};

\item[(c)] part (c) of Corollary \ref{inference-2}.
\end{itemize}
\end{corollary}

\begin{corollary}
\label{c-inference-infty} Let Assumptions \ref{X reg} (with $\mathcal{X}%
=[0,1]^{d}$), \ref{resid reg}, \ref{parameter regression}, \ref{b sieve}, %
\ref{a-functional}(i')(ii')(iii') and \ref{t-functional} hold for a
nonlinear functional. Then
\begin{equation*}
\frac{\sqrt{n}(f(\widehat{h})-f(h_{0}))}{\widehat{V}_{K}^{1/2}}\rightarrow
_{d}N(0,1)
\end{equation*}%
as $n,K\rightarrow \infty $ provided that either (a), (b) or (c) is
satisfied:

\begin{itemize}
\item[(a)]
$\{(X_{i},Y_{i})\}_{i=1}^{n}$ is i.i.d.: $\sqrt{n}V_{K}^{-1/2}K^{-p/d}=o(1)$
and $K^{(2+\delta )/\delta }(\log n)/n=o(1)$ with $\delta < 2$;

\item[(b)] part (b) of Corollary \ref{inference-nl};

\item[(c)] part (c) of Corollary \ref{inference-nl}.
\end{itemize}
\end{corollary}

Previously, \cite{Newey1997} required that $\sup_{x}E[\epsilon
_{i}^{4}|X_{i}=x]<\infty $ and $K^{4}/n=o(1)$ in order to establish
asymptotic normality of student $t$ statistics for nonlinear functionals
with i.i.d. data. Our sufficient conditions are weaker and allow for weakly
dependent data with heavy-tailed errors.

\section{Useful results on random matrices}

\label{ei sec}

\subsection{An exponential inequality for sums of weakly dependent random
matrices}

\label{ineq sec}

In this section we derive a new Bernstein-type inequality for sums of random
matrices formed from absolutely regular ($\beta $-mixing) sequences, where
the dimension, norm, and variance measure of the random matrices are allowed
to grow with the sample size. This inequality is particularly useful for
establishing sharp convergence rates for semi/nonparametric sieve estimators
with weakly dependent data. We first recall an inequality of \cite{Tropp2012}
for independent random matrices.

\begin{theorem}[\protect\cite{Tropp2012}]
\label{troppthm} Let $\{\Xi _i\}_{i=1}^n$ be a finite sequence of
independent random matrices with dimensions $d_1\times d_2$. Assume $E[\Xi
_i]=0$ for each $i$ and $\max _{1\leq i\leq n}\|\Xi _i\|\leq R_n$, and
define
\begin{equation}
\sigma _n^2=\max \left \{\left \|\sum _{i=1}^nE[\Xi _i\Xi _i^{\prime
}]\right \|,\left \|\sum _{i=1}^nE[\Xi _i^{\prime }\Xi _i]\right \|\right
\}\,.  \notag
\end{equation}
Then for all $t\geq 0$,
\begin{equation}
{\mathbb{P}}\left (\left \|\sum _{i=1}^n\Xi _i\right \|\geq t\right )\leq
(d_1+d_2) \mathrm{exp}\left (\frac{-t^2/2}{\sigma _n^2+R_nt/3}\right )\,.
\notag
\end{equation}
\end{theorem}

\begin{corollary}
\label{troppcor} Under the conditions of Theorem \ref{troppthm}, if $R_{n}%
\sqrt{\log (d_{1}+d_{2})}=o(\sigma _{n})$ then
\begin{equation}
\left\Vert \sum_{i=1}^{n}\Xi _{i,n}\right\Vert =O_{p}(\sigma _{n}\sqrt{\log
(d_{1}+d_{2})})\,.  \notag
\end{equation}
\end{corollary}

When $\{X_{i}\}_{i=-\infty }^{\infty }$ is i.i.d., Corollary \ref{troppcor}
is used to provide weak low-level sufficient conditions under which $\Vert
\widetilde{B}_{w}^{\prime }\widetilde{B}_{w}/n-I_{K}\Vert =o_{p}(1)$ holds
(see Lemma \ref{Bconvi.i.d.}).

We now provide an extension of Theorem \ref{troppthm} and Corollary \ref%
{troppcor} for matrix-valued functions of $\beta $-mixing sequences. The $%
\beta $-mixing coefficient between two $\sigma $-algebras ${\mathcal{A}}$
and ${\mathcal{B}}$ is defined as
\begin{equation}
\beta ({\mathcal{A}},{\mathcal{B}})=\frac{1}{2}\sup \sum_{(i,j)\in I\times
J}|{\mathbb{P}}(A_{i}\cap B_{j})-{\mathbb{P}}(A_{i}){\mathbb{P}}(B_{j})|
\end{equation}%
with the supremum taken over all finite partitions $\{A_{i}\}_{i\in
I}\subset {\mathcal{A}}$ and $\{B_{j}\}_{j\in J}\subset {\mathcal{B}}$ of $%
\Omega $ (see, e.g., \cite{Bradley2005}). The $q$th $\beta $-mixing
coefficient of $\{X_{i}\}_{i=-\infty }^{\infty }$ is defined as
\begin{equation}
\beta (q)=\sup_{i}\beta (\sigma (\ldots ,X_{i-1},X_{i}),\sigma
(X_{i+q},X_{i+q+1},\ldots ))\,.
\end{equation}%
The process $\{X_{i}\}_{i=-\infty }^{\infty }$ is said to be \emph{%
algebraically $\beta $-mixing} at rate $\gamma $ if $q^{\gamma }\beta
(q)=o(1)$ for some $\gamma >1$, and \emph{exponentially $\beta $-mixing} if $%
\beta (q)\leq c\mathrm{exp}(-\gamma q)$ for some $\gamma >0$ and $c\geq 0$.
The following extension of Theorem \ref{troppthm} is made using Berbee's
Lemma and a coupling argument.

\begin{theorem}
\label{beta tropp} Let $\{X_i\}_{i=-\infty }^{\infty }$ be a $\beta $-mixing
sequence and let $\Xi _{i,n}=\Xi _n(X_i)$ for each $i$ where $\Xi _n:{%
\mathcal{X}}\to {\mathbb{R}}^{d_1\times d_2}$ is a sequence of measurable $%
d_1\times d_2$ matrix-valued functions. Assume $E[\Xi _{i,n}]=0$ and $\|\Xi
_{i,n}\|\leq R_n$ for each $i$ and define $s_n^2=\max _{1\leq i,j\leq n}\max
\{\|E[\Xi _{i,n}\Xi _{j,n}^{\prime }]\|,\|E[\Xi _{i,n}^{\prime }\Xi
_{j,n}]\|\}$. Let $q$ be an integer between $1$ and $n/2$ and let $%
I_r=q[n/q]+1,\ldots ,n$ when $q[n/q]<n$ and $I_r=\emptyset $ when $q[n/q]=n$%
. Then for all $t\geq 0$,
\begin{equation}
{\mathbb{P}}\left (\left \|\sum _{i=1}^n\Xi _{i,n}\right \|\geq 6t\right
)\leq \frac{n}{q}\beta (q)+{\mathbb{P}}\left (\left \|\sum _{i\in I_r}\Xi
_{i,n}\right \|\geq t\right )+2(d_1+d_2) \mathrm{exp}\left (\frac{-t^2/2}{%
nqs_n^2+qR_nt/3}\right )  \notag
\end{equation}
(where $\|\sum _{i\in I_r}\Xi _{i,n}\|:=0$ whenever $I_r=\emptyset $).
\end{theorem}

\begin{corollary}
\label{beta rate} Under the conditions of Theorem \ref{beta tropp}, if $%
q=q(n)$ is chosen s.t. $\frac{n}{q}\beta (q)=o(1)$ and $R_{n}\sqrt{q\log
(d_{1}+d_{2})}=o(s_{n}\sqrt{n})$ then
\begin{equation}
\left\Vert \sum_{i=1}^{n}\Xi _{i,n}\right\Vert =O_{p}(s_{n}\sqrt{nq\log
(d_{1}+d_{2})})\,.  \notag
\end{equation}
\end{corollary}

When the regressors $\{X_{i}\}_{i=-\infty }^{\infty }$ are $\beta $-mixing,
Corollary \ref{beta rate} is used to provide weak low-level sufficient
conditions under which $\Vert \widetilde{B}_{w}^{\prime }\widetilde{B}%
_{w}/n-I_{K}\Vert =o_{p}(1)$ holds (see Lemma \ref{Bconvbeta}).

We note that both Theorem \ref{beta tropp} and Corollary \ref{beta rate}
allow for non-identically distributed beta-mixing sequences. So the
convergence rate and the inference results in previous sections could be
extended to non-identically distributed regressors $\{X_{i}\}_{i=-\infty
}^{\infty }$ as well, except that notation and regularity conditions will be
slightly more complicated.

\subsection{Empirical identifiability}

\label{ei sec-1}

We now provide a readily verifiable condition under which the theoretical
and empirical $L^{2}$ norms are equivalent over a (weighted) linear sieve
space wpa1. This equivalence, referred to by \cite{Huang2003} as \emph{%
empirical identifiability}, has several applications in nonparametric sieve
estimation. In nonparametric series LS estimation, empirical identifiability
ensures that the estimator is the orthogonal projection of $Y$ onto the
linear sieve space under the empirical inner product and is uniquely defined
wpa1 \citep{Huang2003}. Empirical identifiability is also used to establish
the large-sample properties of sieve conditional moment estimators (see,
e.g., \cite{ChenPouzo2012}). A sufficient condition for empirical
identifiability is now cast in terms of convergence of a random matrix,
which we verify for i.i.d. and $\beta $-mixing sequences.

Recall that $L^{2}(X)$ denotes the space of functions $f:{\mathcal{X}}%
\rightarrow {\mathbb{R}}$ such that $E[f(X_{i})^{2}]<\infty $. A (linear)
subspace ${\mathcal{A}}\subseteq L^{2}(X)$ is said to be \textit{empirically
identifiable} if $\frac{1}{n}\sum_{i=1}^{n}b(X_{i})^{2}=0$ implies $b=0$. A
sequence of spaces $\{{\mathcal{A}}_{K}:K\geq 1\}\subseteq L^{2}(X)$ is
empirically identifiable wpa1 as $K=K(n)\rightarrow \infty $ with $n$ if
\begin{equation}
\lim_{n\rightarrow \infty }{\mathbb{P}}\left( \sup_{a\in A_{K}}\left\vert
\frac{\frac{1}{n}\sum_{i=1}^{n}a(X_{i})^{2}-E[a(X_{i})^{2}]}{E[a(X_{i})^{2}]}%
\right\vert >t\right) =0  \label{sm5}
\end{equation}%
for any $t>0$. \cite{Huang1998} verifies (\ref{sm5}) for i.i.d. data using a
chaining argument. \cite{ChenPouzo2012} use this result to establish
convergence of sieve conditional moment estimators. However, it may be
difficult to verify (\ref{sm5}) via chaining arguments for certain types of
weakly dependent sequences.

To this end, the following is a readily verifiable sufficient condition for
empirical identifiability for (weighted) linear sieve spaces given by $%
B_{K,w} = clsp\{ b_{K1} w_n,\ldots,b_{KK}w_n\}$.

\begin{condition}
\label{ei cond} $\lambda _{\min }(E[b^K_w(X_i)b^K_w(X_i)^{\prime }])>0$ for
each $K$ and $\|\widetilde {B}^{\prime }_w\widetilde {B}_w/n-I_K\|=o_p(1)$.
\end{condition}

\begin{lemma}
\label{eilem} If $\lambda _{\min }(E[b^K_w(X_i)b^K_w(X_i)^{\prime }])>0$ for
each $K$ then
\begin{equation}
\sup _{b\in B_{K,w}}\left |\frac{\frac{1}{n}\sum _{i=1}^nb(X_i)^2-E[b(X_i)^2]%
}{E[b(X_i)^2]}\right |=\|\widetilde {B}^{\prime }_w\widetilde {B}%
_w/n-I_K\|^2\,.  \notag
\end{equation}
\end{lemma}

\begin{corollary}
If Condition \ref{ei cond} holds then $B_{K,w}$ is empirically identifiable
wpa1.
\end{corollary}

Condition \ref{ei cond} is therefore a sufficient condition for (\ref{sm5})
to hold for the linear sieve space $B_{K,w}$.

\begin{remark}
Consider the compact support case in which $\mathcal{X}=[0,1]^{d}$ and $%
w_{n}(x)=1$ for all $x\in \mathcal{X}$ and all $n$ (so that $%
b_{Kk}(x)w_{n}(x)=b_{Kk}(x)$ for all $n$ and $K$) and suppose the density of
$X_{i}$ is uniformly bounded away from zero and infinity over $\mathcal{X}$.
\textbf{(1)} For i.i.d. regressors (and $\lambda _{K,n}\lesssim 1$),
previously \cite{Huang1998} establishes equivalence of the theoretical and
empirical $L^{2}$ norms over the sieve space via a chaining argument with $%
\zeta _{K,n}^{2}K/n=o(1)$. \cite{Huang2003} relaxes this to $K(\log
n)/n=o(1) $ for a polynomial spline basis. Our Lemma \ref{Bconvi.i.d.} shows
that, in fact, $\zeta _{K,n}\sqrt{(\log K)/n}=o(1)$ is sufficient with an
arbitrary linear sieve (provided $\lambda _{K,n}\lesssim 1$). \textbf{(2) }%
For strictly stationary beta-mixing regressors (and $\lambda _{K,n}\lesssim
1 $), Lemma \ref{Bconvbeta} shows the equivalence of the theoretical and
empirical $L^{2}$ norms over any linear sieve space under either $\zeta
_{K,n}\sqrt{(\log n)^{2}/n}=o(1)$ for exponential beta-mixing, or $\zeta
_{K,n}\sqrt{(\log K)/n^{\gamma /(1+\gamma )}}=o(1)$ for algebraic
beta-mixing.
\end{remark}

\section{Sup-norm stability of $L^{2}(X)$ projection onto wavelet sieves}

\label{sec-stable}

In this section we show that the $L^{2}(X)$ orthogonal projection onto
(tensor product) compactly supported wavelet bases is stable in sup norm as
the dimension of the space increases. Consider the orthogonal projection
operator $P_{K}$ defined in expression (\ref{e-pkdef}) where the elements of $b^{K}$ span the tensor products of $d$ univariate
wavelet spaces $\mathrm{Wav}(K_{0},[0,1])$. We show that its $L^\infty$
operator norm $\|P_K\|_\infty$ (see expression (\ref{e-pksupnorm})) is
stable, in the sense that $\|P_K\|_\infty \lesssim 1$ as $K \to \infty$.
We also show that the empirical $L^{2}$ projection $P_{K,n}$ onto the
wavelet sieve is stable in sup norm wpa1. This result is used to establish
that series LS estimators with (tensor-product) wavelet bases attain their
optimal sup-norm rates. A variant of this result for projections arising in
series two-stage LS was used in an antecedent of this paper %
\citep{ChenChristensenNPIVold} but its proof was omitted for brevity.

The following Theorem presents our result for the stability of the
projection with respect to the $L^{2}(X)$ inner product.

\begin{theorem}
\label{t-wavd} Let $\mathcal{X}\supseteq \lbrack 0,1]^{d}$ and let the
density $f_{X}$ of $X_{i}$ be such that $0<\inf_{x\in \lbrack
0,1]^{d}}f_{X}(x)\leq \sup_{x\in \lbrack 0,1]^{d}}f_{X}(x)<\infty $. Let $%
B_{K}$ be the tensor product of $d$ univariate wavelet spaces $\mathrm{Wav}%
(K_{0},[0,1])$ where $\mathrm{Wav}(K_{0},[0,1])$ is as described in Section %
\ref{sieve def} and $K=2^{dJ}$ and $K_{0}=2^{J}>2N$. Then: $\Vert P_{K}\Vert
_{\infty }\lesssim 1$.
\end{theorem}

We now present conditions under which the empirical projection onto a
tensor-product wavelet basis is stable wpa1. Here the projection operator is
\begin{equation*}
P_{K,n}h(x)=b^{K}(x)^{\prime }\left( \frac{B^{\prime }B}{n}\right) ^{-}\frac{%
1}{n}\sum_{i=1}^{n}b^{K}(X_{i})h(X_{i})
\end{equation*}%
where the elements of $b^{K}$ span the tensor products of $d$ univariate
spaces $\mathrm{Wav}(K_{0},[0,1])$. The following Theorem states simple
sufficient conditions for $\Vert P_{K,n}\Vert _{\infty }\lesssim 1$ wpa1.%

\begin{theorem}
\label{c-pstable} Let conditions stated in Theorem \ref{t-wavd} hold. Then $%
\Vert P_{K,n}\Vert _{\infty }\lesssim 1$ wpa1 provided that either (a), (b),
or (c) is satisfied:

\begin{itemize}
\item[(a)] $\{X_i\}_{i=1}^n$ are i.i.d. and $\sqrt{(K\log n)/n}=o(1)$

\item[(b)] $\{X_i\}_{i=1}^n$ are exponentially $\beta $-mixing and $\sqrt{%
K(\log n)^{2}/n}=o(1)$, or

\item[(c)] $\{X_i\}_{i=1}^n$ are algebraically $\beta $-mixing at rate $%
\gamma $ and $\sqrt{(K\log n)/n^{\gamma /(1+\gamma )}}=o(1)$.
\end{itemize}
\end{theorem}

\section{Brief review of B-spline and wavelet sieve spaces}

\label{sieve def}

We first outline univariate B-spline and wavelet sieve spaces on $[0,1]$,
then deal with the multivariate case by constructing a tensor-product sieve
basis.

\paragraph{B-splines}

B-splines are defined by their order $r\geq 1$ (or degree $r-1 \geq 0$) and
number of interior knots $m\geq 0$. Define the knot set
\begin{equation}
0=t_{-(r-1)}=\ldots =t_{0}\leq t_{1}\leq \ldots \leq t_{m}\leq
t_{m+1}=\ldots =t_{m+r}=1\,.
\end{equation}%
We generate a $L^\infty$-normalized B-spline basis recursively using the De
Boor relation (see, e.g., Chapter 5 of \cite{DeVoreLorentz}) then
appropriately rescale the basis functions. Define the interior intervals $%
I_{1}=[t_{0},t_{1}),\ldots ,I_{m}=[t_{m},t_{m+1}]$ and generate a basis of
order $1$ by setting
\begin{equation}
N_{j,1}(x)=1_{I_{j}}(x)
\end{equation}%
for $j=0,\ldots m$, where $1_{I_{j}}(x)=1$ if $x\in I_{j}$ and $%
1_{I_{j}}(x)=0$ otherwise. Bases of order $r>1$ are generated recursively
according to
\begin{equation}
N_{j,r}(x)=\frac{x-t_{j}}{t_{j+r-1}-t_{j}}N_{j,r-1}(x)+\frac{t_{j+r}-x}{%
t_{j+r}-t_{j+1}}N_{j+1,r-1}(x)
\end{equation}%
for $j=-(r-1),\ldots ,m$ where we adopt the convention $\frac{1}{0}:=0$.
Finally, we rescale the basis by multiplying each $N_{j,r}$ by $(m+r)^{1/2}$
for $j = -(r-1),\ldots,m$. This results in a total of $K=m+r$ splines of
order $r$. Each spline is a polynomial of degree $r-1$ on each interior
interval $I_{1},\ldots ,I_{m}$ and is $(r-2)$-times continuously
differentiable on $(0,1)$ whenever $r> 2 $. The mesh ratio is defined as
\begin{equation}
\mbox {mesh}(K)=\frac{\max_{0\leq j\leq m}(t_{j+1}-t_{j})}{\min_{0\leq j\leq
m}(t_{j+1}-t_{j})}\,.
\end{equation}%
We let the space $\mbox {BSpl}(K,[0,1])$ be the closed linear span of these $%
K=m+r$ splines. The space $\mbox {BSpl}(K,[0,1])$ has uniformly bounded mesh
ratio if $\mbox {mesh}(K)\leq \kappa $ for all $N\geq 0$ and some $\kappa
\in (0,\infty )$. We let $\mbox {BSpl}(K,[0,1],\gamma)$ denote the space $%
\mbox {BSpl}(K,[0,1])$ with degree $\gamma$ and uniformly bounded mesh
ratio. See \cite{deBoor2001} and \cite{Schumacker2007} for further details.

\paragraph{Wavelets}

We construct a wavelet basis with support $[0,1]$ following \cite{CDV1993}.
Let $(\varphi,\psi)$ be a Daubechies pair such that $\varphi$ has support $%
[-N+1,N]$. Given $j$ such that $2^j -2N > 0$, the orthonormal (with respect
to the $L^2([0,1])$ inner product) basis for the space $V_j$ consists of $%
2^j-2N$ interior scaling functions of the form $\varphi_{j,k}(x) = 2^{j/2}
\varphi(2^j x - k)$, each of which has support $[2^{-j}(-N+1+k),2^{-j}(N+k)]$
for $k = N,\ldots,2^j-N-1$. These are augmented with $N$ left scaling
functions of the form $\varphi^0_{j,k}(x) = 2^{j/2}\varphi_k^l(2^j x)$ for $%
k = 0,\ldots,N-1$ (where $\varphi^l_0,\ldots,\varphi^l_{N-1}$ are fixed
independent of $j$), each of which has support $[0,2^{-j}(N+k)]$, and $N$
right scaling functions of the form $\varphi_{j,2^j-k}(x) = 2^{j/2}
\varphi^r_{-k}(2^j(x-1))$ for $k = 1,\ldots,N$ (where $\varphi^r_{-1},%
\ldots,\varphi^r_{-N}$ are fixed independent of $j$), each of which has
support $[1-2^{-j}(1-N-k),1]$. The resulting $2^j$ functions $%
\varphi^0_{j,0},\ldots,\varphi^0_{j,N-1},\varphi_{j,N},\ldots,%
\varphi_{j,2^j-N-1},\varphi^1_{j,2^j-N},\ldots,\varphi^1_{j,2^j-1}$ form an
orthonormal basis (with respect to the $L^2([0,1])$ inner product) for the
subspace they span, denoted $V_j$.

An orthonormal wavelet basis for the space $W_j$, defined as the orthogonal
complement of $V_j$ in $V_{j+1}$, is similarly constructed form the mother
wavelet. This results in an orthonormal basis of $2^j$ functions $%
\psi^0_{j,0},\ldots,\psi^0_{j,N-1},\psi_{j,N},\ldots,\psi_{j,2^j-N-1},%
\psi^1_{j,2^j-N},\ldots,\psi^1_{j,2^j-1}$. To simplify notation we ignore
the $0$ and $1$ superscripts on the left and right wavelets and scaling
functions henceforth.

Let $J_0$ and $J$ be integers such that $2^{J_0} \leq 2^J < 2N$. A wavelet
space at resolution level $J$ is the set of $2^J$ functions given by
\begin{equation}
\mathrm{Wav}(J) = \left\{ \sum_{k=0}^{2^{J_0}-1} a_{J_0,k} \varphi_{J_0,k} +
\sum_{j=J_0}^J \sum_{k=0}^{2^j-1} b_{j,k} \psi_{j,k} : a_{J_0,k}, b_{j,k}
\in \mathbb{R }\right\} \,.
\end{equation}
The spaces $V_j$ and $W_j$ are constructed so that $V_{j+1} = V_j \oplus W_j$
for all $j$ with $2^j-2N > 0$. Therefore, we can reexpress $\mathrm{Wav}(J)$
as
\begin{equation}
\mathrm{Wav}(J) = \left\{ \sum_{k=0}^{2^{J}-1} a_{J,k} \varphi_{J,k} :
a_{J,k} \in \mathbb{R }\right\} \,.
\end{equation}
The orthogonal projection onto $\mathrm{Wav}(J)$ is therefore the same,
irrespective of whether we use the bases for $V_J$ or $V_{J_0} \oplus
W_{J_0} \oplus \ldots \oplus W_J$. Note that, by the support of the $%
\varphi_{J,0},\ldots,\varphi_{J,2^J-1}$, the support of at most $2N-1$ basis
functions overlaps on a set of positive Lebesgue measure. We use this local
support to bound the orthogonal projection operator onto (tensor product)
wavelet bases in Section \ref{sec-stable}.

We say that $\mbox {Wav}(K,[0,1])$ has \emph{regularity} $\gamma $ if $N
\geq \gamma$, and write $\mbox {Wav}(K,[0,1],\gamma )$ for a wavelet space
of regularity $\gamma$ with continuously differentiable basis functions. This particular wavelet basis has been used in \cite{ChenReiss}, \cite{Kato2013} and \cite{ChenChristensenNPIVold} for sieve nonparametric instrumental variables regression. And it also satisfies $\zeta _{K,n}\lesssim \sqrt{K}$ with $K = 2^{J}$ (see, e.g., p. 2240 of \cite{Gobet2004} or p. 2377 of \cite{Kato2013}). See
\cite{Johnstone2013} for further details.

\paragraph{Tensor products}

We construct tensor product B-spline or wavelet bases for $[0,1]^d$ as
follows. First, for $x = (x_1,\ldots,x_d) \in [0,1]^d$ we construct $d$
B-spline or wavelet bases for $[0,1]$. We then form the tensor product basis
by taking the product of the elements of each of the univariate bases.
Therefore, $b^K(x)$ may be expressed as
\begin{equation}
b^K(x) = \bigotimes_{l=1}^d b^{K_0}(x_l)
\end{equation}
where the elements of each vector $b^{K_0}(x_l)$ span $\mathrm{BSpl}%
(K_0,[0,1],\gamma)$ with $K_0 = m+r$ for $l = 1,\ldots,d$, or span $\mathrm{%
Wav}(K_0,[0,1],\gamma)$ with $K_0 = 2^J$ for $l = 1,\ldots,d$. We let $%
\mbox
{BSpl}(K,[0,1]^d,\gamma)$ and $\mbox
{Wav}(K,[0,1]^d,\gamma)$ denote the resulting tensor-product spaces spanned
by the $K = (m+r)^d$ or $K = 2^{dJ}$ elements of $b^K$.

\section{Proofs}

\label{proofs sec}

\subsection{Proofs for Section \protect\ref{main sec}}

\begin{proof}[Proof of Lemma \protect\ref{Bconvi.i.d.}]
Follows from Corollary \ref{troppcor} by setting $\Xi _{i,n}=n^{-1}(%
\widetilde {b}^K_w(X_i)\widetilde {b}^K_w(X_i)^{\prime }-I_K)$ and noting
that $R_n\leq n^{-1}(\zeta_{K,n}^2 \lambda_{K,n}^2+1)$, and $\sigma _n^2\leq
n^{-1}(\zeta_{K,n}^2 \lambda_{K,n}^2+1)$.
\end{proof}

\begin{proof}[Proof of Lemma \protect\ref{Bconvbeta}]
Follows from Corollary \ref{beta rate} by setting $\Xi _{i,n}=n^{-1}(%
\widetilde {b}^K_w(X_i)\widetilde {b}^K_w(X_i)^{\prime }-I_K)$ and noting
that $R_n\leq n^{-1}(\zeta_{K,n}^2 \lambda_{K,n}^2+1)$, and $\sigma _n^2\leq
n^{-2}(\zeta_{K,n}^2 \lambda_{K,n}^2+1)$.
\end{proof}

\begin{proof}[Proof of Lemma \protect\ref{reglem gen}]
By rotational invariance, we may rescale $\widehat{h}$ and $\widetilde{h}$
to yield
\begin{equation}
\widehat{h}(x)-\widetilde{h}(x)=\widetilde{b}_{w}^{K}(x)^{\prime }(%
\widetilde{B}_{w}^{\prime }\widetilde{B}_{w}/n)^{-}\widetilde{B}_{w}^{\prime
}e/n
\end{equation}%
where $e=(\epsilon _{1},\ldots ,\epsilon _{n})^{\prime }$.

Let $\widecheck h = \widehat h - \widetilde h$ to simplify notation. By the
mean value theorem, Assumptions \ref{a-wdelta}(i) and \ref{sieve reg gen}%
(i)(iii), for any $(x,x^{\ast })\in \mathcal{D}_n^{2}$ we have
\begin{eqnarray}
|\widecheck{h}(x)-\widecheck{h}(x^{\ast })| &=&|(\widetilde{b}_w^{K}(x)-%
\widetilde{b}_w^{K}(x^{\ast }))^{\prime }(\widetilde{B}^{\prime }_w%
\widetilde{B}_w/n)^{-}\widetilde{B}^{\prime }_we/n| \\
&=&|(x-x^{\ast })^{\prime }\nabla \widetilde{b}^{K}_w(x^{\ast \ast
})^{\prime }(\widetilde{B}^{\prime }_w\widetilde{B}_w/n)^{-}\widetilde{B}%
^{\prime }_we/n| \\
&\leq & C_{\nabla} \lambda_{K,n} n^{\omega_1}K^{\omega_2} \Vert x-x^{\ast }
\Vert \Vert (\widetilde{B}^{\prime }_w\widetilde{B}_w/n)^{-}\Vert \Vert%
\widetilde{B}^{\prime }_we/n\Vert
\end{eqnarray}%
for some $x^{\ast \ast }$ in the segment between $x$ and $x^{\ast }$ and
some finite constant $C_{\nabla}$ (independent of $x,x^*,n,K$). Now, $\Vert (%
\widetilde{B}^{\prime }_w\widetilde{B}_w/n)^{-1}\Vert = O_{p}(1)$ by
Assumption \ref{a-gram}, and we may deduce by Markov's inequality (under
Assumptions \ref{resid reg}(i)(ii)) that $\Vert \widetilde{B}^{\prime
}_we/n\Vert =O_{p}(\sqrt{K/n})$. It follows that
\begin{equation}
\limsup_{n\rightarrow \infty }{\mathbb{P}}\left( C_{\nabla} \Vert (%
\widetilde{B}^{\prime }_w\widetilde{B}_w/n)^{-}\Vert \Vert \widetilde{B}%
^{\prime }_we/n\Vert >\bar{M}\right) =0
\end{equation}%
for any fixed $\bar{M}>0$ (since condition (i) implies $K/n=o(1)$). Let ${%
\mathcal{B}}_{n}$ denote the event on which $C_{\nabla} \Vert (\widetilde{B}%
^{\prime }_w\widetilde{B}_w/n)^{-}\Vert \Vert \widetilde{B}^{\prime
}_we/n\Vert \leq \bar{M}$ and observe that ${\mathbb{P}}({\mathcal{B}}%
_{n}^{c})=o(1)$. On ${\mathcal{B}}_{n}$, for any $C\geq 1$, a finite
positive $\eta_1 =\eta_1(C)$ and $\eta_2 =\eta_2(C)$ can be chosen such that
\begin{equation}
C_{\nabla} \lambda_{K,n} n^{\omega_1}K^{\omega_2} \Vert x-x^{\ast } \Vert
\Vert (\widetilde{B}^{\prime }_w\widetilde{B}_w/n)^{-}\Vert \Vert\widetilde{B%
}^{\prime }_we/n\Vert \leq C\zeta_{K,n} \lambda_{K,n}\sqrt{(\log n)/n}
\end{equation}%
whenever $\Vert x-x^{\ast }\Vert \leq \eta_1 n^{-\eta_2}$, by Assumption \ref%
{sieve reg gen}(ii). Let ${\mathcal{S}}_{n}$ be the smallest subset of $%
\mathcal{D}_n$ such that for each $x\in \mathcal{D}_n$ there exists a $%
x_{n}\in {\mathcal{S}}_{n}$ with $\Vert x_{n}-x\Vert \leq \eta_1 n^{-\eta_2
} $. For any $x\in \mathcal{D}_n$ let $x_{n}(x)$ denote the $x_{n}\in {%
\mathcal{S}}_{n}$ nearest (in Euclidean distance) to $x$. Then on $\mathcal{B%
}_n$ we have
\begin{equation}
|\widecheck{h}(x)-\widecheck{h}(x_{n}(x))|\leq C\zeta_{K,n} \lambda_{K,n}%
\sqrt{(\log n)/n}  \label{bncondition}
\end{equation}%
for any $x\in \mathcal{D}_n$.

Using the fact that ${\mathbb{P}}(A)\leq {\mathbb{P}}(A \cap B) + \mathbb{P}%
(B^c)$, we obtain
\begin{eqnarray}
& & {\mathbb{P}}\left( \Vert \widecheck{h}\Vert _{\infty }\geq 4C\zeta_{K,n}
\lambda_{K,n}\sqrt{(\log n)/n}\right)  \notag \\
&\leq &{\mathbb{P}}\left( \left\{ \Vert \widecheck{h}\Vert _{\infty }\geq
4C\zeta_{K,n} \lambda_{K,n}\sqrt{(\log n)/n}\right\} \cap {\mathcal{B}}%
_{n}\right) +{\mathbb{P}}({\mathcal{B}}_{n}^{c}) \\
&\leq &{\mathbb{P}}\left( \left\{ \sup_{x\in {\mathcal{X}}}|\widecheck{h}(x)-%
\widecheck{h}(x_{n}(x))|\geq 2C\zeta_{K,n} \lambda_{K,n}\sqrt{(\log n)/n}%
\right\} \cap {\mathcal{B}}_{n}\right)  \notag \\
&&+{\mathbb{P}}\left( \left\{ \max_{x_{n}\in {\mathcal{S}}_{n}}|\widecheck{h}%
(x_{n})|\geq 2C\zeta_{K,n} \lambda_{K,n}\sqrt{(\log n)/n}\right\} \cap {%
\mathcal{B}}_{n}\right) +{\mathbb{P}}({\mathcal{B}}_{n}^{c}) \\
&=&{\mathbb{P}}\left( \left\{ \max_{x_{n}\in {\mathcal{S}}_{n}}|\widecheck{h}%
(x_{n})|\geq 2C\zeta_{K,n} \lambda_{K,n}\sqrt{(\log n)/n}\right\} \cap {%
\mathcal{B}}_{n}\right) +o(1)  \label{probterm}
\end{eqnarray}%
where the final line is by (\ref{bncondition}) and the fact that ${\mathbb{P}%
}({\mathcal{B}}_{n}^{c})=o(1)$. The arguments used to control expression (%
\ref{probterm}) differ depending upon whether or not $\{(X_i,Y_i)\}_{i=1}^n$
is i.i.d.

With \textbf{i.i.d. data}, first let ${\mathcal{A}}_{n}$ denote the event on
which $\Vert (\widetilde{B}^{\prime }_w\widetilde{B}_w/n)-I_{K}\Vert \leq
\frac{1}{2}$ and observe that ${\mathbb{P}}({\mathcal{A}}_{n}^{c})=o(1)$
because $\|\widetilde B^{\prime }_w\widetilde B_w/n - I_K \| = o_p(1)$. Let $%
\{\mathcal{A}_n\}$ denote the indicator function of $\mathcal{A}_n$, let $%
\{M_{n}:n\geq 1\}$ be an increasing sequence diverging to $+\infty $, and
define
\begin{eqnarray}
\epsilon _{1,i,n} &:= & \epsilon _{i}\{|\epsilon _{i}|\leq M_{n}\} -
E[\epsilon _{i}\{|\epsilon _{i}|\leq M_{n}\}|X_i] \\
\epsilon_{2,i,n} &:= & \epsilon _{i}-\epsilon _{1,i,n} \\
G_{i,n}(x_n) & := & \widetilde{b}^{K}_w(x_{n})^{\prime }(\widetilde
B_w^{\prime }\widetilde B_w/n)^{-}\widetilde{b}^{K}_w(X_{i})\{\mathcal{A}%
_n\}\,.
\end{eqnarray}
Since ${\mathbb{P}}(A\cap B)\leq {\mathbb{P}}(A)$ and ${\mathbb{P}}(A)\leq {%
\mathbb{P}}(A \cap B) + \mathbb{P}(B^c)$, we have
\begin{subequations}
\begin{eqnarray}
& & {\mathbb{P}}\left( \left\{ \max_{x_{n}\in {\mathcal{S}}_{n}}|%
\widecheck{h}(x_{n})|\geq 2C\zeta_{K,n} \lambda_{K,n}\sqrt{(\log n)/n}%
\right\} \cap {\mathcal{B}}_{n}\right)  \notag \\
& \leq & {\mathbb{P}}\left( \max_{x_{n}\in {\mathcal{S}}_{n}}|\widecheck{h}%
(x_{n})|\geq 2C\zeta_{K,n} \lambda_{K,n}\sqrt{(\log n)/n} \right)  \notag \\
& \leq & {\mathbb{P}}\left( \left\{ \max_{x_{n}\in {\mathcal{S}}_{n}}|%
\widecheck{h}(x_{n})|\geq 2C\zeta_{K,n} \lambda_{K,n}\sqrt{(\log n)/n}%
\right\} \cap {\mathcal{A}}_{n}\right) + \mathbb{P}(\mathcal{A}_n^c)  \notag
\\
& \leq & (\# \mathcal{S}_n) \max_{x_n \in \mathcal{S}_n} \mathbb{P }\left(
\left\{ \left| \frac{1}{n} \sum_{i=1}^n G_{i,n}(x_n) \epsilon_{1,i,n}
\right| \geq C\zeta_{K,n} \lambda_{K,n}\sqrt{(\log n)/n} \right\} \cap
\mathcal{A}_n \right)  \label{rateiid1} \\
& & + \mathbb{P }\left(\left\{ \max_{x_n \in \mathcal{S}_n} \left| \frac{1}{n%
} \sum_{i=1}^n G_{i,n}(x_n) \epsilon_{2,i,n} \right| \geq C\zeta_{K,n}
\lambda_{K,n}\sqrt{(\log n)/n} \right\} \cap \mathcal{A}_n \right) + \mathbb{%
P}(\mathcal{A}_n^c) \,.  \label{rateiid2}
\end{eqnarray}

{Control of (\ref{rateiid1}):} Note that $\Vert (\widetilde{B}_{w}^{\prime }%
\widetilde{B}_{w}/n)^{-1}\Vert \leq 2$ on $\mathcal{A}_{n}$. Therefore, by
the Cauchy-Schwarz inequality and definition of $\epsilon _{1,i,n}$ and $%
\zeta _{K,n}$, $\lambda _{K,n}$, we have:
\end{subequations}
\begin{equation}
|n^{-1}G_{i,n}(x_{n})\epsilon _{1,i,n}|\lesssim \frac{\zeta
_{K,n}^{2}\lambda _{K,n}^{2}M_{n}}{n}\,.
\end{equation}%
Let $E[\,\cdot \,|X_{1}^{n}]$ denote expectation conditional on $%
X_{1},\ldots ,X_{n}$. Assumption \ref{resid reg}(ii) in the i.i.d. data case
implies that $\sup_x E[\epsilon_i^2 |X_i = x] < \infty$. Therefore,
\begin{eqnarray}
&&\sum_{i=1}^{n}E[(n^{-1}G_{i,n}(x_{n})\epsilon _{1,i,n})^{2}|X_1^n ]  \notag
\\
&=&\frac{1}{n^{2}}\sum_{i=1}^{n}E[\epsilon _{1,i,n}^{2}|X_{i}]\widetilde{b}%
_{w}^{K}(x_{n})^{\prime }(\widetilde{B}_{w}^{\prime }\widetilde{B}_{w}/n)^{-}%
\widetilde{b}_{w}^{K}(X_{i})\widetilde{b}_{w}^{K}(X_{i})^{\prime }\{\mathcal{%
A}_{n}\}(\widetilde{B}_{w}^{\prime }\widetilde{B}_{w}/n)^{-}\widetilde{b}%
_{w}^{K}(x_{n}) \\
&\lesssim &\frac{1}{n^{2}}\sum_{i=1}^{n}\widetilde{b}_{w}^{K}(x_{n})^{\prime
}(\widetilde{B}_{w}^{\prime }\widetilde{B}_{w}/n)^{-}\widetilde{b}%
_{w}^{K}(X_{i})\widetilde{b}_{w}^{K}(X_{i})^{\prime }\{\mathcal{A}_{n}\}(%
\widetilde{B}_{w}^{\prime }\widetilde{B}_{w}/n)^{-}\widetilde{b}%
_{w}^{K}(x_{n}) \\
&=&\frac{1}{n}\widetilde{b}_{w}^{K}(x_{n})^{\prime }E[(\widetilde{B}%
_{w}^{\prime }\widetilde{B}_{w}/n)^{-}(\widetilde{B}_{w}^{\prime }\widetilde{%
B}_{w}/n)\{\mathcal{A}_{n}\}(\widetilde{B}_{w}^{\prime }\widetilde{B}%
_{w}/n)^{-}]\widetilde{b}_{w}^{K}(x_{n}) \; \lesssim \; \frac{\zeta
_{K,n}^{2}\lambda _{K,n}^{2}}{n}\,.
\end{eqnarray}%
Bernstein's inequality for independent random variables (see, e.g., pp.
192--193 of \cite{Pollard1984}) then provides that
\begin{eqnarray}
&&(\#\mathcal{S}_{n})\max_{x_{n}\in \mathcal{S}_{n}}\mathbb{P}\left(
\left.\left\{ \left\vert \frac{1}{n}\sum_{i=1}^{n}G_{i,n}(x_{n})\epsilon
_{1,i,n}\right\vert \geq C\zeta _{K,n}\lambda _{K,n}\sqrt{(\log n)/n}%
\right\} \cap \mathcal{A}_{n} \right| X_1^n \right)  \notag \\
&\lesssim &n^{\nu _{1}+\eta _{2}\nu _{2}}\mathrm{exp}\left\{ -\frac{%
C^{2}\zeta _{K,n}^{2}\lambda _{K,n}^{2}(\log n)/n}{C_{1}\zeta
_{K,n}^{2}\lambda _{K,n}^{2}/n+C_{2}\zeta _{K,n}^{2}\lambda
_{K,n}^{2}M_{n}/n\times C\zeta _{K,n}\lambda _{K,n}\sqrt{(\log n)/n}}\right\}
\\
&\lesssim &\mathrm{exp}\left\{ \log n-\frac{C^{2}\zeta _{K,n}^{2}\lambda
_{K,n}^{2}(\log n)/n}{C_{3}\zeta _{K,n}^{2}\lambda _{K,n}^{2}/n}\right\} +%
\mathrm{exp}\left\{ \log n-\frac{C\sqrt{n\log n}}{C_{4}\zeta _{K,n}\lambda
_{K,n}M_{n}}\right\}
\end{eqnarray}%
for finite positive constants $C_{1},\ldots ,C_{4}$ (independent of $%
X_1,\ldots,X_n$). Thus (\ref{rateiid1}) vanishes asymptotically for all
sufficiently large $C$ provided $M_{n}=O(\zeta _{K,n}^{-1}\lambda _{K,n}^{-1}%
\sqrt{n/(\log n)})$.

{Control of the leading term in (\ref{rateiid2}):} First note that $%
|G_{i,n}|\leq 2\zeta _{K,n}^{2}\lambda _{K,n}^{2}$ by the Cauchy-Schwarz
inequality and Assumption \ref{sieve reg gen}(iii). This, together with
Markov's inequality and Assumption \ref{resid reg}(iii) yields
\begin{eqnarray*}
&&{\mathbb{P}}\left( \max_{x_{n}\in {\mathcal{S}}_{n}}\left\vert \frac{1}{n}%
\sum_{i=1}^{n}G_{i,n}\epsilon _{2,i,n}\right\vert \geq C\zeta _{K,n}\lambda
_{K,n}\sqrt{(\log n)/n}\right) \\
&\lesssim &\frac{\zeta _{K,n}^{2}\lambda _{K,n}^{2}E[|\epsilon
_{i}|\{|\epsilon _{i}|>M_{n}\}]}{\zeta _{K,n}\lambda _{K,n}\sqrt{(\log n)/n}}%
\leq \frac{\zeta _{K,n}\lambda _{K,n}\sqrt{n}}{\sqrt{\log n}}\frac{%
E[|\epsilon _{i}|^{2+\delta }\{|\epsilon _{i}|>M_{n}\}]}{M_{n}^{1+\delta }}
\end{eqnarray*}%
which is $o(1)$ provided $\zeta _{K,n}\lambda _{K,n}\sqrt{n/\log n}%
=O(M_{n}^{1+\delta })$. Setting $M_{n}^{1+\delta }\asymp \zeta _{K,n}\lambda
_{K,n}\sqrt{n/\log n}$ trivially satisfies the condition $\zeta
_{K,n}\lambda _{K,n}\sqrt{n/\log n}=O(M_{n}^{1+\delta })$. The condition $%
M_{n}=O(\zeta _{K,n}^{-1}\lambda _{K,n}^{-1}\sqrt{n/(\log n)})$ is satisfied
for this choice of $M_{n}$ provided $\zeta _{K,n}^{2}\lambda
_{K,n}^{2}\lesssim (n/\log n)^{\delta /(2+\delta )}$ (cf. condition (i)).
Finally, it is straightforward to verify that $M_{n}\rightarrow \infty $ as
a consequence of condition (i). Thus, both (\ref{rateiid1}) and (\ref%
{rateiid2}) vanish asymptotically. This completes the proof in the i.i.d.
case.

With \textbf{weakly dependent data} we use $\mathbb{P}(A \cap B) \leq
\mathbb{P}(A)$ to bound remaining term on the right-hand side of (\ref%
{probterm}) by
\begin{eqnarray}
&&{\mathbb{P}}\left( \left\{ \max_{x_{n}\in {\mathcal{S}}_{n}}|\widecheck{h}%
(x_{n})|\geq 2C\zeta_{K,n} \lambda_{K,n}\sqrt{(\log n)/n}\right\} \right)
\notag \\
&\leq &{\mathbb{P}}\left( \max_{x_{n}\in {\mathcal{S}}_{n}}|\widetilde{b}%
^{K}_w(x_{n})^{\prime }\{(\widetilde{B}^{\prime }_w\widetilde{B}%
_w/n)^{-}-I_{K}\}\widetilde{B}^{\prime }_we/n|\geq C\zeta_{K,n} \lambda_{K,n}%
\sqrt{(\log n)/n}\right)  \label{psiterm1} \\
&&+{\mathbb{P}}\left( \max_{x_{n}\in {\mathcal{S}}_{n}}|\widetilde{b}%
^{K}_w(x_{n})^{\prime }\widetilde{B}^{\prime }_we/n|\geq C\zeta_{K,n}
\lambda_{K,n}\sqrt{(\log n)/n}\right) \,.  \label{psiterm2}
\end{eqnarray}%
It is now shown that a sufficiently large $C$ can be chosen to control terms
(\ref{psiterm1}) and (\ref{psiterm2}).

{Control of (\ref{psiterm1}):} The Cauchy-Schwarz inequality and Assumption %
\ref{sieve reg gen}(iii) yield
\begin{equation}
|\widetilde {b}^K_w(x_n)^{\prime }\{(\widetilde {B}^{\prime }_w\widetilde {B}%
_w/n)^{-} - I_K\}\widetilde {B}^{\prime }_we/n|\lesssim \zeta_{K,n}
\lambda_{K,n}\|(\widetilde {B}^{\prime }_w\widetilde {B}_w/n)^{-}-I_K\|%
\times O_p(\sqrt{K/n})
\end{equation}
uniformly for $x_n\in {\mathcal{S}}_n$ (since $\|\widetilde B^{\prime }_w
e/n\| = O_p(\sqrt{K/n})$ under Assumption \ref{resid reg}(i)(ii)). On $%
\mathcal{A}_n$ we have
\begin{equation}
\|(\widetilde {B}^{\prime }_w\widetilde {B}_w/n)^{-}-I_K\| = \|(\widetilde {B%
}^{\prime }_w\widetilde {B}_w/n)^{-1}((\widetilde {B}^{\prime }_w\widetilde {%
B}_w/n)-I_K)\| \leq 2 \|(\widetilde {B}^{\prime }_w\widetilde {B}%
_w/n)-I_K\|\,.
\end{equation}
Thus $\|\widetilde {B}^{\prime }_w\widetilde {B}_w/n-I_K\|=O_p(\sqrt{(\log
n)/K})$ (i.e. condition (ii)) ensures that (\ref{psiterm1}) can be made
arbitrarily small for large enough $C$.

{Control of (\ref{psiterm2}):} Let $M_{n}$ be as in the i.i.d. case and
define
\begin{eqnarray}
\epsilon _{1,i,n} &:= & \epsilon _{i}\{|\epsilon _{i}|\leq M_{n}\} -
E[\epsilon _{i}\{|\epsilon _{i}|\leq M_{n}\}|\mathcal{F}_{i-1}] \\
\epsilon_{2,i,n} &:= & \epsilon _{i}-\epsilon _{1,i,n} \\
g_{i,n}(x_{n})& := &\widetilde{b}_{w}^{K}(x_{n})^{\prime }\widetilde{b}%
_{w}^{K}(X_{i})\{\mathcal{A}_{n}\}\,.
\end{eqnarray}
The relation ${\mathbb{P}}(A)\leq {\mathbb{P}}(A\cap B)+\mathbb{P}(B^{c})$
and the triangle inequality together yield
\begin{subequations}
\begin{eqnarray}
&&{\mathbb{P}}\left( \max_{x_{n}\in {\mathcal{S}}_{n}}|\widetilde{b}%
_{w}^{K}(x_{n})^{\prime }\widetilde{B}_{w}^{\prime }e/n|\geq C\zeta
_{K,n}\lambda _{K,n}\sqrt{(\log n)/n}\right) -{\mathbb{P}}({\mathcal{A}}%
_{n}^{c})  \notag \\
&\leq &(\#{\mathcal{S}}_{n})\max_{x_{n}\in {\mathcal{S}}_{n}}{\mathbb{P}}%
\left( \left\{ \left\vert \frac{1}{n}\sum_{i=1}^{n}g_{i,n}\epsilon
_{1,i,n}\right\vert >\frac{C}{2}\zeta _{K,n}\lambda _{K,n}\sqrt{(\log n)/n}%
\right\} \cap {\mathcal{A}}_{n}\right)  \label{truncnpiv1} \\
&&+{\mathbb{P}}\left( \max_{x_{n}\in {\mathcal{S}}_{n}}\left\vert \frac{1}{n}%
\sum_{i=1}^{n}g_{i,n}\epsilon _{2,i,n}\right\vert \geq \frac{C}{2}\zeta
_{K,n}\lambda _{K,n}\sqrt{(\log n)/n}\right) \,.  \label{truncnpiv2}
\end{eqnarray}%

{Control of (\ref{truncnpiv2}):} First note that $|g_{i,n}|\leq \zeta
_{K,n}^{2}\lambda _{K,n}^{2}$ by the Cauchy-Schwarz inequality and
Assumption \ref{sieve reg gen}(iii). This, together with Markov's inequality
and Assumption \ref{resid reg}(iii) yields
\end{subequations}
\begin{eqnarray*}
&&{\mathbb{P}}\left( \max_{x_{n}\in {\mathcal{S}}_{n}}\left\vert \frac{1}{n}%
\sum_{i=1}^{n}g_{i,n}\epsilon _{2,i,n}\right\vert \geq \frac{C}{2}\zeta
_{K,n}\lambda _{K,n}\sqrt{(\log n)/n}\right) \\
&\lesssim &\frac{\zeta _{K,n}^{2}\lambda _{K,n}^{2}E[|\epsilon
_{i}|\{|\epsilon _{i}|>M_{n}\}]}{\zeta _{K,n}\lambda _{K,n}\sqrt{(\log n)/n}}%
\leq \frac{\zeta _{K,n}\lambda _{K,n}\sqrt{n}}{\sqrt{\log n}}\frac{%
E[|\epsilon _{i}|^{2+\delta }\{|\epsilon _{i}|>M_{n}\}]}{M_{n}^{1+\delta }}
\end{eqnarray*}%
which is $o(1)$ provided $\zeta _{K,n}\lambda _{K,n}\sqrt{n/\log n}%
=O(M_{n}^{1+\delta })$.

{Control of (\ref{truncnpiv1}):} By Assumption \ref{resid reg}(ii), the
predictable variation of the summands in (\ref{truncnpiv1}) may be bounded
by
\begin{eqnarray}
\frac{1}{n^{2}}\sum_{i=1}^{n}E[(g_{i,n}\epsilon _{1,i,n})^{2}|{\mathcal{F}}%
_{i-1}] &\lesssim &n^{-1}\widetilde{b}_{w}^{K}(x_{n})^{\prime }\left(
\widetilde{B}_{w}^{\prime }\widetilde{B}_{w}/n\right) \widetilde{b}%
_{w}^{K}(x_{n}) \\
&\lesssim &\zeta _{K,n}^{2}\lambda _{K,n}^{2}/n\quad
\mbox {on
${\mathcal{A}}_n$}
\end{eqnarray}%
uniformly for $x_{n}\in {\mathcal{S}}_{n}$. Moreover,
\begin{equation}
|n^{-1}g_{i,n}\epsilon _{1,i,n}|\lesssim \frac{\zeta _{K,n}^{2}\lambda
_{K,n}^{2}M_{n}}{n}
\end{equation}%
uniformly for $x_{n}\in {\mathcal{S}}_{n}$. An tail bound for martingales
\cite[Proposition 2.1]{Freedman1975} then provides that
\begin{eqnarray}
&&(\#{\mathcal{S}}_{n})\max_{x_{n}\in {\mathcal{S}}_{n}}{\mathbb{P}}\left(
\left\{ \left\vert \frac{1}{n}\sum_{i=1}^{n}g_{i,n}\epsilon
_{1,i,n}\right\vert >\frac{C}{2}\zeta _{K,n}\lambda _{K,n}\sqrt{(\log n)/n}%
\right\} \cap {\mathcal{A}}_{n}\right)  \notag \\
&\lesssim &n^{\nu _{1}+\eta _{2}\nu _{2}}\mathrm{exp}\left\{ -\frac{%
C^{2}\zeta _{K,n}^{2}\lambda _{K,n}^{2}(\log n)/n}{c_{1}\zeta
_{K,n}^{2}\lambda _{K,n}^{2}/n+c_{2}\zeta _{K,n}^{2}\lambda
_{K,n}^{2}M_{n}/n\times C\zeta _{K,n}\lambda _{K,n}\sqrt{(\log n)/n}}\right\}
\\
&\lesssim &\mathrm{exp}\left\{ \log n-\frac{C^{2}\zeta _{K,n}^{2}\lambda
_{K,n}^{2}(\log n)/n}{c_{3}\zeta _{K,n}^{2}\lambda _{K,n}^{2}/n}\right\} +%
\mathrm{exp}\left\{ \log n-\frac{C\sqrt{n\log n}}{c_{4}\zeta _{K,n}\lambda
_{K,n}M_{n}}\right\}
\end{eqnarray}%
for finite positive constants $c_{1},\ldots ,c_{4}$. Thus (\ref{truncnpiv1})
vanishes asymptotically for all sufficiently large $C$ provided $%
M_{n}=O(\zeta _{K,n}^{-1}\lambda _{K,n}^{-1}\sqrt{n/(\log n)})$. Choosing $%
M_{n}$ as in the i.i.d. case completes the proof.
\end{proof}

\begin{proof}[Proof of Remark \protect\ref{Pkw bound}]
Take any $h\in L_{w,n}^{\infty }$ with $\Vert h\Vert _{\infty ,w}\neq 0$. By
the Cauchy-Schwarz inequality we have
\begin{eqnarray}
|P_{K,w,n}(x)| &\leq &\Vert \widetilde{b}_{w}^{K}(x)\Vert \Vert (\widetilde{B%
}_{w}^{\prime }\widetilde{B}_{w}/n)^{-}\widetilde{B}_{w}^{\prime }H/n\Vert \\
&\leq &\zeta _{K,n}\lambda _{K,n}\Vert (\widetilde{B}_{w}^{\prime }%
\widetilde{B}_{w}/n)^{-}\widetilde{B}_{w}^{\prime }H/n\Vert
\end{eqnarray}%
uniformly over $x$, where $H=(h(X_{1})w_{n}(X_{1}),\ldots
,h(X_{n})w_{n}(X_{n}))^{\prime }$. When $\lambda _{\min }(\widetilde{B}%
_{w}^{\prime }\widetilde{B}_{w}/n)\geq \frac{1}{2}$ (which it is wpa1 since $%
\Vert \widetilde{B}_{w}^{\prime }\widetilde{B}_{w}/n-I_{K}\Vert =o_{p}(1)$),
we have:
\begin{eqnarray*}
\Vert (\widetilde{B}_{w}^{\prime }\widetilde{B}_{w}/n)^{-}\widetilde{B}%
_{w}^{\prime }H/n\Vert ^{2} &=&(H^{\prime }\widetilde{B}_{w}/n)(\widetilde{B}%
_{w}^{\prime }\widetilde{B}_{w}/n)^{-1}(\widetilde{B}_{w}^{\prime }%
\widetilde{B}_{w}/n)^{-1}\widetilde{B}_{w}^{\prime }H/n \\
&\leq &2(H^{\prime }\widetilde{B}_{w}/n)(\widetilde{B}_{w}^{\prime }%
\widetilde{B}_{w}/n)^{-1}\widetilde{B}_{w}^{\prime }H/n \\
&\leq &2\Vert h\Vert _{w,n}^{2}\leq 2\Vert h\Vert _{\infty ,w}^{2}
\end{eqnarray*}%
where the second last line is because $\widetilde{B}_{w}(\widetilde{B}%
_{w}^{\prime }\widetilde{B}_{w})^{-1}\widetilde{B}_{w}^{\prime }$ is a
projection matrix. Thus $\Vert P_{K,w,n}h\Vert _{\infty ,w}/\Vert h\Vert
_{\infty ,w}\leq \sqrt{2}\zeta _{K,n}\lambda _{K,n}$ wpa1 (uniformly in $h$%
). Taking the sup over $h$ yields the desired result.
\end{proof}

\begin{proof}[Proof of Lemma \protect\ref{rate gen}]
It suffices to control the bias term. Note that $\widetilde h = P_{K,w,n}
h_0 $. Therefore, for any $h \in B_{K,w}$ we have, by the usual argument,
\begin{eqnarray}
\|\widetilde h - h_0\|_{\infty,w} & = & \|\widetilde h - h + h -
h_0\|_{\infty,w} \\
& = & \|P_{K,w,n}(h_0 - h) + h - h_0\|_{\infty,w} \\
& \leq & \|P_{K,w,n}(h_0 - h)\|_{\infty,w} + \| h - h_0\|_{\infty,w} \\
& \leq & (1 + \|P_{K,w,n}\|_{\infty,w}) \| h - h_0\|_{\infty,w}\,.
\end{eqnarray}
Taking the infimum over $h \in B_{K,w}$ yields the desired result.
\end{proof}

\begin{proof}[Proof of Theorem \protect\ref{sup norm rate regression}]
The variance term is $O_p(\sqrt{K(\log n)/n})$ by Lemma \ref{reglem gen}:
condition (i) of Lemma \ref{reglem gen} is satisfied by virtue of the
condition $\delta \geq d/p$; condition (ii) is satisfied for $K \asymp
(n/\log n)^{d/(2p+d)}$ directly in the i.i.d. case, and by Lemma \ref%
{Bconvbeta} and the conditions on $p$ for the $\beta$-mixing cases.

For the bias term, it is well known that $\inf_{h\in B_{K,w}}\Vert
h_{0}-h\Vert _{\infty ,w}=O(K^{-p/d})$ under Assumptions \ref{X reg}, \ref%
{parameter regression} and \ref{b sieve} (e.g. \cite{Huang1998} and \cite%
{Chen2007}). It therefore remains to show that $\Vert P_{K,w,n}\Vert
_{\infty }\lesssim 1$ wpa1.

When $B_{K}=\mathrm{BSpl}(K,[0,1]^{d},\gamma )$, we may slightly adapt
Corollary A.1 of \cite{Huang2003} to show that $\Vert P_{K,w,n}\Vert
_{\infty ,w}\lesssim 1$ wpa1, using the fact that the empirical and true $%
L^{2}(X)$ norms are equivalent over $B_{K,w}$ wpa1 by virtue of the
condition $\Vert \widetilde{B}_{w}^{\prime }\widetilde{B}_{w}/n-I_{K}\Vert
=o_{p}(1)$ (see our Lemma \ref{eilem}). This condition is satisfied with $%
K\asymp (n/\log n)^{d/(2p+d)}$ for i.i.d. data (see Lemma \ref{Bconvi.i.d.}%
), and is satisfied in the $\beta $-mixing case by Lemma \ref{Bconvbeta} and
the conditions on $p$.

When $B_{K}=\mathrm{Wav}(K,[0,1]^{d},\gamma )$, the conditions on $K$ in
Theorem \ref{c-pstable} are satisfied with $K \asymp(n/\log n)^{d/(2p+d)}$
under the conditions on $p$ in the Theorem. Therefore, $\Vert P_{K,w,n}\Vert
_{\infty ,w}\lesssim 1$ wpa1.
\end{proof}

\begin{proof}[Proof of Lemma \protect\ref{lem-hhat l2}]
By similar arguments to the proof of Lemma \ref{reglem gen}:
\begin{equation}
\|\widehat h - \widetilde h\|_{L^2(X)} = \|(\widetilde b^K_w)^{\prime
}(\widetilde B^{\prime }_w \widetilde B_w/n)^- \widetilde B_w^{\prime
}e/n\|_{L^2(X)} \leq \|(\widetilde B^{\prime }_w \widetilde B_w/n)^-\| \|
\widetilde B_w^{\prime }e/n\| \,.
\end{equation}
Chebyshev's inequality and Assumption \ref{resid reg}(i)(ii) yield $\|
\widetilde B^{\prime }_w e/n\| = O_p(\sqrt{K/n})$. Moreover, it follows from
Assumption \ref{a-gram} that $\|(\widetilde B^{\prime }_w \widetilde
B_w/n)^-\| = O_p(1)$.

For the remaining term it suffices to show that $\|\widetilde h -
h_0\|_{L^2(X)} = O_p(\|h_0 - h_{0,K}\|_{L^2(X)})$. By the triangle
inequality we bound
\begin{equation}
\|\widetilde h - h_0\|_{L^2(X)} \leq \|\widetilde h - h_{0,K}\|_{L^2(X)} +
\|h_{0,K} - h_0\|_{L^2(X)}\,.
\end{equation}
Recall the definition of the empirical projection $P_{K,w,n}$ from
expression (\ref{e-pkwn def}), and observe that $\widetilde h = P_{K,w,n}
h_0 $ and that $P_{K,w,n} h = h$ for all $h \in B_{K,w}$. Also recall the
definition of $L^2_{w,n}(X)$ as the space of functions with finite norm $%
\|\cdot\|_{w,n}$ where $\|f\|_{w,n}^2 = \frac{1}{n} \sum_{i=1}^n f(X_i)^2
w_n(X_i)$. Since the empirical and theoretical $L^2(X)$ norms are equivalent
over $B_{K,w}$ wpa1 under the condition $\|\widetilde B^{\prime }_w
\widetilde B_w/n - I_K\| = o_p(1)$ (see Section \ref{ei sec-1}). Therefore,
we have
\begin{eqnarray}
\|\widetilde h - h_{0,K}\|_{L^2(X)}^2 & = & \|P_{K,w,n} (h_0 - h_{0,K})
\|_{L^2(X)}^2 \\
& \asymp & \|P_{K,w,n} (h_0 - h_{0,K}) \|_{w,n}^2 \quad \mbox{wpa1} \\
& \leq & \|(h_0 - h_{0,K}) \|_{w,n}^2
\end{eqnarray}
where the second line is by equivalence of the empirical and theoretical $%
L^2(X)$ norms wpa1, and the final line is because $P_{K,w,n}$ is an
orthogonal projection on $L^2_{w,n}(X)$. Finally, Markov's inequality yields
$\|(h_0 - h_{0,K}) \|_{w,n}^2 = O_p(\|h_0 - h_{0,K}\|^2_{L^2(X)})$.
\end{proof}

\subsection{Proofs for Section \protect\ref{inference sec}}

\begin{proof}[Proof of Lemma \protect\ref{lem-omcgce}]
We use a truncation argument together with exponential inequalities for
random matrices. Let $M_n \asymp (\zeta_{K,n}
\lambda_{K,n})^{(2+\delta)/\delta}$ (with $\delta$ as in Assumption \ref%
{resid reg}(iii) be a sequence of positive numbers and let
\begin{eqnarray}
\widehat \Omega_1 & = & \frac{1}{n} \sum_{i=1}^n (\Xi_{1,i} - E[\Xi_{1,i}])
\\
\widehat \Omega_2 & = & \frac{1}{n} \sum_{i=1}^n (\Xi_{2,i} - E[\Xi_{2,i}])
\\
\Xi_{1,i} & = & \epsilon_i^2 \widetilde b^K_w(X_i)\widetilde
b^K_w(X_i)^{\prime }\{ \|\epsilon_i^2 \widetilde b^K_w(X_i)\widetilde
b^K_w(X_i)^{\prime }\| \leq M_n^2\} \\
\Xi_{2,i} & = & \epsilon_i^2 \widetilde b^K_w(X_i)\widetilde
b^K_w(X_i)^{\prime }\{ \|\epsilon_i^2 \widetilde b^K_w(X_i)\widetilde
b^K_w(X_i)^{\prime }\| > M_n^2\} \,.
\end{eqnarray}
Clearly $\widehat \Omega - \Omega = \widehat \Omega_1 + \widehat \Omega_2$,
so it is enough to show that $\|\widehat \Omega_1\| = o_p(1)$ and $%
\|\widehat \Omega_2\| = o_p(1)$.

Control of $\Vert \widehat{\Omega }_{1}\Vert $: By definition, $\Vert \Xi
_{1,i}\Vert \leq M_{n}^{2}$. It follows by the triangle inequality and
Jensen's inequality ($\Vert \cdot \Vert $ is convex) that $\Vert \Xi
_{1,i}-E[\Xi _{1,i}]\Vert \leq 2M_{n}$. Moreover, by Assumption \ref{resid
reg}(ii)
\begin{eqnarray}
E[(\Xi _{1,i}-E[\Xi _{1,i}])^{2}] &\leq &E[\epsilon _{i}^{4}\Vert \widetilde{%
b}_{w}^{K}(X_{i})\Vert ^{2}\widetilde{b}_{w}^{K}(X_{i})\widetilde{b}%
_{w}^{K}(X_{i})^{\prime }\{\Vert \epsilon _{i}^{2}\widetilde{b}%
_{w}^{K}(X_{i})\widetilde{b}_{w}^{K}(X_{i})^{\prime }\Vert \leq M_{n}^{2}\}]
\\
&\leq &M_{n}^{2}E[\epsilon _{i}^{2}\widetilde{b}_{w}^{K}(X_{i})\widetilde{b}%
_{w}^{K}(X_{i})^{\prime }\{\Vert \epsilon _{i}^{2}\widetilde{b}%
_{w}^{K}(X_{i})\widetilde{b}_{w}^{K}(X_{i})^{\prime }\Vert \leq M_{n}^{2}\}]
\\
&\leq &M_{n}^{2}E[E[\epsilon _{i}^{2}|X_{i}]\widetilde{b}_{w}^{K}(X_{i})%
\widetilde{b}_{w}^{K}(X_{i})^{\prime }] \\
&\lesssim &M_{n}^{2}E[\widetilde{b}_{w}^{K}(X_{i})\widetilde{b}%
_{w}^{K}(X_{i})^{\prime }]\;=\;M_{n}^{2}I_{K}
\end{eqnarray}%
where the inequalities are understood in the sense of positive semi-definite
matrices. It follows that $\Vert E[(\Xi _{1,i}-E[\Xi _{1,i}])^{2}]\Vert
\lesssim M_{n}^{2}$. In the i.i.d. case, Corollary \ref{troppcor} yields $%
\Vert \widehat{\Omega }_{1}\Vert =O_{p}(M_{n}\sqrt{(\log K)/n})=o_{p}(1)$.
In the $\beta $-mixing case, Corollary \ref{beta rate} yields $\Vert
\widehat{\Omega }_{1}\Vert =O_{p}(M_{n}\sqrt{q(\log K)/n})$, and the result
follows by taking $q=\gamma ^{-1}\log n$ in the exponentially $\beta $%
-mixing case and $q\asymp n^{1/(1+\gamma )}$ in the algebraically $\beta $%
-mixing case.

Control of $\Vert \widehat{\Omega }_{2}\Vert $: The simple bound $\Vert \Xi
_{2,i}\Vert \leq (\zeta _{K,n}\lambda _{K,n})^{2}\epsilon _{i}^{2}\{\epsilon
_{i}^{2}>M_{n}^{2}/(\zeta _{K,n}\lambda _{K,n})^{2}\}$ together with the
triangle inequality and Jensen's inequality ($\Vert \cdot \Vert $ is convex)
yield
\begin{eqnarray}
E[\Vert \widehat{\Omega }_{2}\Vert ] &\leq &2(\zeta _{K,n}\lambda
_{K,n})^{2}E[\epsilon _{i}^{2}\{|\epsilon _{i}|>M_{n}/(\zeta _{K,n}\lambda
_{K,n})\}] \\
&\leq &2\frac{(\zeta _{K,n}\lambda _{K,n})^{2+\delta }}{M_{n}^{\delta }}%
E[|\epsilon _{i}|^{2+\delta }\{|\epsilon _{i}|>M_{n}/(\zeta _{K,n}\lambda
_{K,n})\}\;=\;o(1)
\end{eqnarray}%
by Assumption \ref{resid reg}(iii) because $M_{n}/(\zeta _{K,n}\lambda
_{K,n})\asymp (\zeta _{K,n}\lambda _{K,n})^{2/\delta }\rightarrow \infty $
and $(\zeta _{K,n}\lambda _{K,n})^{2+\delta }/M_{n}^{\delta }\asymp 1$.
Therefore, $\Vert \widehat{\Omega }_{2}\Vert =o_{p}(1)$ by Markov's
inequality.
\end{proof}

\begin{proof}[Proof of Theorem \protect\ref{t-dist-new}]
First define $u_K^*(x) = v_K^*(x)/\|v_K^*\|_{sd}$. Note that $E[(u_K^*(X_i)
\epsilon_i)^2] = 1$, $E[u_K^*(X_i)^2] = \|v\|_{L^2(X)}^2/\|v\|_{sd}^2 \asymp
1$ (by Assumptions \ref{resid reg}(ii)(iv)), and $\|u_K^*\|_\infty \lesssim
\zeta_{K,n} \lambda_{K,n}$ by the relation between the $L^2$ and sup norms
on $B_{K,w}$.

By Assumption \ref{a-functional}(i)(ii) and the fact that $\widehat h,
\widetilde h \in N_{K,n}$ wpa1, we obtain:
\begin{eqnarray}
\frac{\sqrt{n}(f(\widehat{h})-f(\widetilde{h}))}{V_{K}^{1/2}} &=&\frac{\sqrt{%
n}\frac{\partial f(h_{0})}{\partial h}[\widehat{h}-\widetilde{h}]}{%
V_{K}^{1/2}}+o_p(1) \\
&=&\frac{1}{\sqrt{n}}\sum_{i=1}^{n}u_{K}^{\ast }(X_{i})\epsilon _{i}+\frac{%
\frac{\partial f(h_{0})}{\partial h}[\widetilde{b}_{w}^{K}]^{\prime }((%
\widetilde{B}_{w}^{\prime }\widetilde{B}_{w}/n)^{-}-I_{K})(\widetilde{B}%
_{w}^{\prime }e/\sqrt{n})}{V_{K}^{1/2}}+o_{p}(1)\,.  \label{e-expansion}
\end{eqnarray}%
The leading term is now shown to be asymptotically $N(0,1)$ and the second
term is shown to be asymptotically negligible. The proof of this differs
depending upon whether the data are i.i.d. or weakly dependent.

With \textbf{i.i.d. data}, we first show that the second term on the
right-hand side of (\ref{e-expansion}) is $o_{p}(1)$. Let $\eta >0$ be
arbitrary. Let $C_{\eta }$ be such that $\limsup \mathbb{P}(\Vert \widetilde{%
B}_{w}^{\prime }\widetilde{B}_{w}/n-I_{K}\Vert >C_{\eta }\zeta _{K,n}\lambda
_{K,n}\sqrt{(\log K)/n})\leq \eta $ (we may always choose such a $C_{\eta }$
by Lemma \ref{Bconvi.i.d.}), let $\mathcal{C}_{n,\eta }$ denote the event $%
\Vert \widetilde{B}_{w}^{\prime }\widetilde{B}_{w}/n-I_{K}\Vert \leq C_{\eta
}\zeta _{K,n}\lambda _{K,n}\sqrt{(\log K)/n}$ and let $\{\mathcal{C}_{n,\eta
}\}$ denote its indicator function. Observe that $V_{K}^{1/2}\asymp
\left\Vert \frac{\partial f(h_{0})}{\partial h}[\widetilde{b}%
_{w}^{K}]\right\Vert $ (under Assumption \ref{resid reg}(ii)(iv)). Let $%
E[\,\cdot \,|X_{1}^{n}]$ denote expectation conditional on $X_{1},\ldots
,X_{n}$ and let $\partial \widetilde{b}_{w}^{K}$ denote $\frac{\partial
f(h_{0})}{\partial h}[\widetilde{b}_{w}^{K}]$. By iterated expectations,
\begin{eqnarray}
&&E\left[ \left( \frac{(\partial \widetilde{b}_{w}^{K})^{\prime }((%
\widetilde{B}_{w}^{\prime }\widetilde{B}_{w}/n)^{-}-I_{K})(\widetilde{B}%
_{w}^{\prime }e/\sqrt{n})}{V_{K}^{1/2}}\right) ^{2}\{{\mathcal{C}}_{n,\eta
}\}\right]  \notag \\
&=&\frac{(\partial \widetilde{b}_{w}^{K})^{\prime }E[((\widetilde{B}%
_{w}^{\prime }\widetilde{B}_{w}/n)^{-}-I_{K})E[(\widetilde{B}_{w}^{\prime
}ee^{\prime }\widetilde{B}_{w}/n)|X_{1}^{n}]((\widetilde{B}_{w}^{\prime }%
\widetilde{B}_{w}/n)^{-}-I_{K})\{{\mathcal{C}}_{n,\eta }\}]\partial
\widetilde{b}_{w}^{K}}{V_{K}}  \notag \\
&=&\frac{(\partial \widetilde{b}_{w}^{K})^{\prime }E[((\widetilde{B}%
_{w}^{\prime }\widetilde{B}_{w}/n)^{-}-I_{K})(\frac{1}{n}\sum_{i=1}^{n}E[%
\epsilon _{i}^{2}\widetilde{b}_{w}^{K}(X_{i})\widetilde{b}%
_{w}^{K}(X_{i})^{\prime }|X_{i}])((\widetilde{B}_{w}^{\prime }\widetilde{B}%
_{w}/n)^{-}-I_{K})\{{\mathcal{C}}_{n,\eta }\}]\partial \widetilde{b}_{w}^{K}%
}{V_{K}}  \notag \\
&\lesssim &\frac{(\partial \widetilde{b}_{w}^{K})^{\prime }E[((\widetilde{B}%
_{w}^{\prime }\widetilde{B}_{w}/n)^{-}-I_{K})(\widetilde{B}_{w}^{\prime }%
\widetilde{B}_{w}/n)((\widetilde{B}_{w}^{\prime }\widetilde{B}%
_{w}/n)^{-}-I_{K})\{{\mathcal{C}}_{n,\eta }\}]\partial \widetilde{b}_{w}^{K}%
}{V_{K}}  \notag \\
&\lesssim &C_{\eta }^{2}\zeta _{K,n}^{2}\lambda _{K,n}^{2}(\log K)/n\quad
=\quad o(1)
\end{eqnarray}%
for all $n$ sufficiently large, where the second last line is by Assumption %
\ref{resid reg}(ii) and the final line is because both $\Vert (\widetilde{B}%
_{w}^{\prime }\widetilde{B}_{w}/n)^{-}-I_{K}\Vert \lesssim \Vert (\widetilde{%
B}_{w}^{\prime }\widetilde{B}_{w}/n)-I_{K}\Vert $ and $\Vert (\widetilde{B}%
_{w}^{\prime }\widetilde{B}_{w})/n\Vert \lesssim 1$ hold on $\mathcal{C}%
_{n,\eta }$ for all $n$ sufficiently large under Assumption \ref{a-gram}. As
$\liminf \mathbb{P}(\mathcal{C}_{n,\eta })\geq 1-\eta $ and $\eta $ is
arbitrary, the second term in (\ref{e-expansion}) is therefore $o_{p}(1)$.

Now consider the leading term in (\ref{e-expansion}).
The summands are i.i.d. with mean zero and unit variance. The Lindeberg
condition is easily verified:
\begin{eqnarray}
E[\epsilon_i^2 u_K^*(X_i)^2 \{ |\epsilon_i u_K^*(X_i)| > \eta \sqrt n\}] & =
& E[\epsilon_i^2 u_K^*(X_i)^2 \{ |\epsilon_i | > \eta (\sqrt n/\zeta_{K,n}
\lambda_{K,n})\}] \\
& \leq & \sup_x E[\epsilon_i^2 \{ |\epsilon_i | > \eta (\sqrt n/\zeta_{K,n}
\lambda_{K,n})\}|X_i = x] \; = \;o(1)
\end{eqnarray}
by Assumption \ref{resid reg}(v) because $\zeta_{K,n}^2 \lambda_{K,n}^2/n =
o(1)$. Thus the leading term is asymptotically $N(0,1)$ by the
Lindeberg-Feller theorem.

With \textbf{weakly dependent data} we apply the Cauchy-Schwarz inequality
to the second term in expression (\ref{e-expansion}) to obtain
\begin{eqnarray}
&&\left\vert \frac{\frac{\partial f(h_{0})}{\partial h}[\widetilde{b}%
_{w}^{K}]^{\prime }((\widetilde{B}_{w}^{\prime }\widetilde{B}%
_{w}/n)^{-}-I_{K})(\widetilde{B}_{w}^{\prime }e/\sqrt{n})}{V_{K}^{1/2}}%
\right\vert  \label{e-csbd} \\
&\leq &\frac{\left\Vert \frac{\partial f(h_{0})}{\partial h}[\widetilde{b}%
_{w}^{K}]\right\Vert \Vert (\widetilde{B}_{w}^{\prime }\widetilde{B}%
_{w}/n)^{-}-I_{K}\Vert \Vert \widetilde{B}_{w}^{\prime }e/n\Vert \times
\sqrt{n}}{V_{K}^{1/2}} \\
&\lesssim &\Vert (\widetilde{B}_{w}^{\prime }\widetilde{B}_{w}/n)-I_{K}\Vert
\Vert \widetilde{B}_{w}^{\prime }e/n\Vert \times \sqrt{n}
\end{eqnarray}%
wpa1, because $\left\Vert \frac{\partial f(h_{0})}{\partial h}[\widetilde{b}%
_{w}^{K}]\right\Vert \asymp V_{K}^{1/2}$ and $\Vert (\widetilde{B}%
_{w}^{\prime }\widetilde{B}_{w}/n)^{-}-I_{K}\Vert \leq 2\Vert (\widetilde{B}%
_{w}^{\prime }\widetilde{B}_{w}/n)-I_{K}\Vert $ wpa1 by Assumption \ref%
{a-gram}. Assumption \ref{resid reg}(i)(ii) implies $\Vert \widetilde{B}%
_{w}^{\prime }e/n\Vert =O_{p}(\sqrt{K/n})$, whence the second term in
expression (\ref{e-expansion}) is $o_{p}(1)$ by the condition $\Vert (%
\widetilde{B}_{w}^{\prime }\widetilde{B}_{w}/n)-I_{K}\Vert =o_{p}(K^{-1/2})$.

To show the leading term in (\ref{e-expansion}) is asymptotically $N(0,1)$
we use a martingale CLT (Corollary 2.8 of \cite{McLeish1974}). This
verifying the conditions (a) $\max_{i \leq n} |u_K^*(X_i)\epsilon_i/\sqrt n|
\to_p 0$ and (b) $\frac{1}{n} \sum_{i=1}^n u_K^*(X_i)^2 \epsilon_i^2 \to_p 1$%
. To verify condition (a), let $\eta > 0$ be arbitrary. Then,
\begin{eqnarray}
\mathbb{P}(\max_{i \leq n} |\epsilon_i u_K^*(X_i)/\sqrt n| > \eta) & \leq &
\sum_{i=1}^n \mathbb{P}(|\epsilon_i u_K^*(X_i)/\sqrt n| > \eta) \\
& \leq & \frac{1}{n \eta^2} \sum_{i=1}^n E[\epsilon_i^2
u_K^*(X_i)^2\{|\epsilon_i u_K^*(X_i)/\sqrt n| > \eta\}] \\
& = & \frac{1}{\eta^2} E[\epsilon_i^2 u_K^*(X_i)^2\{|\epsilon_i
u_K^*(X_i)/\sqrt n| > \eta\}]
\end{eqnarray}
which again is $o(1)$ by Assumption \ref{resid reg}(v) since $\zeta_{K,n}^2
\lambda_{K,n}^2/n = o(1)$. For condition (b), note that $\|{\frac{\partial
f(h_0)}{\partial h}[\widetilde{b}_{w}^{K}]}/{{\|v_K^*\|}_{sd}}\| =
\|v_K^*\|_{L^2(X)}/\|v_K^*\|_{sd} \asymp 1$. Then by the Cauchy-Schwarz
inequality, we have
\begin{equation}
\left| \frac{1}{n} \sum_{i=1}^n u_K^*(X_i)^2\epsilon_i^2 - 1 \right| =
\left| \left( \frac{\frac{\partial f(h_0)}{\partial h}[\widetilde{b}_{w}^{K}]%
}{{\|v_K^*\|}_{sd}} \right) ^{\prime }(\widehat \Omega - \Omega ) \left(%
\frac{\frac{\partial f(h_0)}{\partial h}[\widetilde{b}_{w}^{K}]}{{\|v_K^*\|}%
_{sd}} \right) \right| \lesssim \|\widehat \Omega - \Omega\|
\end{equation}
which is $o_p(1)$ by Assumption \ref{a-eucgce}. Therefore, the leading term
in (\ref{e-expansion}) is asymptotically $N(0,1)$.

It remains to show that
\begin{equation}
\frac{\sqrt n(f(\widetilde h) - f(h_0))}{V_K^{1/2}} = o_p(1)\,.
\end{equation}
Assumption \ref{a-functional}(ii) and the fact that $\widetilde h \in
N_{K,n} $ wpa1 together yield
\begin{eqnarray}
\frac{\sqrt n(f(\widetilde h) - f(h_0))}{V_K^{1/2}} & = & \sqrt{\frac{n}{V_K}%
}\frac{\partial f(h_0)}{\partial h}[\widetilde h - h_0] + o_p(1)\,.
\end{eqnarray}
which is $o_p(1)$ by Assumption \ref{a-functional}(iii). %
\end{proof}

\begin{proof}[Proof of Corollary \protect\ref{inference-2}]
The result follows from Theorem \ref{t-dist-new}. Assumption \ref{sieve reg
gen}(iii) is satisfied for these bases under Assumptions \ref{X reg} and \ref%
{b sieve}. Moreover, Assumption \ref{a-gram} is satisfied under the
restrictions on $K$ (see Lemmas \ref{Bconvi.i.d.} and \ref{Bconvbeta}).
Assumption \ref{a-functional} is satisfied provided $\sqrt{n}\Vert
\widetilde{h}-h_{0}\Vert _{\infty }=o(V_{K}^{-1/2})$. But $\Vert \widetilde{h%
}-h_{0}\Vert _{\infty }=O_{p}(K^{-p/d})$ by the proof of Theorem \ref{sup
norm rate regression}, so $\sqrt{n}V_{K}^{-1/2}K^{-p/d}=o(1)$ is sufficient
for Assumption \ref{a-functional} to hold. Moreover, under Assumption \ref%
{resid reg}(iii), Lemma \ref{lem-omcgce} shows that Assumption \ref{a-eucgce}
and the condition $\Vert \widetilde{B}_{w}^{\prime }\widetilde{B}%
_{w}/n-I_{K}\Vert =o_{p}(K^{-1/2})$ are satisfied for weakly dependent data
under the respective conditions on $K$ (see Lemma \ref{Bconvbeta}). %
\end{proof}

\begin{proof}[Proof of Corollary \protect\ref{inference-nl}]
The result follows by Theorem \ref{t-dist-new} with Assumption \ref%
{a-functional}(i')--(iv') in place of Assumption \ref{a-functional} (see
Remark \ref{rmk-snb}). Most conditions of Theorem \ref{t-dist-new} can be
verified in the same way as those for Corollary \ref{inference-2}.
Assumption \ref{a-functional}(iv') is satisfied under the conditions on $K$
because $\Vert \widetilde{h}-h_{0}\Vert _{\infty }=O_{p}(K^{-p/d})$ by the
proof of Theorem \ref{sup norm rate regression}, and $\Vert \widehat{h}-%
\widetilde{h}\Vert _{\infty }=O_{p}(\sqrt{(K\log n)/n})$ by Lemma \ref%
{reglem gen}.
\end{proof}

\begin{proof}[Proof of Lemma \protect\ref{lem-varconsistent}]
Result (2) follows from Theorem \ref{t-dist-new} and Result (1) of Lemma \ref%
{lem-varconsistent} by the continuous mapping theorem. It remains to show
Result (1). By addition and subtraction of terms,
\begin{eqnarray}
\frac{\widehat{\Vert v_{K}^{\ast }\Vert }_{sd}^{2}}{{\Vert v_{K}^{\ast
}\Vert }_{sd}^{2}} &=&\frac{1}{n}\sum_{i=1}^{n}\frac{\epsilon
_{i}^{2}v_{K}^{\ast }(X_{i})^{2}}{{\Vert v_{K}^{\ast }\Vert }_{sd}^{2}}+%
\frac{1}{n}\sum_{i=1}^{n}\frac{\epsilon _{i}^{2}(\widehat{v}_{K}^{\ast
}(X_{i})^{2}-v_{K}^{\ast }(X_{i})^{2})}{{\Vert v_{K}^{\ast }\Vert }_{sd}^{2}}
\\
&&+\frac{1}{n}\sum_{i=1}^{n}\frac{(\widehat{h}%
(X_{i})-h_{0}(X_{i}))^{2}v_{K}^{\ast }(X_{i})^{2}}{{\Vert v_{K}^{\ast }\Vert
}_{sd}^{2}}+\frac{1}{n}\sum_{i=1}^{n}\frac{(\widehat{h}%
(X_{i})-h_{0}(X_{i}))^{2}(\widehat{v}_{K}^{\ast }(X_{i})^{2}-v_{K}^{\ast
}(X_{i})^{2})}{{\Vert v_{K}^{\ast }\Vert }_{sd}^{2}}  \notag \\
&&-\frac{2}{n}\sum_{i=1}^{n}\frac{\epsilon _{i}(\widehat{h}%
(X_{i})-h_{0}(X_{i}))v_{K}^{\ast }(X_{i})^{2}}{{\Vert v_{K}^{\ast }\Vert }%
_{sd}^{2}}-\frac{2}{n}\sum_{i=1}^{n}\frac{\epsilon _{i}(\widehat{h}%
(X_{i})-h_{0}(X_{i}))(\widehat{v}_{K}^{\ast }(X_{i})^{2}-v_{K}^{\ast
}(X_{i})^{2})}{{\Vert v_{K}^{\ast }\Vert }_{sd}^{2}}  \notag \\
=: &&T_{1}+T_{2}+T_{3}+T_{4}+T_{5}+T_{6}\,.  \notag
\end{eqnarray}

Control of $T_1$: $T_1 \to_p 1$ by Assumption \ref{a-eucgce}.

Control of $T_{2}$: Let
\begin{eqnarray}
\partial &=&\frac{\frac{\partial f(h_{0})}{\partial h}[\widetilde{b}_{w}^{K}]%
}{{\Vert v_{K}^{\ast }\Vert }_{sd}} \\
\widehat{\partial } &=&\frac{\frac{\partial f(\widehat{h})}{\partial h}[%
\widetilde{b}_{w}^{K}]}{{\Vert v_{K}^{\ast }\Vert }_{sd}} \\
\widehat{\widehat{\partial \,}} &=&(\widetilde{B}_{w}^{\prime }\widetilde{B}%
_{w}/n)^{-1}\widehat{\partial }\,.
\end{eqnarray}%
Then with this notation,
\begin{equation}
|T_{2}|=\left\vert (\widehat{\widehat{\partial \,}})^{\prime }\widehat{%
\Omega }\widehat{\widehat{\partial \,}}-\partial ^{\prime }\widehat{\Omega }%
\partial \right\vert =|(\widehat{\widehat{\partial \,}}+\partial )^{\prime }%
\widehat{\Omega }(\widehat{\widehat{\partial \,}}-\partial )|\leq \Vert (%
\widehat{\widehat{\partial \,}}+\partial )^{\prime }\Vert \Vert \widehat{%
\Omega }\Vert \Vert (\widehat{\widehat{\partial \,}}-\partial )\Vert \,.
\end{equation}%
Note that $\Vert \widehat{\Omega }\Vert =O_{p}(1)$ by Assumption \ref%
{a-eucgce}(ii). By the triangle inequality and definition of $\partial $, $%
\widehat{\partial }$, and $\widehat{\widehat{\partial \,}}$,
\begin{equation}
\Vert \widehat{\widehat{\partial \,}}-\widehat{\partial }\Vert \leq \Vert (%
\widetilde{B}_{w}^{\prime }\widetilde{B}_{w}/n)^{-}-I_{K}\Vert (\Vert
\widehat{\partial }-\partial \Vert +\Vert \partial \Vert )\,.
\end{equation}%
Assumption \ref{a-gram} implies $\Vert (\widetilde{B}_{w}^{\prime }%
\widetilde{B}_{w}/n)^{-}-I_{K}\Vert =o_{p}(1)$; $\Vert \widehat{\partial }%
-\partial \Vert =o_{p}(1)$ by Assumption \ref{a-functional}(iv), because $%
\widehat{h}\in N_{K,n}$ wpa1; and $\Vert \partial \Vert \asymp 1$ because $%
\Vert v_{K}^{\ast }\Vert _{L^{2}(X)}\asymp \Vert v_{K}^{\ast }\Vert _{sd}$
under Assumption \ref{resid reg}(ii)(iv). Therefore, $\Vert \widehat{%
\widehat{\partial \,}}-\widehat{\partial }\Vert =o_{p}(1)$, $\Vert \widehat{%
\widehat{\partial \,}}+\widehat{\partial }\Vert =O_{p}(1)$, and so $%
|T_{2}|=o_{p}(1)$.

Control of $T_3$: First note that
\begin{equation}
|T_3| \leq \|\widehat h - h_0\|_{\infty,w}^2 \times \frac{1}{n} \sum_{i=1}^n
\frac{v_K(X_i)^2}{\|v_K^*\|^2_{sd}} = o_p(1) \times O_p(1) = o_p(1)
\end{equation}
where $\|\widehat h - h_0\|_{\infty,w} = o_p(1)$ by hypothesis and $n^{-1}
\sum_{i=1}^n v_K(X_i)^2/\|v_K^*\|^2_{sd}$ by Markov's inequality and the
fact that $\|v_K^*\|_{L^2(X)} \asymp \|v_K^*\|_{sd}$ under Assumption \ref%
{resid reg}(ii)(iv).

Control of $T_4$: by the triangle inequality definition of $\widehat v_K^*$
and $v_K^*$:
\begin{eqnarray}
|T_4| & \leq & \|\widehat h - h_0\|^2_{\infty,w} \times \left( \frac{1}{n}
\sum_{i=1}^n \frac{\widehat v_K^*(X_i)^2}{\|v_K^*\|^2_{sd}} + \frac{1}{n}
\sum_{i=1}^n \frac{v_K^*(X_i)^2}{\|v_K^*\|^2_{sd}} \right) \\
& = & o_p(1) \times \left( \widehat{\widehat{ \partial \,}}^{\prime
}\widehat \Omega \widehat{\widehat{ \partial \,}} + \partial^{\prime
}\widehat \Omega \partial \right) \; \leq \; o_p(1) \times \|\widehat \Omega
\| \times \left( \| \widehat{\widehat{ \partial \,}} \|^2 + \|\partial\|^2
\right)\,.
\end{eqnarray}
Moreover, $\|\widehat \Omega\| = O_p(1)$ by Assumption \ref{a-eucgce}, $%
\|\partial\| \asymp 1$ by Assumption \ref{resid reg}(ii)(iv), and $\|%
\widehat{\widehat{ \partial \,}}\| \leq \|\widehat{\widehat{ \partial \,}} -
{\widehat{ \partial}}\| + \|{\widehat{ \partial }} - \partial\| +
\|\partial\| = O_p(1)$ by Assumption \ref{a-gram} and \ref{a-functional}%
(iv). It follows that $|T_4| = o_p(1)$.

Control of $T_5$: By the inequality $2|a| \leq 1 + a^2$, we have
\begin{equation}
|T_5| \leq \|\widehat h - h_0\|_{\infty,w} \frac{1}{n} \sum_{i=1}^n \frac{%
(1+\epsilon_i^2) v_K^*(X_i)^2}{\|v_K^*\|_{sd}^2} = o_p(1) \times O_p(1) =
o_p(1)
\end{equation}
where $\|\widehat h - h_0\|_{\infty,w} = o_p(1)$ by hypothesis, $n^{-1}
\sum_{i=1}^n \epsilon_i^2 v_K^*(X_i)^2/\|v_K^*\|^2_{sd} \to_p 1$ by
Assumption \ref{a-eucgce}, and the remaining term is $O_p(1)$ by the
arguments for $T_3$.

Control of $T_6$: The proof is essentially the same as that for $T_2$,
except we replace $\widehat \Omega$ by the matrix $\widehat \mho = n^{-1}
\sum_{i=1}^n \epsilon_i (\widehat h(X_i) - h_0(X_i)) \widetilde b^K_w(X_i)
\widetilde b^K_w(X_i)^{\prime }$. By the inequality $2|a| \leq 1 + a^2$, it
follows that
\begin{eqnarray}
\|\widehat \mho\| & \leq & \|\widehat h - h_0\|_{\infty,w} \times \left\|
n^{-1}\sum_{i=1}^n (1+\epsilon_i^2) \widetilde b^K_w(X_i) \widetilde
b^K_w(X_i)^{\prime }\right\| \\
& = & \|\widehat h - h_0\|_{\infty,w} \times \left\| \widetilde B_w^{\prime
}\widetilde B_w/n + \widehat \Omega \right\| \; = \; o_p(1) \times O_p(1) \;
=\; o_p(1)
\end{eqnarray}
because $\|\widehat h - h_0\|_{\infty,w} = o_p(1)$, $\|\widetilde
B_w^{\prime }\widetilde B_w/n \| = O_p(1)$ by Assumption \ref{a-gram}, and $%
\|\widehat \Omega\| = O_p(1)$ by Assumption \ref{a-eucgce}.
\end{proof}

\begin{proof}[Proof of Theorem \protect\ref{c-inference}]
This follows from Lemma \ref{lem-varconsistent}.

First, Assumption \ref{a-gram} is satisfied for i.i.d. and $\beta $-mixing
data under the respective conditions on $K$ (see Lemmas \ref{Bconvi.i.d.}
and \ref{Bconvbeta}). Moreover, Lemma \ref{lem-omcgce} shows that Assumption %
\ref{a-eucgce} and the condition $\Vert \widetilde{B}_{w}^{\prime }%
\widetilde{B}_{w}/n-I_{K}\Vert =o_{p}(K^{-1/2})$ is satisfied for weakly
dependent data under the respective conditions on $K$ (see Lemma \ref%
{Bconvbeta}). Therefore Theorem \ref{t-dist-new} may be applied for
asymptotic normality of $f(\widehat{h})$.

To apply Lemma \ref{lem-varconsistent} it remains to show that $\Vert
\widehat{h}-h_{0}\Vert _{\infty ,w}=o_{p}(1)$. But $\Vert \widetilde{h}%
-h_{0}\Vert _{\infty }=o_{p}(1)$ by assumption, and $\Vert \widehat{h}-%
\widetilde{h}\Vert _{\infty }=O_{p}(\zeta _{K,n}\lambda _{K,n}\sqrt{(\log
n)/n})=o_{p}(1)$ by Lemmas \ref{reglem gen}, \ref{Bconvi.i.d.} and \ref%
{Bconvbeta} under the conditions on $K$.
\end{proof}

\subsection{Proofs for Section \protect\ref{ei sec}}

\begin{proof}[Proof of Corollary \protect\ref{troppcor}]
Follows from Theorem \ref{troppthm} with $t=C\sigma _n\sqrt{\log (d_1+d_2)}$
for sufficiently large $C $, and applying the condition $R_n\sqrt{\log
(d_1+d_2)}=o(\sigma _n)$.
\end{proof}

\begin{proof}[Proof of Theorem \protect\ref{beta tropp}]
By Berbee's lemma (enlarging the probability space as necessary) the process
$\{X_i\}$ can be coupled with a process $X_i^*$ such that $%
Y_{k}:=\{X_{(k-1)q+1},\ldots ,X_{kq}\}$ and $Y_k^*:=\{X_{(k-1)q+1}^*,\ldots
,X_{kq}^*\}$ are identically distributed for each $k\geq 1$, ${\mathbb{P}}%
(Y_k\neq Y_k^*)\leq \beta (q)$ for each $k\geq 1 $ and $\{Y_1^*,Y_3^*,\ldots
\}$ are independent and $\{Y_2^*,Y_4^*,\ldots \}$ are independent (see Lemma
2.1 of \cite{Berbee1987}). Let $I_e$ and $I_o $ denote the indices of $%
\{1,\ldots ,n\}$ corresponding to the odd- and even-numbered blocks, and $%
I_r $ the indices in the remainder, so $I_r=q[n/q]+1,\ldots ,n$ when $%
q[n/q]<n$ and $I_r=\emptyset $ when $q[n/q]=n$.

Let $\Xi _{i,n}^*=\Xi (X_{i,n}^*)$. By the triangle inequality,
\begin{equation}
\begin{array}{rcl}
&  & {\mathbb{P}}\left (\|\sum _{i=1}^n\Xi _{i,n}\|\geq 6t\right ) \\
& \leq & {\mathbb{P}}(\|\sum _{i=1}^{[n/q]q}\Xi _{i,n}^*\|+\|\sum _{i\in
I_r}\Xi _{i,n}\|+\|\sum _{i=1}^{[n/q]q}(\Xi _{i,n}^*-\Xi _{i,n})\|\geq 6t)
\\
& \leq & \frac{n}{q}\beta (q)+{\mathbb{P}}\left (\|\sum _{i\in I_r}\Xi
_{i,n}\|\geq t\right )+{\mathbb{P}}\left (\|\sum _{i\in I_e}\Xi
_{i,n}^*\|\geq t\right )+{\mathbb{P}}\left (\|\sum _{i\in I_o}\Xi
_{i,n}^*\|\geq t\right )%
\end{array}%
\end{equation}
To control the last two terms we apply Theorem \ref{troppthm}, recognizing
that $\sum _{i\in I_e}\Xi _{i,n}^*$ and $\sum _{i\in I_o}\Xi _{i,n}^*$ are
each the sum of fewer than $[n/q]$ independent $d_1\times d_2$ matrices,
namely $W_k^*=\sum _{i=(k-1)q+1}^{kq}\Xi _{i,n}^*$. Moreover each $W_k^*$
satisfies $\|W_k^*\|\leq qR_n$ and $\max \{\|E[W_k^*W_k^{*\prime
}]\|,\|E[W_k^{*\prime }W_k^*]\|\}\leq q^2s_n$. Theorem \ref{troppthm} then
yields
\begin{equation}
{\mathbb{P}}\left (\left \|\sum _{i\in I_e}\Xi _{i,n}^*\right \|\geq t\right
)\leq (d_1+d_2) \mathrm{exp}\left (\frac{-t^2/2}{nqs_n^2+qR_nt/3}\right )
\end{equation}
and similarly for $I_o$.
\end{proof}

\begin{proof}[Proof of Corollary \protect\ref{beta rate}]
Follows from Theorem \ref{beta tropp} with $t=Cs_n\sqrt{nq\log (d_1+d_2)}$
for sufficiently large $C$, and the conditions $\frac{n}{q}\beta (q)=o(1)$
and $R_n\sqrt{q\log (d_1+d_2)}=o(s_n\sqrt{n})$.
\end{proof}

\begin{proof}[Proof of Lemma \protect\ref{eilem}]
Let $G=E[b^{K}_w(X_i)b^{K}_w(X_i)^{\prime }]$. Since $B_{K,w}=clsp%
\{b_{K1}w_n,\ldots ,b_{KK}w_n\}$, we have:
\begin{eqnarray}
&&\sup \{\textstyle |\frac{1}{n}\sum _{i=1}^{n}b(X_{i})^{2}-1|:b\in
B_{K,w},E[b(X)^{2}]=1\}  \notag \\
&=&\sup \{|c^{\prime }(B^{\prime }_wB_w/n-G)c|:c\in {\mathbb{R}}^{K},\Vert
G^{1/2}c\Vert =1\} \\
&=&\sup \{|c^{\prime }G^{1/2}(G^{-1/2}(B^{\prime
}_wB_w/n)G^{-1/2}-I_{K})G^{1/2}c|:c\in {\mathbb{R}}^{K},\Vert G^{1/2}c\Vert
=1\} \\
&=&\sup \{|c^{\prime }(\widetilde {B}^{\prime }_w\widetilde {B}%
_w/n-I_{K})c|:c\in {\mathbb{R}}^{K},\Vert c\Vert =1\} \\
&=&\Vert \widetilde {B}^{\prime }_w\widetilde {B}_w/n-I_{K}\Vert _{2}^{2}
\end{eqnarray}
as required.
\end{proof}

\subsection{Proofs for Section \protect\ref{sec-stable}}

We first present a general result that allows us to bound the $L^{\infty }$
operator norm of the $L^{2}(X)$ projection $P_{K}$ onto a linear sieve space
$B_{K}\equiv clsp\{b_{K1},\ldots ,b_{KK}\}$ by the $\ell ^{\infty }$ norm of
the inverse of its corresponding Gram matrix.

\begin{lemma}
\label{lem-p-matrix} If there exists a sequence of positive constants $%
\{c_{K}\}$ such that (i) $\sup_{x\in \mathcal{X}}\Vert b^{K}(x)\Vert _{\ell
^{1}}\lesssim c_{K}$ and (ii) $\max_{1\leq k\leq K}\Vert b_{Kk}\Vert
_{L^{1}(X)}\lesssim c_{K}^{-1}$, then
\begin{equation*}
\Vert P_{K}\Vert _{\infty }\lesssim \Vert \left(
E[b^{K}(X_{i})b^{K}(X_{i})^{\prime }]\right) ^{-1}\Vert _{\ell ^{\infty }}\,.
\end{equation*}
\end{lemma}

\begin{proof}[Proof of Lemma \protect\ref{lem-p-matrix}]
By H\"{o}lder's inequality (with (i)), definition of the operator norm, and H%
\"{o}lder's inequality again (with (ii)), we obtain:
\begin{eqnarray*}
|P_{K}f(x)| &=&|b^{K}(x)^{\prime }\left( E[b^{K}(X_{i})b^{K}(X_{i})^{\prime
}]\right) ^{-1}E[b^{K}(X_{i})f(X_{i})]| \\
&\leq &\Vert b^{K}(x)\Vert _{\ell ^{1}}\Vert \left(
E[b^{K}(X_{i})b^{K}(X_{i})^{\prime }]\right)
^{-1}E[b^{K}(X_{i})f(X_{i})]|\Vert _{\ell ^{\infty }} \\
&\lesssim &c_{K}\Vert \left( E[b^{K}(X_{i})b^{K}(X_{i})^{\prime }]\right)
^{-1}E[b^{K}(X_{i})f(X_{i})]|\Vert _{\ell ^{\infty }} \\
&\leq &c_{K}\Vert \left( E[b^{K}(X_{i})b^{K}(X_{i})^{\prime }]\right)
^{-1}\Vert _{\ell ^{\infty }}\Vert E[b^{K}(X_{i})f(X_{i})]|\Vert _{\ell
^{\infty }} \\
&=&c_{K}\Vert \left( E[b^{K}(X_{i})b^{K}(X_{i})^{\prime }]\right) ^{-1}\Vert
_{\ell ^{\infty }}\max_{1\leq k\leq K}E[|b_{Kk}(X_{i})f(X_{i})|] \\
&\leq &c_{K}\Vert \left( E[b^{K}(X_{i})b^{K}(X_{i})^{\prime }]\right)
^{-1}\Vert _{\ell ^{\infty }}\max_{1\leq k\leq K}E[|b_{Kk}(X_{i})|]\Vert
f\Vert _{\infty } \\
&\lesssim &\Vert \left( E[b^{K}(X_{i})b^{K}(X_{i})^{\prime }]\right)
^{-1}\Vert _{\infty }\Vert f\Vert _{\infty }
\end{eqnarray*}%
uniformly in $x$. The result now follows by taking the supremum over $x\in
\mathcal{X}$.
\end{proof}

We will bound $\Vert \left( E[b^{K}(X_{i})b^{K}(X_{i})^{\prime }]\right)
^{-1}\Vert _{\ell ^{\infty }}$ for (tensor product) wavelet bases using the
following Lemma.

\begin{lemma}
\label{t-dms} Let $A\in \mathbb{R}^{K\times K}$ be a positive definite
symmetric matrix such that $A_{i,j}=0$ whenever $|i-j|>m/2$ for $m$ even.
Then: $\Vert A^{-1}\Vert _{\ell ^{\infty }}\leq \frac{2C}{1-\lambda }$ where
\begin{eqnarray*}
\kappa &=&\lambda _{\max }(A)/\lambda _{\min }(A) \\
\lambda &=&\left( \frac{\sqrt{\kappa }-1}{\sqrt{\kappa }+1}\right)
^{2/m}\;<\;1 \\
C &=&\Vert A^{-1}\Vert \max \{1,(1+\sqrt{\kappa })^{2}/(2\kappa )\}\,.
\end{eqnarray*}
\end{lemma}

\begin{proof}[Proof of Lemma \protect\ref{t-dms}]
By definition of the matrix infinity norm, we have
\begin{equation*}
\Vert A^{-1}\Vert _{\infty }=\max_{j\leq K}\sum_{k=1}^{K}|(A^{-1})_{j,k}|\,.
\end{equation*}%
The result now follows by Theorem 2.4 of \cite{DemkoMossSmith} (which states
that $\left\vert (A^{-1})_{i,j}\right\vert \leq C\lambda ^{|i-j|}$ for all $%
i,j$) and geometric summation.
\end{proof}

\begin{proof}[Proof of Theorem \protect\ref{t-wavd}]
We first prove the univariate case (i.e. $d = 1$) before generalizing to the
multivariate case.

By the definition of wavelet basis, we may assume without loss of generality
that $b_{K1}=\varphi _{J,0}$, \ldots , $b_{KK}=\varphi _{J,2^{J}-1}$ with $%
K=2^{J}$.

For any $x\in \lbrack 0,1]$ the vector $b^{K}(x)$ has, at most, $2N$
elements that are nonzero (as a consequence of the compact support of the $%
\varphi _{J,k})$. It follows that
\begin{equation}
\Vert b^{K}(x)\Vert _{\ell ^{1}}\leq (2N)2^{J/2}\max \{\Vert \varphi \Vert
_{\infty },\Vert \varphi _{0}^{l}\Vert _{\infty },\ldots ,\Vert \varphi
_{N-1}^{l}\Vert _{\infty },\Vert \varphi _{-1}^{r}\Vert _{\infty },\ldots
,\Vert \varphi _{-N}^{r}\Vert _{\infty }\}\lesssim 2^{J/2}
\end{equation}%
uniformly in $x$. Therefore $\sup_{x\in \lbrack 0,1]}\Vert b^{K}(x)\Vert
_{\ell ^{1}}\lesssim \sqrt{K}$. Let $k$ be such that $N\leq k\leq 2^{J}-N-1$%
. By boundedness of $f_{X}$ and a change of variables, we have (with $\mu $
denoting Lebesgue measure)
\begin{eqnarray}
E[|\varphi _{J,k}(X_{i})|] &\leq &\sup_{x\in \lbrack 0,1]}f_{X}(x)\int_{%
\mathbb{R}}2^{J/2}|\varphi (2^{J}x-k)|\,\mathrm{d}\mu (x) \\
&=&\sup_{x\in \lbrack 0,1]}f_{X}(x)2^{-J/2}\int_{\mathbb{R}}|\varphi (y)|\,%
\mathrm{d}\mu (y) \\
&=&\sup_{x\in \lbrack 0,1]}f_{X}(x)2^{-J/2}\Vert \varphi \Vert _{L^{1}(\mu )}
\end{eqnarray}%
where $\Vert \varphi \Vert _{L^{1}(\mu )}<\infty $ because $\varphi $ has
compact support and is continuous. Similar arguments can be used to show the
same for the $N$ left and right scaling functions. It follows that $%
\max_{k}\Vert b_{Kk}\Vert _{L^{1}(X)}\lesssim 2^{-J/2}=K^{-1/2}$. Therefore,
the $b_{K1},\ldots ,b_{KK}$ satisfy the conditions of Lemma \ref%
{lem-p-matrix} and hence $\Vert P_{K}\Vert _{\infty }\lesssim \Vert \left(
E[b^{K}(X_{i})b^{K}(X_{i})^{\prime }]\right) ^{-1}\Vert _{\ell ^{\infty }}$.

It remains to prove that $\Vert \left( E[b^{K}(X_{i})b^{K}(X_{i})^{\prime
}]\right) ^{-1}\Vert _{\ell ^{\infty }}\lesssim 1$. We first verify the
conditions of Lemma \ref{t-dms}. Disjoint support of the $\varphi _{J,k}$
implies that $(E[b^{K}(X_{i})b^{K}(X_{i})^{\prime }])_{k,j}=0$ whenever $%
|k-j|>2N-1$. For positive definiteness, we note that
\begin{equation}
\lambda _{\max }(E[b^{K}(X_{i})b^{K}(X_{i})^{\prime }])\leq \left(
\sup_{x\in \lbrack 0,1]}f_{X}(x)\right) \lambda _{\max }\left(
\int_{[0,1]}b^{K}(x)b^{K}(x)^{\prime }\,\mathrm{d}\mu (x)\right) =\left(
\sup_{x\in \lbrack 0,1]}f_{X}(x)\right)
\end{equation}%
(where we understand the integral performed element wise) because $\varphi
_{J,0},\ldots ,\varphi _{J,2^{J}-1}$ are an orthonormal basis for $V_{J}$
with respect to the $L^{2}([0,1])$ inner product. Similarly, $\lambda _{\min
}(E[b^{K}(X)b^{K}(X)^{\prime }])\geq \inf_{x\in \lbrack 0,1]}f_{X}(x)$.
Therefore
\begin{equation*}
\kappa \leq (\sup_{x\in \lbrack 0,1]}f_{X}(x))/(\inf_{x\in \lbrack
0,1]}f_{X}(x))<\infty
\end{equation*}%
uniformly in $K$, and
\begin{equation*}
\Vert \left( E[b^{K}(X_{i})b^{K}(X_{i})^{\prime }]\right) ^{-1}\Vert \leq
1/(\inf_{x\in \lbrack 0,1]}f_{X}(x))<\infty
\end{equation*}%
uniformly in $K$. This verifies the conditions of Lemma \ref{t-dms} for $%
A=E[b^{K}(X_{i})b^{K}(X_{i})^{\prime }]$. It follows by Lemma \ref{t-dms}
that $\Vert \left( E[b^{K}(X_{i})b^{K}(X_{i})^{\prime }]\right) ^{-1}\Vert
_{\ell ^{\infty }}\lesssim 1$, as required.

We now adapt the preceding arguments to the multivariate case. For any $x =
(x_1,\ldots,x_d) \in [0,1]^d$ we define $b^K(x) = \otimes_{l=1}^d
b^{K_0}(x_l)$ where $b^{K_0}(x_l) = (\varphi_{J,0}(x_l), \ldots,
\varphi_{J,2^J-1}(x_l))^{\prime }$ and $K_0 = 2^J$.

Recall that $K=2^{Jd}$. For any $x=(x_{1},\ldots ,x_{d})\in \lbrack 0,1]^{d}$
we have
\begin{eqnarray}
\Vert b^{K}(x)\Vert _{\ell ^{1}} &=&\prod_{l=1}^{d}\Vert
b^{K_{0}}(x_{l})\Vert _{\ell ^{1}} \\
&\leq &\left( (2N)2^{J/2}\max \{\Vert \varphi \Vert _{\infty },\Vert \varphi
_{0}^{l}\Vert _{\infty },\ldots ,\Vert \varphi _{N-1}^{l}\Vert _{\infty
},\Vert \varphi _{-1}^{r}\Vert _{\infty },\ldots ,\Vert \varphi
_{-N}^{r}\Vert _{\infty }\}\right) ^{d} \\
&=&\lesssim (2^{J/2})^{d}\;=\;\sqrt{K}\,.
\end{eqnarray}%
With slight abuse of notation we let $X_{i1},\ldots ,X_{id}$ denote the $d$
elements of $X_{i}$. For $0\leq k_{1},\ldots ,k_{d}\leq 2^{J}-1$, Fubini's
theorem and a change of variables yields
\begin{eqnarray}
E\left[ \left\vert \prod_{l=1}^{d}\varphi _{J,k}(X_{il})\right\vert \right]
&\leq &\sup_{x\in \lbrack 0,1]^{d}}f_{X}(x)\int_{\mathbb{R}^{d}}\left(
\prod_{l=1}^{d}|\varphi _{J,k_{l}}(x_{l})|\right) d\mu (x_{1},\ldots ,x_{d})
\\
&=&\sup_{x\in \lbrack 0,1]^{d}}f_{X}(x)\prod_{l=1}^{d}\left( \int_{\mathbb{R}%
}|\varphi _{J,k_{l}}(x_{l})|\,\mathrm{d}\mu (x_{l})\right) \\
&\lesssim &(2^{-J/2})^{d}\;=\;K^{-1/2}\,.
\end{eqnarray}%
This verifies the conditions of Lemma \ref{lem-p-matrix} and hence $\Vert
P_{K}\Vert _{\infty }\lesssim \Vert \left(
E[b^{K}(X_{i})b^{K}(X_{i})^{\prime }]\right) ^{-1}\Vert _{\ell ^{\infty }}$.

The tensor product basis is an orthonormal basis with respect to Lebesgue
measure on $[0,1]^{d}$ (by Fubini's theorem). Therefore, the minimum and
maximum eigenvalues of $E[b^{K}(X_{i})b^{K}(X_{i})^{\prime }]$ may be shown
to be bounded below and above by $\inf_{x\in \lbrack 0,1]^{d}}f_{X}(x)$ and $%
\sup_{x\in \lbrack 0,1]^{d}}f_{X}(x)$ as in the univariate case. Again,
compact support of the $\varphi _{J,k}$ and the tensor product construction
implies that $E[b^{K}(X_{i})b^{K}(X_{i})^{\prime }]$ is banded: $%
(E[b^{K}(X_{i})b^{K}(X_{i})^{\prime }])_{k,j}=0$ whenever $|k-j|>(2N-1)^{d}$%
. This verifies the conditions of Lemma \ref{t-dms} for $%
E[b^{K}(X_{i})b^{K}(X_{i})^{\prime }]$. It follows by Lemma \ref{t-dms} that
$\Vert \left( E[b^{K}(X_{i})b^{K}(X_{i})^{\prime }]\right) ^{-1}\Vert _{\ell
^{\infty }}\lesssim 1$, as required.
\end{proof}

\begin{theorem}
\label{t-wavd-emp} Under conditions stated in Theorem \ref{t-wavd}, we have $%
\Vert P_{K,n}\Vert _{\infty }\lesssim 1$ wpa1 provided the following are
satisfied:

\begin{enumerate}
\item[(i)] $\Vert \left( {B^{\prime }B}/{n}\right)
-E[b^{K}(X_{i})b^{K}(X_{i})^{\prime }]\Vert =o_{p}(1)$, and

\item[(ii)] $\max_{1\leq k\leq K}\left\vert \frac{\frac{1}{n}%
\sum_{i=1}^{n}|b_{Kk}(X_{i})|-E[|b_{Kk}(X_{i})|]}{E[|b_{Kk}(X_{i})|]}%
\right\vert =o_{p}(1).$
\end{enumerate}
\end{theorem}

\begin{proof}[Proof of Theorem \protect\ref{t-wavd-emp}]
Condition (ii) $\max_{1\leq k\leq K}\frac{\frac{1}{n}%
\sum_{i=1}^{n}|b_{Kk}(X_{i})|-E[|b_{Kk}(X_{i})|]}{E[|b_{Kk}(X_{i})|]}%
=o_{p}(1)$ implies
\begin{equation}
\max_{1\leq k\leq K}\frac{1}{n}\sum_{i=1}^{n}|b_{Kk}(X_{i})|\lesssim
\max_{1\leq k\leq K}\Vert b_{Kk}\Vert _{L^{1}(X)}\lesssim K^{-1/2}
\end{equation}%
where the final inequality is by the proof of Theorem \ref{t-wavd}.
Moreover, $\sup_{x}\Vert b^{K}(x)\Vert _{\ell ^{1}}\lesssim \sqrt{K}$ by the
proof of Theorem \ref{t-wavd}. It follows analogously to Lemma \ref%
{lem-p-matrix} that $\Vert P_{K,n}\Vert _{\infty }\lesssim \Vert \left(
B^{\prime }B/n\right) ^{-1}\Vert _{\infty }$ wpa1 (noting that $B^{\prime
}B/n$ is invertible wpa1 because $\Vert \left( B^{\prime }B/n\right)
-E[b^{K}(X_{i})b^{K}(X_{i})^{\prime }]\Vert =o_{p}(1)$ and $\lambda
_{K,n}\lesssim 1$).

Condition (i) $\Vert \left( B^{\prime }B/n\right)
-E[b^{K}(X_{i})b^{K}(X_{i})^{\prime }]\Vert =o_{p}(1)$ implies (1) $\lambda
_{\min }(B^{\prime }B/n)\gtrsim \lambda _{\min
}(E[b^{K}(X_{i})b^{K}(X_{i})^{\prime }])$, (2) $\lambda _{\max }(B^{\prime
}B/n)\lesssim \lambda _{\max }(E[b^{K}(X_{i})b^{K}(X_{i})^{\prime }])$, and
(3) $\Vert \left( B^{\prime }B/n\right) ^{-1}\Vert \lesssim \Vert \left(
E[b^{K}(X_{i})b^{K}(X_{i})^{\prime }]\right) ^{-1}\Vert $ all hold wpa1.
Moreover, $\lambda _{\min }(E[b^{K}(X_{i})b^{K}(X_{i})^{\prime }])\gtrsim 1$
and $\lambda _{\max }(E[b^{K}(X_{i})b^{K}(X_{i})^{\prime }])\lesssim 1$ by
the proof of Theorem \ref{t-wavd}. It follows by Lemma \ref{t-dms} that $%
\Vert \left( B^{\prime }B/n\right) ^{-1}\Vert _{\ell ^{\infty }}\lesssim 1$
wpa1, as required.
\end{proof}

\begin{proof}[Proof of Theorem \protect\ref{c-pstable}]
Condition (i) of Theorem \ref{t-wavd-emp} is satisfied because $\lambda
_{K,n}\lesssim 1$ and the condition $\Vert (\widetilde{B}^{\prime }%
\widetilde{B}/n)-I_{K}\Vert =o_{p}(1)$ under the conditions on $K$ (see
Lemma \ref{Bconvi.i.d.} for the i.i.d. case and Lemma \ref{Bconvbeta} for
the weakly dependent case). Therefore,
\begin{eqnarray}
\Vert ({B^{\prime }B}/{n})-E[b^{K}(X_{i})b^{K}(X_{i})^{\prime }]\Vert &\leq
&[\lambda _{\min }(E[b^{K}(X_{i})b^{K}(X_{i})^{\prime }])]^{-1}\Vert (%
\widetilde{B}^{\prime }\widetilde{B}/n)-I_{K}\Vert \\
&\lesssim &\Vert (\widetilde{B}^{\prime }\widetilde{B}/n)-I_{K}\Vert
\;=\;o_{p}(1)\,.
\end{eqnarray}

It remains to verify condition (ii) of Theorem \ref{t-wavd-emp}. Let $%
b_{K1}=\varphi _{J,0}^{d},\ldots ,b_{KK}=\varphi _{J,2^{J}-1}^{d}$ with $%
K=2^{dJ}$ as in the proof of Theorem \ref{t-wavd}. Similar arguments to the
proof of Theorem \ref{t-wavd} yield the bounds $\Vert b_{Kk}\Vert _{\infty
}\lesssim 2^{dJ/2}=\sqrt{K}$ uniformly for $1\leq k\leq K$. Let $f_{X}(x)$
denote the density of $X$. Then by $\inf_{x\in \lbrack 0,1]^{d}}|f_{X}(x)|>0$
and Fubini's theorem
\begin{eqnarray}
E[|b_{Kk}(X)|] &\geq &\left( \inf_{x\in \lbrack 0,1]^{d}}f_{X}(x)\right)
\int_{[0,1]^{d}}\left( \prod_{l=1^{d}}|\varphi _{J,k_{l}}(x_{l})|\right) \,%
\mathrm{d}\mu (x_{1},\ldots ,x_{d}) \\
&=&\left( \inf_{x\in \lbrack 0,1]^{d}}f_{X}(x)\right) \prod_{l=1}^{d}\left(
\int_{[0,1]}|\varphi _{J,k_{l}}(x_{l})|\,\mathrm{d}\mu (x_{l})\right) \,.
\end{eqnarray}%
A change of variables argument yields $\int_{[0,1]}|\varphi
_{J,k_{l}}(x_{l})|\,\mathrm{d}\mu (x_{l})\gtrsim 2^{-J/2}$ uniformly for $%
0\leq k_{l}\leq 2^{J}-1$, and so $E[|b_{Kk}(X)|]\gtrsim 2^{-dJ/2}=K^{-1/2}$
uniformly for $1\leq k\leq K$.

For the \textbf{i.i.d. case}, define $b_{Kk}^{\ast
}(X_{i})=n^{-1}(|b_{Kk}(X_{i})|-E[|b_{Kk}(X_{i})|])/(E[|b_{Kk}(X_{i})|])$
for each $1\leq k\leq K$. It may be deduced from the preceding bounds and
the fact that $E[b_{Kk}(X_{i})^{2}]\asymp 1$ that $\Vert b_{Kk}^{\ast }\Vert
_{\infty }\lesssim K/n$ and $E[b_{Kk}^{\ast }(X_{i})^{2}]\lesssim K/n^{2}$.
By the union bound and Bernstein's inequality (see, e.g., pp. 192--193 of
\cite{Pollard1984}) we obtain, for any $t>0$,
\begin{eqnarray}
&&\mathbb{P}\left( \max_{1\leq k\leq K}\left\vert \frac{\frac{1}{n}%
\sum_{i=1}^{n}|b_{Kk}(X_{i})|-E[|b_{Kk}(X_{i})|]}{E[|b_{Kk}(X_{i})|]}%
\right\vert >t\right)  \notag \\
&\leq &\sum_{k=1}^{K}\mathbb{P}\left( \left\vert \frac{\frac{1}{n}%
\sum_{i=1}^{n}|b_{Kk}(X_{i})|-E[|b_{Kk}(X_{i})|]}{E[|b_{Kk}(X_{i})|]}%
\right\vert >t\right) \\
&\leq &2\exp \left\{ \log K-\frac{t^{2}/2}{c_{1}K/n+c_{2}K/nt}\right\}
\label{e-l1bern}
\end{eqnarray}%
where $c_{1}$ and $c_{2}$ are finite positive constants independent of $t$.
The right-hand side of (\ref{e-l1bern}) vanishes as $n\rightarrow \infty $
since $K\log n/n=o(1)$.

For the \textbf{beta-mixing regressors} \textbf{case}, we may extend the
proof for the i.i.d. case using a coupling argument similar to the proof of
Theorem \ref{beta tropp} to deduce that
\begin{eqnarray}
&&\mathbb{P}\left( \max_{1\leq k\leq K}\left\vert \frac{\frac{1}{n}%
\sum_{i=1}^{n}|b_{Kk}(X_{i})|-E[|b_{Kk}(X_{i})|]}{E[|b_{Kk}(X_{i})|]}%
\right\vert >t\right)  \notag \\
&\lesssim &\frac{n}{q}\beta (q)+\exp \left\{ \log n-\frac{t^{2}}{%
c_{1}Kq/n+c_{2}Kq/nt}\right\} \,.
\end{eqnarray}%
The right-hand side is $o(1)$ provided $\frac{n}{q}\beta (q)=o(1)$ and $%
(qK\log n)/n=o(1)$. Both these conditions are satisfied under the conditions
on $K$, taking $q=\gamma ^{-1}\log n$ in the exponentially $\beta $-mixing
case and $q\asymp n^{\gamma /(1+\gamma )}$ in the algebraically $\beta $%
-mixing case.
\end{proof}

{\
\bibliographystyle{chicago}
\bibliography{regrate}
}

\end{document}